\documentclass[10pt, oneside]{amsart}   	
\usepackage{geometry}                		
\usepackage{graphicx}				
\usepackage{amssymb}
\usepackage{amsmath}
\usepackage{amsfonts}
\usepackage{amsthm}
\usepackage{color}
\usepackage{amscd}
\usepackage[usenames,dvipsnames,svgnames,table]{xcolor}
\usepackage[utf8]{inputenc}
\usepackage[parfill]{parskip}
\usepackage[colorlinks=true, pdfstartview=FitV,linkcolor=ForestGreen,citecolor=ForestGreen, urlcolor=black]{hyperref}
\usepackage{enumerate}
\usepackage{esint}
\usepackage{cite}
\usepackage{tikz}
\usepackage{url} 

\newcommand{\dd}{\, {\rm d}}

\newcommand{\abs}[1]{\left\vert#1\right\vert}
\newcommand{\Abs}[1]{\left\vert#1\right\vert}
\newcommand{\norm}[1]{\left\Vert#1\right\Vert}  

\newcommand{\R}{\ensuremath{{\mathbb R}}}
\newcommand{\N}{\ensuremath{{\mathbb N}}}

\newcommand{\beq}{\begin{equation}}
\newcommand{\eeq}{\end{equation}}
\newcommand{\beqs}{\begin{equation*}}
\newcommand{\eeqs}{\end{equation*}}
\newcommand{\bal}{\begin{equation}\begin{aligned}}
\newcommand{\eal}{\end{aligned}\end{equation}}
\newcommand{\bals}{\begin{equation*}\begin{aligned}}
\newcommand{\eals}{\end{aligned}\end{equation*}}

\newcounter{num} \numberwithin{num}{section}

\newtheorem{theorem}[num]{Theorem}
\newtheorem{proposition}[num]{Proposition}
\newtheorem{lemma}[num]{Lemma}

\theoremstyle{definition}
\newtheorem{definition}[num]{Definition}

\theoremstyle{remark}
\newtheorem{remark}[num]{Remark}

\numberwithin{equation}{section}

\title{Quantitative De Giorgi Methods in Kinetic Theory for Non-Local Operators}
\author{Amélie Loher}
\date{Jan 6, 2024}

\address[Amélie Loher]{DPMMS, University of Cambridge, Wilberforce road, Cambridge CB3 0WA, UK}
\email{ajl221@cam.ac.uk}

\begin{document}

\begin{abstract}
We derive quantitatively the Harnack inequalities for kinetic integro-differential equations. This implies Hölder continuity. Our method is based on trajectories and exploits a term arising due to the non-locality in the energy estimate. This permits to quantitatively prove the intermediate value lemma for the full range of non-locality parameter $s \in (0, 1)$. Our results recover the results from Imbert and Silvestre \cite{IS} for the inhomogeneous Boltzmann equation in the non-cutoff case. The paper is self-contained.
\end{abstract}

\maketitle
\tableofcontents

\section{Introduction}
\subsection{Problem Formulation}

We consider non-local kinetic equations of the form
\beq
	\partial_t f + v \cdot \nabla_x f = \mathcal L f + h, \quad t \in \R, ~ x \in \R^d, ~ v \in \R^d,
\label{eq:1.1}
\eeq
where we assume $h = h(t, x, v)$ is a real scalar field in $L^\infty$ and for a non-negative measurable kernel $K: \R \times \R^d \times \R^d \times \R^d \to [0, +\infty)$ we define
\beq
	\mathcal L f(t, x, v) := \textrm{PV}\int_{\R^d} K(t, x, v, w) \big[f(t, x, w) - f(t, x, v)\big] \dd w.
\label{eq:1.2}
\eeq
The integral is to be understood in the principal value sense, denoted with $PV$. The question we raise is whether solutions to \eqref{eq:1.1}-\eqref{eq:1.2} satisfy Hölder continuity and Harnack inequalities. The Weak Harnack inequality and Hölder continuity have been obtained by Cyril Imbert and Luis Silvestre in \cite{IS}, yet part of their proof relies on an argument by contradiction. Our aim is to develop a constructive method to deduce the results. 

\subsection{Contribution}

Our contribution consists in a quantitative proof of regularity for non-local kinetic equations of type \eqref{eq:1.1}-\eqref{eq:1.2}. Our results are applicable to the non-cutoff Boltzmann equation. We prove Harnack inequalities and Hölder continuity. The technical assumptions that are required for these statements will be discussed below \eqref{eq:coercivity}-\eqref{eq:cancellation2}.
In short they define a suitable notion of ellipticity on the non-local coefficients of \eqref{eq:1.1}. The analogue statement for local equations requires the coefficients to be uniformly elliptic, that is the eigenvalues of the diffusion matrix lie in $[\lambda, \Lambda]$ for given $0 < \lambda \leq \Lambda$.
The parameter $s\in (0, 1)$ arises from the non-locality of \eqref{eq:1.1} and is hidden in the conditions on the kernel. 
\begin{theorem}[Weak Harnack inequality]\label{thm:weakH}
Let $f$ be a non-negative super-solution to \eqref{eq:1.1}-\eqref{eq:1.2} in $[-3, 0] \times B_1 \times B_1$ with a non-negative kernel $K$ satisfying \eqref{eq:coercivity}-\eqref{eq:cancellation2} for $\bar R = 2$. 
Then there is $C$ and $\zeta > 0$ depending on $s, d, \lambda,\Lambda$ such that for $r_0 < \frac{1}{3}$ the Weak Harnack inequality is satisfied:
\beq
	\Bigg(\int_{\tilde Q_{\frac{r_0}{2}}^-} f^\zeta(z) \dd z\Bigg)^{\frac{1}{\zeta}} \leq C \Bigg(\inf_{Q_{\frac{r_0}{2}}} f + \norm{h}_{L^\infty(Q_1)}\Bigg),
\label{eq:weakH}
\eeq
where $\tilde Q_{\frac{r_0}{2}}^-:= Q_{\frac{r_0}{2}}\Big(\big(-\frac{5}{2}r_0^{2s} + \frac{1}{2} (\frac{r_0}{2})^{2s}, 0, 0\big)\Big)$.
\end{theorem}
In particular, this implies Hölder continuity. 
\begin{theorem}[Hölder continuity]
Let $f$ be a non-negative weak solution of \eqref{eq:1.1}-\eqref{eq:1.2} in $ [-3, 0] \times B_1 \times B_1$ with a non-negative kernel $K$ satisfying \eqref{eq:coercivity}-\eqref{eq:cancellation2} for $\bar R = 2$. Assume $f$ is essentially bounded in $(-3,0]\times B_1\times \R^d$. Then $f$ is Hölder continuous in $Q_{\frac{1}{2}}$ with Hölder exponent $\alpha \in (0,1)$ depending on $s, d, \lambda, \Lambda, \norm{h}_{L^\infty(Q_1)}$ with
\bal
	{[f]}_{C^\alpha(Q_{\frac{1}{2}})} &:= \sup_{z_1, z_2 \in Q_{\frac{1}{2}},z_1\neq z_2} \frac{\abs{f(z_1)-f(z_2)}}{\abs{z_1-z_2}^\alpha}\\
	&\leq C \big(1+\norm{h}_{L^\infty(Q_1)}\big)\Big(\norm{f}_{L^\infty((-1, 0]\times B_1 \times \R^d)}+\norm{h}_{L^\infty(Q_1)}\Big),
\label{eq:holder}
\eal
where $C$ depends on $d, s, \lambda$ and $\Lambda$.
\label{thm:holder}
\end{theorem}
Finally, we are able to show a \textit{non-linear} Strong Harnack inequality.
\begin{theorem}[Not-so-Strong Harnack inequality]\label{thm:not-so-strong-H}
Let $f$ be a non-negative solution to \eqref{eq:1.1}-\eqref{eq:1.2} in $[-3, 0] \times B_1 \times B_1$ with a non-negative kernel $K$ satisfying \eqref{eq:coercivity}-\eqref{eq:cancellation2} for $\bar R = 2$. Assume $0 \leq f \leq 1$ in $(-3,0]\times B_1\times \R^d$. 
Then for any $p \in \Big[2, 2+\frac{2s}{d(1+s)}\Big)$ there is $C$ and $\beta \in (0, 1)$ depending on $s, d, \lambda,\Lambda, p$ such that for $r_0 < \frac{1}{3}$
\beq
	\sup_{\tilde Q_{\frac{r_0}{4}}^-} f \leq C \Bigg(\inf_{Q_{\frac{r_0}{4}}} f + \norm{h}_{L^\infty(Q_1)}\Bigg)^{\beta},
\label{eq:strongH}
\eeq
where $\tilde Q_{\frac{r_0}{4}}^-:= Q_{\frac{r_0}{4}}\big((-\frac{5}{2}r_0^{2s} + \frac{1}{2} (\frac{r_0}{2})^{2s}, 0, 0)\big)$. 
\end{theorem}
By quantifying the proof of the Harnack inequalities, we implemented the whole structure of the De Giorgi argument in the kinetic setting. In particular, we established the Second De Giorgi Lemma, also known as Intermediate Value Lemma, for kinetic integro-differential equations with a non-locality exponent $s \in (0, 1)$. To the best of our knowledge this is the first appearance of a non-local Intermediate Value Lemma in the kinetic setting for the full range $s \in (0, 1)$. The barrier functions $\varphi_0$ and $\varphi_2$ are defined in detail in Subsection \ref{subsec:ivl}, where we prove Theorem \ref{thm:IVL}. For now, we content ourselves with an illustration of the barriers in Figure \ref{fig:barriers}, so that we can figuratively make sense of the following statement.
\begin{theorem}[Intermediate Value Lemma]\label{thm:IVL}
Let $f$ be a sub-solution to \eqref{eq:1.1}-\eqref{eq:1.2} in $[-3, 0] \times B_{1} \times B_{1}$ under assumptions \eqref{eq:coercivity}-\eqref{eq:cancellation2} for $\bar R = 2$ with $0 \leq f \leq 1$ in $(-3, 0]\times B_{1}\times \R^d$. Given $\delta_1, \delta_2 \in (0, 1)$ there is $\mu \lesssim (\delta_1\delta_2)^{6d + 16} $ and $\nu \gtrsim (\delta_1\delta_2)^{18d + 46}$ such that if $\norm{h}_{L^\infty(Q_{3r_0})} \leq C\mu^2$ for $r_0 <  \frac{1}{3}$ and
\beq
	\Abs{\{f\leq 0\}\cap Q^-_{r_0}} \geq \delta_1\abs{Q^-_{r_0}} \quad \textrm{ and }  \quad \Abs{\{f > 1 -\mu^2\}\cap Q_{r_0}} \geq \delta_2\abs{Q_{r_0}},
\label{eq:IVLc}
\eeq
then $f$ satisfies
\beq
	\Abs{\{\varphi_0 < f< \varphi_2\}\cap (-3, 0] \times B_{(\frac{1}{2})^{1+2s}} \times B_{\frac{1}{2}}} \geq \nu\Abs{Q_{\frac{1}{2}}},
\label{eq:IVL}
\eeq
with barriers $\varphi_0, \varphi_2$ defined in \eqref{eq:barriers} and illustrated in Figure \ref{fig:barriers}. Recall $$Q^-_{r_0} = \big(-3r_0^{2s}, -2r_0^{2s}\big] \times B_{r_0^{1+2s}} \times B_{r_0}.$$
\end{theorem}
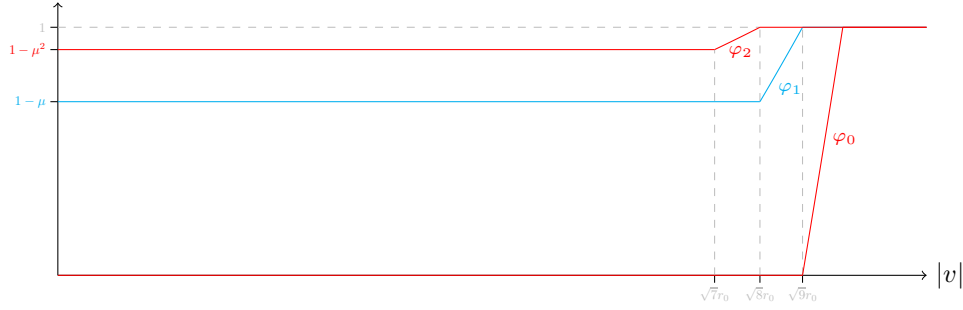
\begin{figure}{
\centering\begin{center}
\begin{tikzpicture}[scale=3.3]
  \draw[->] (0, 0) -- (3.5, 0) node[right] {$\abs{v}$};
  \draw[->] (0, 0) -- (0, 1.1);
\draw (0,1) -- (-0.03, 1) node[anchor=east, scale =0.5, lightgray] {$1$};
\draw (0,0.91) -- (-0.03, 0.91) node[anchor=east, scale =0.5, red] {$1-\mu^2$};
\draw (0,0.7) -- (-0.03, 0.7) node[anchor=east, scale =0.5, cyan] {$1-\mu$};
\draw (2.64575131106, 0) -- (2.64575131106, -0.03) node[anchor=north, scale =0.5, lightgray] {$\sqrt{7} r_0$};
\draw (2.82842712475, 0) -- (2.82842712475, -0.03) node[anchor=north, scale =0.5, lightgray] {$\sqrt{8} r_0$};
\draw (3, 0) -- (3, -0.03) node[anchor=north, scale =0.5, lightgray] {$\sqrt{9} r_0$};
\draw[dashed, lightgray] (0, 1) -- (2.82842712475,1);
\draw[dashed, lightgray]  (2.82842712475, 0)  -- (2.82842712475, 1);
\draw[dashed, lightgray]  (3, 0)-- (3, 1);
\draw[dashed, lightgray] (2.64575131106, 0) -- (2.64575131106, 0.91);
   \draw[scale = 1, cyan] (2.95, 0.75) node {\scriptsize{$\varphi_1$}};
  \draw[scale=1, domain=0:2.82842712475, smooth, variable=\v, cyan] plot ({\v}, {1-0.3});
  \draw[scale=1, domain=2.82842712475:3, smooth, variable=\v, cyan] plot ({\v}, {1+ 0.3*((\v *\v)-9)});
  \draw[scale=1, domain=3:3.5, smooth, variable=\v, cyan] plot ({\v}, {1});
  \draw[scale = 1, red] (3.17, 0.55) node {\scriptsize{$\varphi_0$}};
  \draw[scale=1, domain=0:3, smooth, variable=\v, red]  plot ({\v}, {0}); 
  \draw[scale=1, domain= 3:3.16227766017, smooth, variable=\v, red] plot ({\v}, {1+ (\v *\v)-10}) ;
;
  \draw[scale=1, domain= 3.16227766017:3.5, smooth, variable=\v, red] plot ({\v}, {1});
   \draw[scale = 1, red] (2.75, 0.9) node {\scriptsize{$\varphi_2$}};
  \draw[scale=1, domain=0:2.64575131106, smooth, variable=\v, red] plot ({\v}, {1-0.09}); 
  \draw[scale=1, domain=2.64575131106:2.82842712475, smooth, variable=\v, red] plot ({\v}, {1+ 0.09*((\v *\v)-8)});
  \draw[scale=1, domain=2.82842712475:3.5, smooth, variable=\v, red] plot ({\v}, {1});
\end{tikzpicture}
\end{center}
}
\caption{The barriers $\varphi_0, \varphi_1, \varphi_2$ in the Intermediate Value Lemma, Theorem \ref{thm:IVL}, for $r_0 < \frac{1}{3}$, $\mu < 1$ and fixed $x \in B_{(3r_0)^{1+2s}}$.}\label{fig:barriers}
\end{figure}
We view Theorem \ref{thm:IVL} as our main contribution. In Figure \ref{fig:cylinders} we illustrate the domains of our main results.
\begin{figure}
\centering
\begin{tikzpicture}[scale =0.95]
\draw[dashed] (-6, 0) -- (6, 0) node[anchor=north, scale=1] {$Q_1$};
  \draw [blue] (-3.5,0) -- (3.5,0); 
   \draw[blue] (-3.5,-7) -- (3.5,-7);
  \draw [blue](-3.5,0) -- (-3.5,-7);
  \draw [blue](3.5,0) -- (3.5,-7) node[anchor=north, scale=1] {$Q_{\frac{1}{2}}$};
   \draw [violet](-1,0) -- (1, 0);
   \draw [violet](-1,-2) -- (1, -2);
   \draw [violet](-1,0) -- (-1, -2);
   \draw [violet](1,0) -- (1, -2)  node[anchor=north, scale=1] {$Q_{r_0}$};
    \draw[violet] (-1,-4) -- (1, -4);
   \draw [violet](-1,-6) -- (1, -6);
   \draw [violet](-1,-4) -- (-1, -6);
   \draw [violet](1,-4) -- (1, -6)  node[anchor=north, scale=1] {$Q^-_{r_0}$};
     \draw [orange](-0.5,0) -- (0.5, 0);
   \draw[orange] (-0.5,-1) -- (0.5, -1);
   \draw [orange](-0.5,0) -- (-0.5, -1);
   \draw [orange](0.5,0) -- (0.5, -1)  node[anchor=north, scale=0.65] {$Q_{\frac{r_0}{2}}$};
    \draw[orange] {(-0.5,-4.5) -- (0.5, -4.5)};
   \draw [orange](-0.5,-5.5) -- (0.5, -5.5);
   \draw [orange](-0.5,-4.5) -- (-0.5, -5.5);
   \draw [orange](0.5,-4.5) -- (0.5, -5.5) node[anchor=north, scale=0.65] {$\tilde Q^-_{\frac{r_0}{2}}$};
     \draw [red](-0.25,0) -- (0.25, 0);
   \draw [red](-0.25,-0.5) -- (0.25, -0.5);
   \draw[red] (-0.25,0) -- (-0.25, -0.5);
   \draw[red] (0.25,0) -- (0.25, -0.5) node[anchor=north, scale=0.65] {$Q_{\frac{r_0}{4}}$};
     \draw[red] (-0.25,-4.5) -- (0.25, -4.5);
  \draw[red] (-0.25,-5) -- (0.25, -5);
   \draw[red] (-0.25,-4.5) -- (-0.25, -5);
   \draw[red] (0.25,-4.5) -- (0.25, -5)node[anchor=north, scale=0.65] {$\tilde Q^-_{\frac{r_0}{4}}$};
\end{tikzpicture}
\caption{\small{For a solution in $Q_1$, we prove Hölder continuity \ref{thm:holder} in $Q_{\frac{1}{2}}$. The Intermediate Value Lemma, Theorem \ref{thm:IVL}, relates $Q_{r_0}$ to $Q^-_{r_0}$ in the past. The Weak Harnack inequality, Theorem \ref{thm:weakH}, relates $Q_{\frac{r_0}{2}}$ with $\tilde Q^-_{\frac{r_0}{2}}$ and the Not-so-Strong Harnack inequality, Theorem \ref{thm:not-so-strong-H}, relates $Q_{\frac{r_0}{4}}$ with $\tilde Q^-_{\frac{r_0}{4}}$ in the past. Note that all cylinders depend on $s \in (0, 1)$.}}\label{fig:cylinders}
\end{figure}
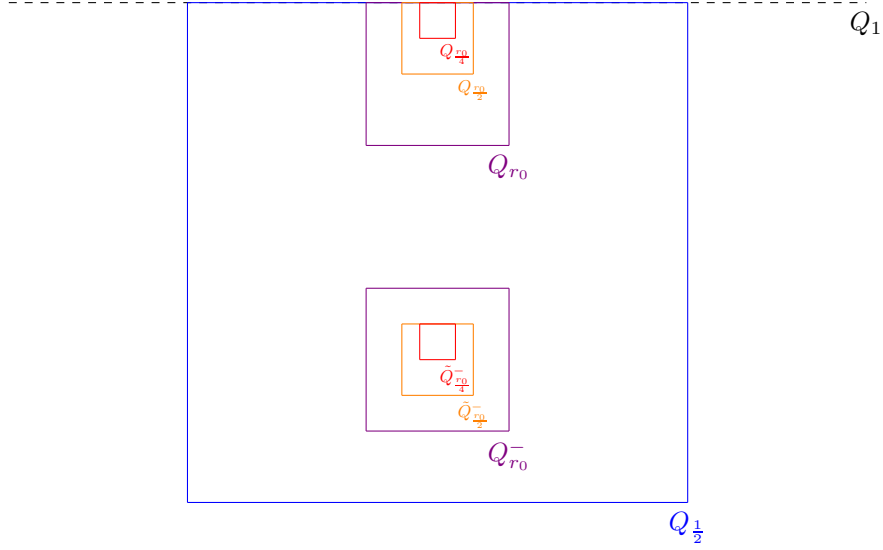

\subsection{Ellipticity Class of $K$}
We introduce the ellipticity class of the kernel $K$ in \eqref{eq:1.2} as follows.
Even though the kernel depends in general on time and space we will often omit to write out this dependency for the sake of brevity.
We let $s \in (0, 1)$, $\bar R >0$ and $0 < \lambda < \Lambda$. Let $\varphi : \R^d \to \R$ be supported in $B_{\frac{\bar R}{2}}$. Then we assume that $K: [-\bar R^{2s}, 0]\times B_{{\bar R}^{1+2s}} \times \R^d  \times \R^d \to [0, +\infty)$ is coercive in the sense that
\beq
	\lambda \int_{B_{\bar R/2}}\int_{B_{\bar R/2}}  \frac{\abs{\varphi(v)-\varphi(w)}^2}{\abs{v-w}^{d+2s}} \dd v\dd w \leq \int_{B_{\bar R}}\int_{B_{\bar R}} \big(\varphi(v)-\varphi(w)\big)^2K(v, w) \dd w\dd v,
\label{eq:coercivity}
\eeq
and
\beq
	\lambda \int_{B_{\bar R/2}}\int_{B_{\bar R/2}} \frac{\abs{\varphi(v)-\varphi(w)}}{\abs{v-w}^{\frac{d+2s}{2}}} \dd v\dd w \leq \int_{B_{\bar R}}\int_{B_{\bar R}} \abs{\varphi(v)-\varphi(w)}K^{\frac{1}{2}}(v, w) \dd w\dd v.
\label{eq:coercivity-sqrt}
\eeq
Moreover, we require the following upper bound for $r > 0$
\beq
	\forall v \in B_{\bar R} \quad \int_{B_r(v)}K(v,w)\abs{v -w}^2\dd w \leq \Lambda r^{2-2s}.
\label{eq:upperbound-2}
\eeq
Finally, instead of the usual symmetry assumption $K(v, w) = K(w, v)$, which corresponds to the divergence form structure of local equations, we assume the following cancellation
\beq
	\forall v \in \R^d \quad \Bigg\vert \textrm{PV} \int_{\R^d} \big(K(v, w) - K(w, v)\big)\dd w\Bigg\vert \leq \Lambda,
\label{eq:cancellation1}
\eeq
and if $s \geq \frac{1}{2}$ we assume for all $r>0$ 
\beq
	\forall v \in\R^d \quad \Bigg\lvert \textrm{PV} \int_{B_r(v)} (v -w)\big(K(v, w) - K(w, v)\big)\dd w\Bigg\rvert \leq \Lambda \big(1+r^{1-2s}\big).
\label{eq:cancellation2}
\eeq
We remark that the coercivity \eqref{eq:coercivity} together with the cancellation assumption \eqref{eq:cancellation1} imply 
\beq
	\lambda \int_{\R^d}\int_{\R^d} \frac{\abs{\varphi(v)-\varphi(w)}^2}{\abs{v-w}^{d+2s}} \dd v\dd w \leq \int_{\R^d}\int_{\R^d} \big(\varphi(v)-\varphi(w)\big)K(v, w) \varphi(v) \dd w\dd v + \Lambda \norm{\varphi}_{L^2(\R^d)}^2,
\label{eq:coercivity-full}
\eeq
where $\varphi$ is supported in $B_{\frac{\bar R}{2}}$.
Moreover the upper bound \eqref{eq:upperbound-2} is equivalent to 
\bal
	\forall v \in B_{\bar R}, \forall r > 0 \quad \int_{B_{2r}(v)\setminus B_r(v)}K(v,w)\dd w \leq \Lambda r^{-2s}, \\
	\forall w \in B_{\bar R}, \forall r > 0 \quad \int_{B_{\bar R}\cap B_{2r}(w)\setminus B_r(w)}K(v,w)\dd v \leq \Lambda r^{-2s}.
\label{eq:upperbound}
\eal
We also note that \eqref{eq:cancellation2} holds for $s \in \Big(0, \frac{1}{2}\Big)$ as a consequence of \eqref{eq:cancellation1} and \eqref{eq:upperbound-2}. These assumptions are inspired from the work of Imbert and Silvestre \cite{IS}. 
In particular it is clear that the fractional Laplacian $(-\Delta_v)^s$ satisfies \eqref{eq:coercivity}-\eqref{eq:cancellation2}. The reason why we do not want to restrict the kernel to the fractional Laplacian is so that we are able to apply our result to the non-cutoff Boltzmann equation. 
\subsection{Non-cutoff Boltzmann equation}
The Boltzmann equation is given by
\beq
	\partial_t f + v\cdot \nabla_x f = Q(f, f),
\label{eq:boltzmann}
\eeq
where the collision operator $Q$ has the form
\beqs
	Q(f, g) := \int_{\R^d}\int_{\mathbb{S}^{d-1}} \big[f(w_*)g(w) - f(v_*)g(v)\big] B(\abs{v - v_*}, \cos \theta)\dd \sigma \dd v_*, 
\eeqs
with 
\beqs
	w = \frac{v + v_*}{2} +\frac{\abs{v - v_*}}{2}\sigma, \quad w_* = \frac{v + v_*}{2} -\frac{\abs{v - v_*}}{2}\sigma,
\eeqs	
and $\cos \theta$ is defined as 
\beqs
	\cos \theta := \frac{v - v_*}{\abs{v-v_*}} \cdot\sigma, \quad \sin(\theta/2) := \frac{w -v}{\abs{w-v}}\cdot \sigma.
\eeqs
The cross-section $B$ is given by 
\beqs
	B(r, \cos\theta) = r^\gamma b(\cos\theta), \quad b(\cos\theta) \approx \abs{\sin(\theta/2)}^{-(d-1)-2s},
\eeqs
with $\gamma \in (-d, 1]$ and $s \in (0, 1)$. A solution to \eqref{eq:boltzmann} describes the density of particles at time $t$ and location $x$ with velocity $v$. The collision operator can be rephrased as
\beqs
	Q(f, g) = \int \big[g(w) - g(v)\big] K_f(v, w)\dd w + \rm{lower ~ order ~ terms},
\eeqs
where the kernel $K_f$ depends implicitly on the solution $f$ and is determined by
\beq
	\int_{\R^d}\int_{\mathbb{S}^{d-1}} f(w_*)\big[g(w) - g(v)\big] B(\abs{v - v_*}, \cos \theta)\dd \sigma \dd v_* = \int_{\R^d} \big[g(w) - g(v)\big] K_f(v, w)\dd w.
\label{eq:boltzmann_kernel}
\eeq
We define the following macroscopic quantities associated to \eqref{eq:boltzmann} 
\bal
	&M(t, x) := \int f(t, x, v) \dd v,\\
	&E(t, x) := \int f(t, x, v) \abs{v}^2 \dd v,\\
	&H(t, x) := \int f(t, x, v) \ln f(t, x, v) \dd v.
\label{eq:hydros}
\eal
They describe the mass density, the energy and the entropy respectively. Under the physically meaningful assumption that $M$ is uniformly bounded from above and below, and that $E$ and $H$ are both uniformly bounded above, it can be checked that $K_f$ satisfies \eqref{eq:coercivity}-\eqref{eq:cancellation2} with $\lambda$ depending on the mass, energy and entropy of $f$ and $\Lambda$ depending on the mass, energy and $\norm{f}_{L^\infty}$. We refer the interested reader to Section 3  and Lemma A.6 of \cite{IS}. The fact that the Boltzmann kernel satisfies \eqref{eq:coercivity-sqrt} can be demonstrated in the same way as \cite[Lemma A.6]{IS} and is stated in \ref{appendix} below. Therefore Theorem \ref{thm:holder} implies
\begin{theorem}[Imbert-Silvestre]
Let $s \in (0, 1)$ and $\gamma \in (-d, 1]$ be such that $\gamma + 2s \leq 2$. Let $f$ be a non-negative solution to the Boltzmann equation \eqref{eq:boltzmann} for $(t, x, v) \in (-3, 0]\times B_1\times B_1$. Assume that $f$ is essentially bounded in $(-3, 0]\times B_1\times \R^d$ and that there exists $M_0, M_1, E_1$ such that $M_0 \leq M(t, x) \leq M_1, ~ E(t, x) \leq E_1$ for all $(t, x) \in (-3, 0]\times B_1$. Then $f$ is Hölder continuous in $Q_{1/2} := \big(-(1/2)^{2s}, 0\big]\times B_{(1/2)^{1+2s}} \times B_{1/2}$ and for $\alpha \in (0, 1)$ there holds
\beqs
	[f]_{C^\alpha(Q_{1/2})} \leq C,
\eeqs
where $C$ depends on $s, d, \norm{f}_{L^\infty}, M_0, M_1, E_1$. 
\label{thm:boltzmann}
\end{theorem}
\begin{remark}
The entropy $H$ satisfies a uniform bound from above in terms of $M_1, E_1$ and $\norm{f}_{L^\infty}$. Thus we do not require the additional assumption $H \leq H_1$ in Theorem \ref{thm:boltzmann}.
\end{remark}

\subsection{Invariant transformations}
Let $f$ solve \eqref{eq:1.1}. Then for $r \in [0, 1]$ the scaled function $f_r(t, x, v) = f(r^{2s}t, r^{1+2s}x, rv)$ satisfies
\beqs
	\partial_t f_r + v\cdot \nabla_x f_r = \mathcal L_r f_r + h_r,
\eeqs
where $\mathcal L_r$ is the non-local operator associated to the scaled kernel
\beqs
	K_r(t, x, v, w) = r^{d+2s} K(r^{2s}t, r^{1+2s}x, r v, rw), 
\eeqs
and the source is scaled as
\beqs
	h_r(t, x, v) = r^{2s} h(r^{2s}t, r^{1+2s}x, r v).
\eeqs
For $r \in [0, 1]$ the scaled kernel $K_r$ satisfies \eqref{eq:coercivity}-\eqref{eq:cancellation2} in the larger radius $\frac{\bar R}{r}$ instead of $\bar R$. Moreover, $\norm{h_r}_{L^\infty(Q_1)} \leq r^{2s}\norm{h}_{L^\infty(Q_1)} \leq \norm{h}_{L^\infty(Q_1)}$.

Furthermore, the equation is invariant under Galilean transformations $z \to z_0 \circ z = (t_0 + t, x_0+x + tv_0, v_0 + v)$ with $z_0 = (t_0, x_0, v_0)\in \R^{1+2d}$. If $f$ is a solution of \eqref{eq:1.1} then its Galilean transformation $f_{z_0}(z) = f(z_0 \circ z)$ solves
\beqs
	\partial_t f_{z_0} + v\cdot \nabla_x f_{z_0} = \mathcal L_{z_0} f_{z_0} + h_{z_0},
\eeqs
where $\mathcal L_{z_0}$ is the non-local operator associated to the translated kernel
\beqs
	K_{z_0}(t, x, v, w) = K(z_0 \circ z, v_0 + w), 
\eeqs
and with source
\beqs
	h_{z_0}(t, x, v) = h(z_0 \circ z).
\eeqs
Again the modified kernel $K_{z_0}$ satisfies \eqref{eq:coercivity}-\eqref{eq:cancellation2} provided that $K$ does. 

In view of these invariances we introduce kinetic cylinders 
\beqs
	Q_r(z_0) := \Big\{(t, x, v) : -r^{2s} \leq t - t_0 \leq 0, \abs{v - v_0} < r, \Abs{x - x_0 -(t-t_0)v_0} < r^{1+2s}\Big\},
\eeqs
for $r > 0$ and $z_0 = (t_0, x_0, v_0) \in \R^{1+2d}$. For later reference, we also introduce the cylinder shifted to the past
\beqs
	Q^-_r(z_0) := Q_r\big(z_0 - (2r^{2s}, 2r^{2s}v_0, 0)\big),
\eeqs
so that in particular for $z_0 = 0$
\beqs
	Q^-_r := Q_r\big(-2r^{2s}, 0, 0\big) = \big(-3r^{2s}, -2r^{2s}\big] \times B_{r^{1+2s}} \times B_r.
\eeqs
Similarly the cylinder shifted to the future is denoted as
\beqs
	Q^+_r(z_0) := Q_r\big(z_0 + (2r^{2s}, 2r^{2s}v_0, 0)\big).
\eeqs
Finally, if we merely use the space and velocity domain of a cylinder, we denote it as $Q^t_r := B_{r^{1+2s}} \times B_r$ where $t \in \mathcal I$ for some specified interval $\mathcal I \subset \R$.

\subsection{Further remarks}
Our equation involves a transport term, which transfers some regularity of the velocity variable to the space variable. It also involves a non-local diffusion in the $v$ variable. 
Our motivation to study the regularity of this type of equation is linked to the question of well-posedness for smooth classical solutions of the inhomogeneous Boltzmann equation without cut-off. There are linear kinetic equations whose solutions in the hydrodynamic limit are described by a fractional diffusion \cite{mellet, MMM, BM}. Indeed diffusion limits of the linear Boltzmann equation with a heavy-tailed distribution of infinite variance as an equilibrium distribution give rise to a fractional diffusion equation \cite{MMM}. Such heavy-tailed distribution functions arise in astrophysical plasmas \cite{mendisrosen} or also in granular gases through dissipative collision mechanisms \cite{ErnstBrito}. However, the only source of fractional diffusion at the kinetic level stems from long range interactions of the Boltzmann collision kernel. 

In the limit case $s \to 1$, equation \eqref{eq:1.1} models the local kinetic Fokker-Planck equation, whose study is motivated by applications to the Landau equation \cite{GIMV}. For the local case, there is a non-constructive method discussed in \cite{GIMV}. A constructive proof first appeared in a series of works \cite{whangzhang1, whangzhang2, whangzhang3} for ultraparabolic equations that has further been developed in \cite{JGCI} to local kinetic Fokker-Planck type equations. The construction is based on a Poincaré-type inequality and Kruzhkov's method \cite{kruzhkov}. A novel constructive approach for local kinetic Fokker-Planck type equations has been devised by Guerand and Mouhot in \cite{JGCM}. Their method relies on trajectories. For general $s \in (0, 1)$, Cyril Imbert and Luis Silvestre (together with Clément Mouhot in \cite{IMSfrench, IMS}) made important contributions in a series of papers \cite{IS59, IS, ISschauder, IMSfrench, IMS} that culminated in the final work \cite{ISglobal}. In \cite{IS59}, Silvestre proves for a certain range of $s$ that any solution $f$ is a priori essentially bounded provided that the mass, energy and entropy defined above \eqref{eq:hydros} satisfy uniform bounds from above and that the mass is bounded below. 
With an additional non-degeneracy assumption in \cite{IS}, they obtain the Weak Harnack inequality with a quantitative argument when $s \in (0, \frac{1}{2})$. Note that their method, which uses barrier functions, can be extended to $s \in (0, 1)$ under the additional symmetry assumption $K(v, v+w) = K(v, v-w)$, cf. section 7 in \cite{IS}. When $s \in (\frac{1}{2}, 1)$ they prove an Intermediate Value Lemma using an argument by contradiction. In \cite{ISschauder}, Imbert and Silvestre derive Schauder estimates for kinetic equations, which can then be bootstrapped for the non-cutoff Boltzmann equation to obtain smooth solutions \cite{ISglobal}. 
Let us also mention a work by Stokols \cite{Stokols}, which combines the method of \cite{GIMV} with fractional estimates from \cite{CCV} to obtain a non-constructive proof of Hölder continuity in the non-local case. The assumptions on the kernel in \cite{Stokols} are stronger than ours: the kernel is assumed to satisfy pointwise upper and lower bounds in $(v, w)$, as well as the symmetry conditions $K(v, w) = K(w, v)$ and $K(v, v+w) = K(v, v-w)$. We need more general assumptions to be able to apply our results to the Boltzmann equation.

The paper presents a constructive proof of Hölder continuity and Harnack inequalities using a De Giorgi-type argument. De Giorgi methods were originally established for non-linear elliptic equations by De Giorgi \cite{DG}. Moser then showed how to deduce a Harnack inequality \cite{Mos1, Mos2}. The Harnack inequalities are local regularity results. In particular, the Weak Harnack inequality implies Hölder continuity \cite{gilbargtrudinger}. By using an argument based on trajectories, we derive the Intermediate Value Lemma, which in turn implies Hölder continuity for the non-cutoff Boltzmann equation. In order to cover the full range of non-locality parameter $s \in (0, 1)$ we have to carefully exploit the cross-term arising in the energy estimate similar to Section 4 in \cite{CCV}. This is a purely non-local effect. To the best of our knowledge, this is the first Intermediate Value Lemma for non-local kinetic equation covering the full range $s \in (0, 1)$. 

The Not-so-Strong Harnack inequality is obtained by combining the Weak Harnack inequality with De Giorgi's first lemma. The Strong Harnack inequality for non-local equations has been discussed in \cite{caff_silv1, caff_silv2, CKW} for the elliptic case and in \cite{bass_levin, chen_kumagai, chen_kumagai_wang, kass11} for the parabolic case. Most of these works consider non-local operators with symmetric kernels that verify a pointwise upper and lower bound. In general the Strong Harnack inequality fails for non-local equations, if the function is not assumed to be non-negative \cite{kass_counterex}. Instead one gets an additional term from the tail of the negative part of the function on the right hand side \cite{kass11}. The parabolic Harnack inequality without tail term has so far only been proved using probabilistic tools such as heat kernel estimates \cite{chen_kumagai, chen_kumagai_wang}. There are no analogous results in the hypo-elliptic case. In \cite{kasswei}, they prove the Harnack inequality with a tail term on the right hand side for non-local parabolic equations. Under certain assumptions \cite[Conditions (a) and (b) on p. 47]{kasswei} the non-local tail term can be bounded by a local quantity, which could then be absorbed on the left hand side of the inequality via a standard covering argument, yielding the Strong Harnack inequality. However, these additional assumptions do not naturally follow. But from the Weak Harnack inequality it is a straight-forward consequence to prove a weaker non-linear Strong Harnack inequality for essentially bounded solutions in the sense of Theorem \ref{thm:not-so-strong-H}.

\subsection{Structure of the article}
The core of the proof relies on the following method:
\bals
	\left.\begin{array}{r@{\mskip\thickmuskip}l}
	\textrm{Energy estimate} &\xrightarrow[{0 \leq f \leq 1}]{} \textrm{First De Giorgi Lemma}\\
	\textrm{Weak Poincaré inequality} &\xrightarrow[{0 \leq f \leq 1}]{} \textrm{Second De Giorgi Lemma}\\
	\end{array} \right\}
	\quad \xrightarrow \quad
 	\textrm{Weak Harnack}.
\eals
Note that in the final statement we require our function to be essentially bounded for $v$ in $\R^d$. The reason behind this is to be able to take care of the tails of the non-local operator. The energy estimate and the weak Poincaré inequality hold for any sub-solution of the equation, whereas for the First and Second De Giorgi Lemma we need to assume that $f \leq 1$ almost everywhere.

In Section \ref{sec:2} we fix the notation, discuss definitions of weak solutions for \eqref{eq:1.1} and state a result of boundedness in $H^s \times H^s$ of the bilinear form \eqref{eq}.

For De Giorgi's First Lemma we prove an energy estimate in Section \ref{subsec:intest} and gain integrability in Lebesgue spaces using estimates on the fundamental solution of the fractional Kolmogorov equation. We deduce De Giorgi's First Lemma \ref{lem:6.6} by a classical \textit{De Giorgi iteration} in Section \ref{subsec:firstDG}.

De Giorgi's Second Lemma, Theorem \ref{thm:IVL}, is proved in Section \ref{subsec:secondDG} by first proving a weak Poincaré inequality in $L^1$ using the trajectorial approach of \cite{JGCM}. The difficulties that arise for small values of $s$ are dealt with by introducing suitable level functions inspired from \cite{CCV}.

We deduce Hölder continuity, Theorem \ref{thm:holder}, and the Harnack inequalities, Theorem \ref{thm:weakH} and Theorem \ref{thm:not-so-strong-H}, in Section \ref{sec:6}. Hölder continuity follows by standard arguments either from the Weak Harnack inequality or by the measure-to-pointwise estimate. To prove the Harnack inequalities, we use a covering argument, which we adapt from \cite{JGCM}. The geometric construction for the covering had to account for the specific scaling of the fractional diffusion. Once we have obtained the Weak Harnack inequality, we combine it with the First De Giorgi Lemma to deduce the non-linear Strong Harnack inequality. 

\section{Weak Formulation}
\label{sec:2}
\subsection{Notation}
A constant is called universal, if it only depends on the dimension, the fractional exponent $s$ and $\lambda, \Lambda$ in \eqref{eq:coercivity}-\eqref{eq:cancellation2}. We use the notation $a \lesssim b$ if there exists a universal constant $C$ such that $a \leq Cb$, and $a \approx b$ if $a \lesssim b$ and $b \lesssim a$. Moreover, we say that $a \lesssim_d b$ if $a \leq C b$ where $C = C(d)$. For a real number $a$ we denote $a_+ = \sup(a, 0) \geq0$ and $a_- = \inf(a, 0) \leq 0$. 

For a given domain $\Omega \subset \R^d$ we denote with $\dot H^s(\Omega)$ the space that is equipped with the norm
\beqs
	\norm{f}^2_{\dot H^s(\Omega)} := \int_\Omega\int_\Omega \frac{\abs{f(v)-f(w)}^2}{\abs{v-w}^{d+2s}}\dd w \dd v.
\eeqs
The space $H^s(\Omega)$ is correspondingly equipped with the norm
\beqs
	\norm{f}^2_{H^s(\Omega)} := \norm{f}^2_{\dot H^s(\Omega)} + \norm{f}^2_{L^2(\Omega)}.
\eeqs
The space $H_0^s(\Omega)$ is defined as the closure of the space of smooth functions in $\R^d$ with compact support contained in $\Omega$, where the closure is taken with respect to the $H^s(\Omega)$ norm. We denote the dual of $H_0^s(\Omega)$ with $H^{-s}(\Omega)$.

Finally the transport operator is denoted by
\beqs
	\mathcal T := \partial_t + v \cdot \nabla_x. 
\eeqs

\subsection{Bilinear Form}
We introduce the bilinear form associated to the operator $\mathcal L$ in \eqref{eq:1.2} for $(t, x) \in \R^{1+d}$
\bal
	\mathcal E\big(\varphi, g\big)(t, x) &:= - \int \big(\mathcal L \varphi\big)(t, x, v)g(t, x, v) \dd v \\
	&= \lim_{\varepsilon \to 0} \int \int_{\abs{v - w} > \varepsilon}  K(t, x, v, w) \big[\varphi(t, x, v) - \varphi(t, x, w)\big]g(t, x, v)\dd w \dd v.
\label{eq}
\eal
In the sequel, we will abuse notation by ignoring the limit as $\varepsilon \to 0$ and understanding some integrals in the principal value sense. Moreover, even though $K$ can depend on time and space in general, we will subsequently often omit to write out this dependency explicitly. The following theorem has first been proved in \cite[Corollary 5.2]{IS}. We give a different proof for the anti-symmetric part of the operator.
\begin{theorem}
Let $\bar R > 0$, $\Lambda > 0$ and  $K: \R^d \times \R^d \to \R^d$ be an non-negative kernel satisfying \eqref{eq:upperbound-2}, \eqref{eq:cancellation1} and \eqref{eq:cancellation2}. Then for $f, g\in H^s(\R^d)$ supported in $B_{\frac{\bar R}{2}}$ there holds
\beqs
	\mathcal E(f, g) \leq C \norm{f}_{H^s(\R^d)}\norm{g}_{H^s(\R^d)},
\eeqs
where $C$ depends on $s, d$ and $\Lambda$. 
\label{thm:4.1}
\end{theorem}
\begin{proof}
By density of $C^\infty$ in Sobolev spaces it suffices to consider smooth $f, g$. We divide the proof into the symmetric and anti-symmetric part of the kernel: we write
\bals
	\mathcal E(f, g) &=\frac{1}{2}  \underbrace{\int_{B_{\bar R}}  \int_{\R^d} \big(f(v) - f(w)\big)\big(g(v) - g(w)\big)K(v, w)\dd w\dd v}_{=: I_1(f, g)} \\
	&\quad +\frac{1}{2}\underbrace{\int_{B_{\bar R}} \int_{\R^d} \big(f(v) - f(w)\big)\big(g(v) + g(w)\big)K(v, w)\dd w\dd v}_{=: I_2(f, g)},
\eals
where the integrals are understood in a principal value sense.

For $I_1$ we use Young's inequality and we dyadically decompose the domain of integration which gives for $h = f$ or $h  = g$
\beq
	 I_1(h,  h) =\sum_{k=-\infty}^{+\infty} P(2^k),
\label{eq:4.5is}
\eeq
where for $r > 0$
\beqs
	P(r) := \int_{\Sigma_r} \Abs{h(v)- h(w)}^2K(v, w)\dd w\dd v,
\eeqs
with $\Sigma_r = \big\{(v, w) \in B_{\bar R}\times \R^d : r \leq \abs{v-w}<2r\big\}$.
Let $m = \frac{v+w}{2}$ and consider an intermediate point $u \in B_{\frac{r}{4}}(m)$. We use $\abs{h(v)- h(w)}^2\leq 2\abs{h(v)-h(u)}^2 + 2\abs{h(u)-h(w)}^2$ to bound
\bals
	P(r) &\lesssim \frac{1}{r^d}\int_{\Sigma_r}\int_{B_{\frac{r}{4}}(m)}K(v, w)\Big(\abs{h(v)-h(u)}^2 + \abs{h(u)-h(w)}^2\Big)\dd u\dd w\dd v\\
	&\leq \frac{1}{r^d}\int_{\tilde{\Sigma}_r}\abs{h(v)-h(u)}^2\Bigg(\int_{\Omega_{v, u}} K(v, w)\dd w\Bigg)\dd u\dd v \\
	&\qquad+ \frac{1}{r^d}\int_{\tilde{\Sigma}_r}\abs{h(u)-h(w)}^2\Bigg(\int_{\Omega_{w, u}} K(v, w)\dd v\Bigg)\dd u\dd w,
\eals
where we used Fubini's theorem and write $\tilde{\Sigma}_r := \big\{(v, w)\in B_{\bar R}\times\R^d: \frac{r}{4}\leq \abs{v-w}<\frac{5r}{4}\big\}$ and $\Omega_{v,w}$ for the set containing the $u$ corresponding to any pair $(v, w)$. We note that $\Omega_{v,w} \subset B_{2r}(v)\setminus B_r(v)$. Using \eqref{eq:upperbound} we get
\bals
	P(r) &\lesssim \frac{\Lambda}{r^{d+2s}}\Bigg(\int_{\tilde{\Sigma}_r}\abs{h(v)-h(u)}^2\dd u\dd v+\int_{\tilde{\Sigma}_r}\abs{h(u)-h(w)}^2\dd u\dd w\Bigg)\\
	&\leq \Lambda\Bigg(\int_{\tilde{\Sigma}_r}\frac{\abs{h(v)-h(u)}^2}{\abs{v-u}^{d+2s}}\dd u\dd v+ \int_{\tilde{\Sigma}_r}\frac{\abs{h(u)-h(w)}^2}{\abs{u-w}^{d+2s}}\dd u\dd w\Bigg)\\
	&\lesssim \Lambda \norm{h}^2_{H^s}.
\eals
We can apply this final estimate to each term in the sum \eqref{eq:4.5is} and use Cauchy-Schwarz inequality to obtain
\beq
	I_1(f, g) \lesssim \Lambda \norm{f}_{ H^s(\R^d)}\norm{g}_{ H^s(\R^d)}.
\label{eq:aux1}
\eeq

For $I_2$ we distinguish the far and near part for $0 < R < \frac{\bar R}{2}$ to be determined below
\bal\label{eq:I2-split}
	I_2(f, g) =  \underbrace{\int_{B_{ \bar R}}\int_{B_{\bar R} \cap \abs{v-w} < R} \dots}_{=: I_{21}(f, g)} + \underbrace{\int_{B_{ \bar R}}\int_{\abs{v-w} > R } \dots}_{=: I_{22}(f, g)}.
\eal
We rewrite $I_{21}$ with Fubini's theorem
\bals
	I_{21}(f, g) = &\underbrace{\int_{B_{ \bar R}\times B_{ \bar R} } \big(K(v, w) - K(w, v)\big)f(v) g(v)\chi_{\abs{v-w} < R}\dd w\dd v}_{=: I_{211}(f, g)} \\
	&+\underbrace{\int_{B_{ \bar R}\times B_{\bar R}}\big(K(v, w) - K(w, v)\big)f(v) g(w)\chi_{\abs{v-w} < R}\dd w\dd v}_{=: I_{212}(f, g)}.
\eals
Then for $I_{211}$ we use the cancellation assumption \eqref{eq:cancellation1} and \eqref{eq:cancellation2}. We get
\bal
	I_{211}(f, g) &\lesssim \int_{B_{ \bar R}}\Bigg\vert\textrm{PV} \int_{B_{\bar R}} \big(K(v, w) - K(w, v)\big)\dd w\Bigg\vert \Abs{f(v)g(v)}\dd v\\
	&\lesssim \Lambda \int_{B_{ \bar R}}\Abs{f(v)g(v)}\dd v \\
	&\leq \Lambda \norm{f}_{L^2(B_{ \bar R})}\norm{g}_{L^2(B_{ \bar R})}.
\label{eq:aux2}
\eal
For $I_{212}$ we use Taylor's theorem to write 
\bal\label{eq:I-212}
	I_{212}(f, g) = \int_{B_{ \bar R}}&\int_{B_{ \bar R}} \big[K(v, w) - K(w, v)\big]\chi_{\abs{v-w} < R}f(v)\\
	&\times\Bigg[g(v) + \nabla g(v) \cdot (w-v)\\
	&\qquad + \frac{1}{2}(w-v)^T\Bigg(\int_{0}^1\textrm D^2 g(\tau w + (1-\tau)v)\dd \tau\Bigg)(w - v)\Bigg]\dd w\dd v.
\eal
We can bound the first term as above in \eqref{eq:aux2}. For the second term, in case that $s \in \Big(0, \frac{1}{2}\Big)$, we use \eqref{eq:upperbound-2} so that 
\bals
	\int_{B_{ \bar R}}\Bigg\lvert\int_{B_{ \bar R}} \big[K(v, w) - K(w, v)\big]&\chi_{\abs{v-w} < R}(w-v)\dd w \Bigg\rvert\Abs{ f(v)\nabla g(v)} \dd v \\
	&\lesssim \Lambda R^{1-2s} \int_{B_{ R}} \Abs{f(v)\nabla g(v) }\dd v\\
	&\leq \Lambda R^{1-2s} \norm{f}_{L^2(B_{\bar R})}\norm{\nabla g}_{L^2(B_{\bar R})} \\
	&\leq \Lambda R^{1-2s} \norm{f}_{L^2(B_{\bar R})} \norm{g}_{L^2(B_{\bar R})}^{\frac{1}{2}}\norm{\textrm D^2g}_{L^2(B_{\bar R})}^{\frac{1}{2}}\\
	&\leq \Lambda \norm{f}_{L^2(B_{\bar R})} \norm{g}_{L^2(B_{\bar R})}^{1-s}\norm{g}_{H^2(B_{\bar R})}^s.
\eals
We used the Gagliardo-Nirenberg inequality and chose $R^2 = \frac{\norm{g}_{L^2}}{\norm{\textrm D^2g}_{L^2}}$. 
Else if $s \in \big[1/2, 1)$, we use \eqref{eq:cancellation2}, so that with this choice of $R$ we get
\bals
	\int_{B_{ \bar R}}\Bigg\lvert\int_{B_{\bar R}} \big[K(v, w) - K(w, v)\big]&\chi_{\abs{v-w} < R}(w-v)\dd w \Bigg\rvert\Abs{ f(v)\nabla g(v)} \dd v \\
	&\lesssim \Lambda \big(1+R^{1-2s} \big)\int_{B_{ \bar R}} \Abs{f(v)\nabla g(v) }\dd v\\
	&\leq \Lambda \big(1+R^{1-2s} \big)\norm{f}_{L^2(B_{\bar R})}\norm{\nabla g}_{L^2(B_{\bar R})} \\
	&\leq \Lambda \big(1+R^{1-2s} \big)\norm{f}_{L^2(B_{\bar R})} \norm{g}_{L^2(B_{\bar R})}^{\frac{1}{2}}\norm{\textrm D^2g}_{L^2(B_{\bar R})}^{\frac{1}{2}}\\
	&\leq \Lambda \norm{f}_{L^2(B_{\bar R})}\Big(\norm{g}_{L^2(B_{\bar R})} + \norm{g}_{L^2(B_{\bar R})}^{1-s}\norm{g}_{H^2(B_{\bar R})}^s\Big).
\eals
For the last term in \eqref{eq:I-212}, we have with the Cauchy-Schwarz inequality and \eqref{eq:upperbound-2}
\bals
	&\int_{B_{ \bar R}}\int_{B_{ \bar R}} \big[K(v, w) - K(w, v)\big]\chi_{\abs{v-w} < R}f(v)(w-v)^T \Bigg(\int_{0}^{1}\textrm D^2 g(\tau w + (1-\tau)v)\dd \tau\Bigg)(w - v)\dd w\dd v\\
	&\leq \norm{f}_{L^2(B_{ \bar R})}\\
	&\qquad \times\Bigg\{\int_{B_{ \bar R}}\Bigg\vert\int_{B_{ \bar R}} \big[K(v, w) - K(w, v)\big]\chi_{\abs{v-w} < R}\abs{w-v}^2\int_{0}^{1}\textrm D^2 g(\tau w + (1-\tau)v)\dd \tau\dd w\Bigg\vert^2\dd v\Bigg\}^{\frac{1}{2}}\\
	&\leq \norm{f}_{L^2(B_{ \bar R})}\Bigg\{\int_{B_{ \bar R}}\Bigg\vert\int_{B_{ \bar R}} \Abs{K(v, w) - K(w, v)}\chi_{\abs{v-w} < R}\abs{w-v}^2\dd w\Bigg\vert\\
	&\qquad \times\Bigg\vert \int_{B_{\bar R}} \Abs{K(v, w) - K(w, v)}\chi_{\abs{v-w} < R}\abs{w-v}^2\Bigg\vert\int_{0}^{1}\textrm D^2 g(\tau w + (1-\tau)v)\dd \tau\Bigg\vert^2\dd w\Bigg\vert\dd v\Bigg\}^{\frac{1}{2}}\\
	&\leq \Lambda R^{1-s}\norm{f}_{L^2(B_{ \bar R})}\\
	&\qquad \times\Bigg\{\int_{B_{ \bar R}}\int_{B_{ \bar R}} \Abs{K(v, w) - K(w, v)}\chi_{\abs{v-w} < R}\abs{w-v}^2\int_{0}^{1} \Abs{\textrm D^2 g(\tau w + (1-\tau)v)}^2\dd \tau\dd w\dd v\Bigg\}^{\frac{1}{2}}.
\eals
For $\tau \in [0, \frac{1}{2})$ we perform the change of variable $v \to \tau w +(1-\tau)v$, whereas for $\tau \in [\frac{1}{2},1]$ we perform the change of variable $w \to \tau w +(1-\tau)v$. We get in the former case with \eqref{eq:upperbound-2} and Fubini's theorem
\bals
	\int_{B_{ \bar R}}\int_{B_{ \bar R}} &\Abs{K(v, w) - K(w, v)}\chi_{\abs{v-w} < R}\abs{w-v}^2\int_{0}^{\frac{1}{2}} \abs{\textrm D^2 g(\tau w + (1-\tau)v)}^2\dd \tau\dd w\dd v \\
	&\leq\int_{0}^{\frac{1}{2}}\int_{B_{\bar R}}\int_{B_{ \tilde R}} \Abs{K(\tilde v, w) - K(w, \tilde v)}\chi_{\abs{\tilde v-w} < R}\abs{w-\tilde v}^2 \abs{\textrm D^2 g(\tilde v)}^2 (1-\tau)^{-3}\dd \tilde v\dd w\dd \tau\\
	&\leq C\Lambda R^{2-2s}\norm{\textrm{D}^2 g}_{L^2(B_{\tilde R})},
\eals
where $0 < \tilde R$ is chosen such that $B_{(1-\tau)\bar R}(\tau w) \subset B_{\tilde R}$ for all $w \in B_{\bar R}$ and $\tau \in [0, \frac{1}{2})$. In the latter case we have similarly
\bals
	\int_{B_{ \bar R}}\int_{B_{ \bar R}} &\Abs{K(v, w) - K(w, v)}\chi_{\abs{v-w} < R}\abs{w-v}^2\int_{\frac{1}{2}}^1 \abs{\textrm D^2 g(\tau w + (1-\tau)v)}^2\dd \tau\dd w\dd v \\
	&\leq\int_{\frac{1}{2}}^1\int_{B_{\bar R}}\int_{B_{ \tilde R}} \Abs{K(v, \tilde w) - K(\tilde w, v)}\chi_{\abs{v-\tilde w} < R}\abs{\tilde w- v}^2 \abs{\textrm D^2 g(\tilde w)}^2\tau^{-3}\dd \tilde w\dd v\dd \tau\\
	&\leq C\Lambda R^{2-2s}\norm{\textrm{D}^2 g}_{L^2(B_{\tilde R})},
\eals
when we choose $0 < \tilde R$ such that $B_{\tau \bar R}((1-\tau) v) \subset B_{\tilde R}$ for all $v \in B_{\bar R}$. Thus we get overall
\bals
	\int_{B_{ \bar R}}\int_{B_{ \bar R}} \big[K(v, w) - K(w, v)\big]\chi_{\abs{v-w} < R}&f(v)(w-v)^T\int_{0}^1\textrm D^2 g(\tau w + (1-\tau)v)\dd \tau(w - v)\dd w\dd v\\
	&\leq C\Lambda^2 R^{2-2s}\norm{f}_{L^2(B_{\bar R})}\norm{\textrm{D}^2 g}_{L^2(B_{\tilde R})}\\
	&\leq C\Lambda\norm{f}_{L^2(B_{\bar R} )} \norm{g}_{L^2(B_{\bar R})}^{1-s}\norm{g}_{H^2(B_{\bar R})}^s.
\eals
Therefore we have obtained 
\bal
	I_{21}(f, g) \lesssim_\Lambda  \norm{f}_{L^2(B_{ \bar R})}\norm{g}_{L^2(B_{\bar  R})} +  \norm{f}_{L^2(B_{\bar R})} \norm{g}_{L^2(B_{\bar R})}^{1-s}\norm{g}_{H^2(B_{\bar R})}^s.
\label{eq:aux3}
\eal

For the far part $I_{22}$ in \eqref{eq:I2-split} we apply Taylor's theorem 
\bals
	I_{22}(f, g) &= \int_{B_{\bar R}}\int_{\R^d\setminus B_R(v)} \big[f(v) - f(w)\big]\big[g(v)+g(w)\big] K(v, w)\dd w\dd v \\
	&= \int_{B_{\bar R}}\int_{\R^d\setminus B_R(v)} K(v, w)\big[f(v) - f(w)\big]\Bigg[2g(v) + \nabla g(v)\cdot (w-v)\\
	&\qquad\qquad\qquad\qquad \qquad\qquad+\frac{1}{2}(w-v)^T\cdot \Bigg(\int_0^1\textrm D^2 g(\tau w + (1-\tau)v)\dd \tau\Bigg) \cdot (w-v)\Bigg] \dd w\dd v.
\eals
The trickiest term to estimate is the last one. We get with Fubini and Cauchy-Schwarz
\bals
	&\int_{B_{\bar R}}\int_{\R^d\setminus B_R(v)}K(v, w) \big[f(v) - f(w)\big] \abs{w-v}^2 \int_0^1\textrm D^2 g(\tau w + (1-\tau) v) \dd \tau \dd w\dd v\\
	&\leq \norm{f}_{L^2(B_{\bar R})}\Bigg\{\int_{B_{\bar R}}\Bigg\vert\int_{\R^d\setminus B_R(v)}K(v, w)  \abs{w-v}^2 \int_0^1\textrm D^2 g(\tau w + (1-\tau) v) \dd \tau \dd w\Bigg\vert^2\dd v\Bigg\}^{\frac{1}{2}}\\
	&\quad + \norm{f}_{L^2(\R^d \setminus B_{\bar R})}\Bigg\{\int_{\R^d}\Bigg\vert\int_{B_{\bar R}\setminus B_R(w)}K(v, w)  \abs{w-v}^2 \int_0^1\textrm D^2 g(\tau w + (1-\tau) v) \dd \tau \dd v\Bigg\vert^2\dd w\Bigg\}^{\frac{1}{2}}\\
	&\leq \Lambda R^{1-s}\norm{f}_{L^2(B_{\bar R})}\Bigg\{\int_{B_{\bar R}}\int_{\R^d\setminus B_R(v)}K(v, w)  \abs{w-v}^2 \int_0^1 \Abs{\textrm D^2 g(\tau w + (1-\tau) v)}^2 \dd \tau \dd w\dd v\Bigg\}^{\frac{1}{2}}\\
	&\quad +  \Lambda R^{1-s}\norm{f}_{L^2(\R^d \setminus B_{\bar R})}\Bigg\{\int_{B_{\bar R}}\int_{\R^d\setminus B_R(v)}K(v, w)  \Abs{w-v}^2 \int_0^1 \abs{\textrm D^2 g(\tau w + (1-\tau) v)}^2 \dd \tau \dd w\dd v\Bigg\}^{\frac{1}{2}}.
\eals
With the same change of variables as above we find for $\tilde R$ sufficiently large such that $B_{(1-\tau)\bar R}(\tau w) \subset B_{\tilde R}$ for all $w \in \R^d \setminus B_R(v)$ and $\tau \in [0, \frac{1}{2}]$
\bals
	\int_{B_{\bar R}}\int_{\R^d\setminus B_R(v)}K(v, w) \abs{w-v}^2 &\int_0^{\frac{1}{2}} \Abs{\textrm D^2 g(\tau w + (1-\tau) v)}^2 \dd \tau \dd w\dd v\\
	&\leq C\int_{B_{\tilde R}}\int_{\R^d\setminus B_R(\tilde v)}K(\tilde v, w) \abs{w - \tilde v}^2 \abs{\textrm D^2 g(\tilde v)}^2 \dd w\dd \tilde v   \\
	&\leq C\Lambda R^{2-2s} \norm{\textrm D^2 g}^2_{L^2(\R^d)}.
\eals
Similarly
\bals
	\int_{B_{\bar R}}\int_{\R^d\setminus B_R(v)}K(v, w) \abs{w-v}^2 &\int_{\frac{1}{2}}^1 \abs{\textrm D^2 g(\tau w + (1-\tau) v)}^2 \dd \tau \dd w\dd v\\
	&\leq C\int_{B_{\tilde R  }}\int_{\R^d \setminus B_{R}(v)} K(v,\tilde w) \abs{\tilde w - v}^2 \abs{\textrm D^2 g(\tilde w)}^2\dd \tilde w \dd v \\
	&\leq C\Lambda R^{2-2s} \norm{\textrm D^2 g}_{L^2(\R^d)}^2.
\eals
Thus with our choice of $R$ we get
\bal
	\int_{B_{\bar R}}\int_{\R^d\setminus B_R(v)}K(v, w) &\big[f(v) - f(w)\big] \abs{w-v}^2 \int_0^1\textrm D^2 g(\tau w + (1-\tau) v) \dd \tau \dd w\dd v\\
	&\leq C\Lambda R^{2-2s}\norm{f}_{L^2(\R^d)}\norm{\textrm D^2 g}_{L^2(\R^d)}\\
	&\leq C\Lambda\norm{f}_{L^2(\R^d)} \norm{g}_{L^2(\R^d)}^{1-s}\norm{g}_{H^2(\R^d)}^s.
\label{eq:auxint}
\eal
Overall we find
\bals
	I_{22}(f, g) &\lesssim \int_{B_{\bar R}}\Bigg(\int_{\R^d\setminus B_R(v)} \big(f(v) - f(w)\big)^2K(v, w)\dd w\Bigg)^{\frac{1}{2}} \Bigg(\int_{\R^d\setminus B_R(v)}K(v, w)\dd w\Bigg)^{\frac{1}{2}}g(v)\dd v\\
	&\quad + \int_{B_{\bar R}}\Bigg(\int_{\R^d\setminus B_R(v)} \big(f(v) - f(w)\big)^2K(v, w)\dd w\Bigg)^{\frac{1}{2}}\\
	&\qquad \qquad \qquad \times \Bigg(\int_{\R^d\setminus B_R(v)}\abs{w-v}^2K(v, w)\dd w\Bigg)^{\frac{1}{2}}\abs{\nabla g(v)}\dd v\\
	&\quad+C\Lambda R^{2-2s}\norm{f}_{L^2(\R^d)}\norm{\textrm D^2 g}_{L^2(\R^d)}\\
	&\lesssim \Lambda R^{-s}\norm{f}_{H^s}\norm{g}_{L^2} + \Lambda R^{1-s}\norm{f}_{H^s}\norm{\nabla g}_{L^2}+ \Lambda R^{2-2s}\norm{f}_{L^2(\R^d)}\norm{\textrm D^2 g}_{L^2(\R^d)}\\
	&\leq  \Lambda R^{-s}\norm{f}_{H^s}\norm{g}_{L^2} + \Lambda R^{1-s}\norm{f}_{H^s}\norm{g}_{L^2}^{\frac{1}{2}}\norm{\textrm D^2g}_{L^2}^{\frac{1}{2}}+ \Lambda R^{2-2s}\norm{f}_{L^2(\R^d)}\norm{\textrm D^2 g}_{L^2(\R^d)}\\
	&\leq  \Lambda \norm{f}_{H^s}\norm{g}_{L^2}^{1-\frac{s}{2}} \norm{g}_{H^2}^{\frac{s}{2}} + \Lambda\norm{f}_{L^2(\R^d)} \norm{g}_{L^2(\R^d)}^{1-s}\norm{g}_{H^2(\R^d)}^s.
\eals
We used the same proof as for the symmetric part to deduce the third inequality, repeatedly the Cauchy Schwarz inequality, the upper bound \eqref{eq:upperbound-2}, the Gagliardo-Nirenberg inequality and the choice of $R$ as above.

Now we use Littlewood-Paley theory inspired from the proof of Theorem 4.1 in \cite{IS}. We denote with $\Delta_i$ the Littlewood-Paley projectors defined by $\Delta_i f(z) = \int_{\xi: 2^i < \abs\xi < 2^{i+1}} \hat f(\xi) e^{2\pi i \xi \cdot z} \dd \xi$ where $\hat f$ is the Fourier transform of $f$. We decompose $f = \sum_{i=0}^\infty \Delta_i f$, using the convention that all low modes are contained in $\Delta_0$ so that the index $i \geq 0$. Note that for $s \geq 0$
\beqs
	\norm{\Delta_i f}_{H^s} \approx 2^{is}\norm{\Delta_i f}_{L^2}.
\eeqs
Moreover, we bound as in \cite{IS} (for a justification see \cite{bahouri})
\beqs
	\norm{\Delta_i g}_{L^2}^{1-s}\norm{\Delta_i g}_{H^2}^s \lesssim 2^{si} \norm{\Delta_i g}_{H^s}.
\eeqs
and
\beqs
	\norm{\Delta_i g}_{L^2}^{1-\frac{s}{2}}\norm{\Delta_i g}_{H^2}^{\frac{s}{2}} \lesssim \norm{\Delta_i g}_{H^s}.
\eeqs
Then
\bal
	I_{2}(f, g) &= \sum_{i, j} I_{2}(\Delta_i f, \Delta_j g) \\
	&\lesssim \sum_{j \leq i} \norm{\Delta_i f}_{L^2} \norm{\Delta_j g}_{L^2}^{1-s}\norm{\Delta_j g}_{H^2}^s + \sum_{i < j} \norm{\Delta_j g}_{L^2} \norm{\Delta_i f}_{L^2}^{1-s}\norm{\Delta_i f}_{H^2}^s\\
	&\quad + \sum_{i, j}\Big( \norm{\Delta_i f}_{H^s}\norm{\Delta_j g}^{1-\frac{s}{2}}_{L^2} \norm{\Delta_j g}^{\frac{s}{2}}_{H^2} + \norm{\Delta_i f}_{L^2}\norm{\Delta_j g}_{L^2} \Big)\\
	&\lesssim \sum_{i, j} \Big(2^{-s\abs{i-j}}\norm{\Delta_if}_{H^s}\norm{\Delta_jg}_{H^s} +\norm{\Delta_i f}_{H^s}\norm{\Delta_j g}_{H^s}\Big) + \norm{f}_{L^2}\norm{g}_{L^2}\\
	&\leq \sum_{k = 0}^\infty 2^{-sk}\Bigg( \sum_{i = 0}^\infty \norm{\Delta_if}_{H^s}\norm{\Delta_{i+k}g}_{H^s} + \norm{\Delta_{i+k}f}_{H^s}\norm{\Delta_{i}g}_{H^s} \Bigg)\\
	&\quad+ \norm{f}_{H^s}\norm{g}_{H^s}  + \norm{f}_{L^2}\norm{g}_{L^2}\\
	&\leq \sum_{k = 0}^\infty 2^{1-sk}\Bigg( \sum_{i = 0}^\infty \norm{\Delta_if}^2_{H^s}\Bigg)^{\frac{1}{2}}\Bigg( \sum_{i = 0}^\infty \norm{\Delta_ig}^2_{H^s}\Bigg)^{\frac{1}{2}}\\
	&\quad + \norm{f}_{H^s}\norm{g}_{H^s}+ \norm{f}_{L^2}\norm{g}_{L^2}\\
	&\lesssim \norm{f}_{H^s}\norm{g}_{H^s}.
\label{eq:aux3.1}
\eal
Note that splitting the sum  into $j \leq i$ and $i < j$ works since the adjoints of the corresponding integral operators satisfy the same bounds as in \eqref{eq:aux3} and \eqref{eq:auxint} due to the cancellation assumptions. Thus we conclude with \eqref{eq:aux1} and \eqref{eq:aux3.1}.
\end{proof}

To motivate the following definition of weak solutions, we recall Lemma 5.6 of \cite{IS}, which states that the operator $\mathcal L$ is bounded from the space $L^\infty\big(\R^d\setminus B_{\bar R/2}\big)+H^s(\R^d)$ to $H^{-s}(B_{\bar R})$, where the former space is equipped with the norm 
\beqs
	\norm{f}_{L^\infty(\R^d\setminus B_{\bar R/2})+H^s(\R^d)} = \inf\Big\{\norm{f_1}_{L^\infty(\R^d\setminus B_{\bar R/2})} + \norm{f_2}_{H^s(\R^d)} : f = f_1 + f_2, ~ f_1 = 0 \textrm{ in } B_{\bar R/2}\Big\}.
\eeqs 
\begin{lemma}
Let $\varphi \in H^s(\R^d)$ with $\textrm{supp } \varphi$ compactly contained in $B_{\bar R/2}$. Let $K$ be a kernel with \eqref{eq:upperbound-2}, \eqref{eq:cancellation1}, \eqref{eq:cancellation2}. Then for $f \in L^\infty(\R^d\setminus B_{\bar R/2})+H^s(\R^d)$ there exists $C$ depending on $s, d, \Lambda$ and $\textrm{supp }\varphi$ so that
\beqs
	\mathcal E(f, \varphi)\leq C\norm{f}_{L^\infty(\R^d\setminus B_{\bar R/2})+H^s(\R^d)}\norm{\varphi}_{H^s(\R^d)}.
\eeqs
\end{lemma}
\begin{proof} 
The proof comes from Lemma 5.6 in \cite{IS}. It suffices to consider smooth $f, \varphi$ just as above. We write $f = f_1 + f_2$ with $f_1$ and $f_2$ as in the definition of the norm for the space $L^\infty(\R^d\setminus B_{\bar R/2})+H^s(\R^d)$. By Theorem \ref{thm:4.1} we have $\Abs{\mathcal E(f_2, \varphi)} \leq C \norm{f_2}_{H^s(\R^d)}\norm{\varphi}_{H^s(\R^d)}$. For $\mathcal E(f_1, \varphi)$ we get
\bals
	\Abs{\mathcal E(f_1, \varphi)} &= \Bigg\vert\lim_{\varepsilon \to 0}\int\int_{\abs{w-v}>\varepsilon} \big[f_1(w)-f_1(v)\big]\varphi(v)K(v, w)\dd w\dd v\Bigg\vert\\
	&= \lim_{\varepsilon \to 0}\int_{\textrm{supp }\varphi}\int_{\R^d\setminus B_\varepsilon(v)} f_1(w)K(v, w)\dd w\varphi(v)\dd v\\
	&= \int_{\textrm{supp }\varphi}\int_{\R^d\setminus B_\delta(v)} f_1(w)K(v, w)\dd w\varphi(v)\dd v,
\eals	
where $\delta = \textrm{dist }\big(\textrm{supp }\varphi, \R^d\setminus B_{\bar R/2}\big)$. Using the upper bound on the kernel \eqref{eq:upperbound-2}, we find
\beqs
	\Abs{\mathcal E(f_1, \varphi)} \leq \Lambda\delta^{-2s}\norm{f_1}_{L^\infty(\R^d\setminus B_{\bar R/2})}\int_{\textrm{supp }\varphi}\varphi(v)\dd v\leq C\norm{f_1}_{L^\infty(\R^d\setminus B_{\bar R/2})}\norm{\varphi}_{H^s}.
\eeqs
\end{proof}
The following definition coincides with \cite[Definition 5.7]{IS}.
\begin{definition}[Weak Solutions]
Assume $0 \leq K$ satisfies \eqref{eq:coercivity}-\eqref{eq:cancellation2} for $\bar R > 0$. We say that $f: \big(-(\frac{\bar R}{2})^{2s}, 0\big) \times B_{(\frac{\bar R}{2})^{1+2s}}\times \R^d \to \R$ is a weak sub-solution of \eqref{eq:1.1}-\eqref{eq:1.2} in $Q_{\frac{\bar R}{2}} =: \mathcal I \times \Omega_x \times \Omega_v$ if 
\begin{enumerate}
	\item $f \in C^0(\mathcal I, L^2(\Omega_x\times \Omega_v))\cap L^2(\mathcal I \times \Omega_x, L^\infty(\R^d\setminus \Omega_v)+H^s\big(\R^d))$,
	\item $\mathcal T f \in L^2(\mathcal I\times \Omega_x, H^{-s}(\Omega_v))$,
	\item for all non-negative $\varphi \in L^2(\mathcal I \times \Omega_x, H^s(\R^d))$ such that for every $t, x$ the support of $\varphi(t, x, \cdot)$ is compactly contained in $\Omega_v$ there holds
	\beqs
		\int_{Q_{\frac{\bar R}{2}}} (\mathcal T f)\varphi\dd z + \int_{\Omega_x}\int_{\mathcal I}\mathcal E(f, \varphi)\dd t\dd x - \int_{Q_{\frac{\bar R}{2}}} h\varphi\dd z \leq 0.
	\eeqs
\end{enumerate}
A function $f$ is a super-solution of \eqref{eq:1.1}-\eqref{eq:1.2} in $Q_{\frac{\bar R}{2}}$ if $-f$ is a sub-solution in $Q_{\frac{\bar R}{2}}$. It is a solution if it is a sub- and super-solution.
\end{definition}

\section{Integral Estimates}
\label{subsec:intest}
\subsection{Kolmogorov's fundamental solutions}
\label{subsec:kolmogorov}
In this subsection, we consider the fractional Kolmogorov equation given by
\beq
	\partial_t f + v \cdot \nabla_x f + (-\Delta)^s f = h - m,
\label{eq:2.1}
\eeq
for some $h \in L^2\big([-\tau, 0]\times \R^{d}, H^{-s}(\R^d)\big)$ and some non-negative measure $m \in M^1\big([-\tau, 0]\times \R^{2d}\big)$ with finite mass.  Then there exists $h_1, h_2 \in L^2\big([-\tau, 0]\times \R^{2d}\big)$ so that $h = h_1 + (-\Delta)^{\frac{s}{2}}_v h_2$ and 
\bals
	\norm{h}_{L^2([-\tau,0]\times \R^{d}, H^{-s}(\R^d))} \approx \norm{h_1}_{L^2([-\tau, 0]\times \R^{2d})}+\norm{h_2}_{L^2([-\tau, 0]\times \R^{2d})}.
\eals
\begin{proposition}\label{prop:2.2}
Let $0 \leq f$ solve \eqref{eq:2.1} in $[-\tau, 0]\times \R^{2d}$ for $h = h_1 + (-\Delta)^{\frac{s}{2}}_v h_2$ with $h_1, h_2\in L^1\cap L^2([-\tau,0]\times\R^{2d})$, $0\leq m \in M^1([-\tau,0]\times \R^{2d})$ and with $f(0, x, v) = f_0(x, v) \in L^1\cap L^2(\R^{2d})$. Then
\beq
	\norm{f}_{L^p([-\tau, 0]\times \R^{2d})} \lesssim \norm{h_1}_{L^2([-\tau, 0]\times \R^{2d})}+\norm{h_2}_{L^2([-\tau, 0]\times \R^{2d})} + \norm{f_0}_{L^2(\R^{2d})},
\label{eq:2.2}
\eeq
where $2 \leq p < 2 + \frac{2s}{d(1+s)}$. 
Moreover, for $\sigma \in \left[0, \frac{s}{1+2s}\right)$
\bal
	\norm{f}_{L^1_{t, v}W_x^{\sigma, 1}([-\tau, 0]\times \R^{2d})} \lesssim &\norm{h_1}_{L^1([-\tau, 0]\times \R^{2d})} + \norm{h_2}_{L^1([-\tau, 0]\times \R^{2d})}\\
	&+ \norm{m}_{M^1([-\tau, 0]\times \R^{2d})}  + \norm{f_0}_{L^1(\R^{2d})}.
\label{eq:2.3}
\eal
\end{proposition}
\begin{proof}
Equation \eqref{eq:2.1} admits a fundamental solution, see for example Theorem 1.1 in \cite{Niebel} or Section 2.4 in \cite{IS}, given by
\bals
	f(t, x, v) &= \int_{\R^{2d+1}}\big[h(t', x', v') - m(t', x', v')\big]J\left(t-t', x-x'-(t-t')v', v-v'\right)\dd t'\dd x'\dd v' \\
	&\qquad+ \int_{\R^{2d}}f_0(x', v')J\left(t, x-x'-(t-t')v', v-v'\right)\dd x'\dd v',
\eals
where $(t, x, v) \in (0, \tau)\times \R^{2d}$ and
\beqs
	J(t, x, v) = \frac{C_d}{t^{d+\frac{d}{s}}}\mathcal J\left(\frac{x}{t^{1+\frac{1}{2s}}}, \frac{v}{t^{\frac{1}{2s}}}\right), \qquad \hat{\mathcal J}(\varphi, \xi) = \exp\left(-\int_0^1 \abs{\xi -\tau\varphi}^{2s}\dd \tau\right).
\eeqs
Since $m$ and $J$ are non-negative, we deduce that
\bals
	0 \leq f(t, x, v) &\leq \int_{\R^{2d+1}}h(t', x', v')J\left(t-t', x-x'-(t-t')v', v-v'\right)\dd t'\dd x'\dd v' \\
	&\qquad+ \int_{\R^{2d}}f_0(x', v')J\left(t, x-x'-(t-t')v', v-v'\right)\dd x'\dd v'.
\eals
We remark that for any $r \geq 1$ there holds for $t>0$
\bals
	&\left\Vert J(t, \cdot, \cdot)\right\Vert_{L^r(\R^{2d})} =t^{-d(1+\frac{1}{s})(1-\frac{1}{r})}\left\Vert\mathcal J\right\Vert_{L^r(\R^{2d})}, \\
	&\left\Vert(-\Delta)^{\frac{s}{2}}_vJ(t, \cdot, \cdot)\right\Vert_{L^r(\R^{2d})} =t^{-d(1+\frac{1}{s})(1-\frac{1}{r})-\frac{1}{2}}\left\Vert(-\Delta)^{\frac{s}{2}}_v\mathcal J\right\Vert_{L^r(\R^{2d})}.
\eals
In particular for $r = p^* := \frac{2d(1+s) + 2s}{2d(1+s) + s} \in (1, 2)$ we deduce
\bals
	&\left \Vert J(t, \cdot, \cdot)\right\Vert_{L^{p^*}(\R^{2d})} =t^{\frac{1}{2}-\frac{1}{p^*}}\left\Vert \mathcal J\right\Vert_{L^{p^*}(\R^{2d})}, \\
	&\left \Vert (-\Delta)^{\frac{s}{2}}_vJ(t, \cdot, \cdot)\right\Vert_{L^{p^*}(\R^{2d})} =t^{-\frac{1}{p^*}}\left\Vert (-\Delta)^{\frac{s}{2}}_v\mathcal J\right\Vert_{L^{p^*}(\R^{2d})}.
\eals

We define as in Section 2.4 of \cite{IS} the modified convolution
\bals
	f*_t g(x, v):= \int_{\R^{2d}}f(x', v')g(x - x' -tv', v-v')\dd x'\dd v'.
\eals	
We remark that the modified convolution satisfies the usual Young inequality independent of $t$:
\bals
	\norm{f*_t g}_{L^r_{x, v}} \leq \norm{f}_{L^p_{x, v}}\norm{g}_{L^q_{x, v}}
\eals
for $1 + \frac{1}{r} = \frac{1}{p} + \frac{1}{q}$. Following the proof of Proposition 2.2 in \cite{IS}, we split $f = \tilde f_0 + f_1 + f_2$ with
\bals
	&\tilde f_0(t, \cdot, \cdot) := f_0*_tJ(t, \cdot, \cdot),\\
	&f_1(t, \cdot, \cdot) := \int_0^t h_1*_{(t-t')}J(t-t', \cdot, \cdot)\dd t',\\
	&f_2(t, \cdot, \cdot) := \int_0^t h_2*_{(t-t')}(-\Delta)_v^{\frac{s}{2}}J(t-t', \cdot, \cdot)\dd t'.
\eals
Let $q \in [1, p^*)$ be such that $\frac{1}{p} = \frac{1}{q} - \frac{1}{2}$. By Young's inequality we get for $\alpha = d\left(1+\frac{1}{s}\right)\left(1-\frac{1}{q}\right)+\frac{1}{2} < \frac{1}{p^*}<\frac{1}{q}$
\bals
	\norm{\tilde f_0(t, \cdot, \cdot)}_{L^p(\R^{2d})} &\leq \norm{f_0}_{L^2(\R^{2d})}\norm{J(t, \cdot, \cdot)}_{L^{q}(\R^{2d})} = \norm{f_0}_{L^2(\R^{2d})}\norm{\mathcal J}_{L^{q}(\R^{2d})}t^{\frac{1}{2}-\alpha},\\
	\norm{f_1(t, \cdot, \cdot)}_{L^p(\R^{2d})} &\leq \int_0^t \norm{h_1(t')}_{L^2(\R^{2d})}\norm{J(t-t', \cdot, \cdot)}_{L^{q}(\R^{2d})}\dd t'\\
	&= \int_0^t \norm{h_1(t')}_{L^2(\R^{2d})}\norm{\mathcal J}_{L^{q}(\R^{2d})}(t-t')^{\frac{1}{2}-\alpha}\dd t',\\
	\norm{f_2(t, \cdot, \cdot)}_{L^p(\R^{2d})} &\leq \int_0^t \norm{h_2(t')}_{L^2(\R^{2d})}\norm{(-\Delta)^{\frac{s}{2}}_v J(t-t', \cdot, \cdot)}_{L^{q}(\R^{2d})}\dd t'\\
	&= \int_0^t \norm{h_2(t')}_{L^2(\R^{2d})}\norm{(-\Delta)^{\frac{s}{2}}_v \mathcal J}_{L^{q}(\R^{2d})}(t-t')^{-\alpha}\dd t'.
\eals
Since $p\left(\frac{1}{2} - \alpha\right) > -1$ we get that $\tilde f_0 \in L^p([-\tau, 0]\times\R^{2d})$ and
\beqs
	\norm{\tilde f_0}_{L^p([-\tau, 0]\times\R^{2d})} \leq C \norm{f_0}_{L^2(\R^{2d})}.
\eeqs
For $f_1$ and $f_2$ we apply Young's inequality again and get
\bals
	&\norm{f_1}_{L^p([-\tau, 0]\times\R^{2d})} \leq C\norm{h_1}_{L^2([-\tau, 0]\times\R^{2d})}\tau^{\frac{1}{2}-\alpha+\frac{1}{q}},\\
	&\norm{f_2}_{L^p([-\tau, 0]\times\R^{2d})} \leq C\norm{h_2}_{L^2([-\tau, 0]\times\R^{2d})}\tau^{\frac{1}{q}-\alpha}.
\eals
This implies \eqref{eq:2.2}. 
To prove \eqref{eq:2.3} we follow the idea of Lemma 10 in \cite{JGCM}. We split for $(t,x,v)\in(0,\tau)\times\R^{2d}$ 
\beqs
	J = J_\varepsilon + J_\varepsilon^\perp \quad\textrm{with}\quad J_\varepsilon(t, x, v) := \eta\left(\frac{t}{\varepsilon}\right)J(t, x, v),
\eeqs
where $\varepsilon > 0$ and $\eta$ is a smooth function on $\R_+$ such that $0 \leq \eta \leq 1$, equal to $1$ in $[0, 1]$ and $0$ on $[2, +\infty)$ (we assume without loss of generality $\tau \geq 1$). Then we estimate for $l \in \N$
\bals
	\Abs{\nabla_x^l J_\varepsilon^\perp(t, x, v)} &\lesssim t^{-(d+\frac{d}{s})}\varepsilon^{-l(1+\frac{1}{2s})}\big\lvert\nabla_y^l \mathcal J_\varepsilon^\perp(y, w)\big\rvert,\\
	\Abs{(-\Delta)^{\frac{s}{2}}_v\nabla_x^l J_\varepsilon^\perp(t, x, v)} &+ t\abs{\nabla_x\nabla_x^l J_\varepsilon^\perp} \\
	&\lesssim t^{-(d+\frac{d}{s})}\varepsilon^{-l(1+\frac{1}{2s})}\Big[\varepsilon^{-\frac{1}{2s}}\big\lvert\nabla_y\nabla_y^l \mathcal J_\varepsilon^\perp(y,w)\big\rvert+\varepsilon^{-\frac{1}{2}}\big\lvert(-\Delta)_w^{\frac{s}{2}}\nabla_y^l\mathcal J_\varepsilon^\perp(y, w)\big\rvert\Big]\\
	&\lesssim t^{-(d+\frac{d}{s})}\varepsilon^{-l(1+\frac{1}{2s})-\frac{1}{2s}}\Big[\big\lvert\nabla_y\nabla_y^l \mathcal J_\varepsilon^\perp(y, w)\big\rvert+\big\lvert(-\Delta)_w^{\frac{s}{2}}\nabla_y^l\mathcal J_\varepsilon^\perp(y, w)\big\rvert\Big].
\eals	
Now assuming $\varepsilon < 1$ these estimates yield
\bals
	\norm{J_\varepsilon^\perp}_{L^1_{t, v}([-\tau, 0]\times \R^d; W_x^{l, 1}(\R^d))}&+\norm{(-\Delta)^{\frac{s}{2}}J_\varepsilon^\perp}_{L^1_{t, v}([-\tau,0]\times \R^d; W_x^{l, 1}(\R^d))}\\
	&+\norm{t \nabla_x J_\varepsilon^\perp}_{L^1_{t, v}([-\tau, 0]\times \R^d; W_x^{l, 1}(\R^d))} \lesssim \tau\varepsilon^{-l(1+\frac{1}{2s})-\frac{1}{2s}},
\eals
and
\bals
	\norm{J_\varepsilon}_{L^1_{t, v}([-\tau, 0]\times \R^{2d})}+\norm{(-\Delta)^{\frac{s}{2}}J_\varepsilon}_{L^1_{t, v}([-\tau, 0]\times \R^{2d})} + \norm{t \nabla_x J_\varepsilon}_{L^1_{t, v}([-\tau, 0]\times \R^{2d})}\lesssim \tau\varepsilon^{\frac{1}{2s}}.
\eals
The splitting on $J$ yields a splitting on the solution $f = f_\varepsilon + f_\varepsilon^\perp$. Young's convolution inequality and the convolution inequality on $M^1 \ast L^1 \to L^1$ imply
\bals
	\norm{f_\varepsilon^\perp}_{L^1_{t, v}W^{l, 1}_x([-\tau,0]\times \R^{2d})}\lesssim \tau\varepsilon^{-l(1+\frac{1}{2s})-\frac{1}{2s}}\big(\norm{h_1}_{L^1([-\tau,0]\times \R^{2d})}&+\norm{h_2}_{L^1([-\tau,0]\times \R^{2d})}\\
	&+ \norm{m}_{M^1([-\tau, 0]\times \R^{2d})}+\norm{f_0}_{L^1(\R^{2d})}\big),
\eals
and
\bals
	\norm{f_\varepsilon}_{L^1([-\tau,0]\times \R^{2d})}\lesssim \tau\varepsilon^{\frac{1}{2s}}\big(\norm{h_1}_{L^1([-\tau,0]\times \R^{2d})}&+\norm{h_2}_{L^1([-\tau,0]\times \R^{2d})}\\
	&+\norm{m}_{L^1([-\tau, 0]\times \R^{2d})}+\norm{f_0}_{L^1(\R^{2d})}\big).
\eals
The decomposition above holds for all $\varepsilon >0$; thus we can conclude the proof with the same justification as in the proof of Lemma 10 in \cite{JGCM}: Let $\sigma \in \big[0, \frac{s}{1+2s}\big)$. Using the notation $\langle\zeta\rangle :=(1+\abs{\zeta}^2)^{\frac{1}{2}}$ we can decompose 
\bal
	(1 - \Delta_x)^{\frac{\sigma}{2}}f(t, x, v) &= \int_{\R^{2d}}e^{i\zeta\cdot(x-\varphi)}\langle\zeta\rangle^\sigma f(t, \varphi, v)\dd\varphi\dd\zeta \\
	&= \sum_{k\geq-1}\int_{\R^{2d}}e^{i\zeta\cdot(x-\varphi)}a_k(\zeta) f(t, \varphi, v)\dd\varphi\dd\zeta\\
	&= \sum_{k\geq-1}\int_{\R^{2d}}e^{i\zeta\cdot(x-\varphi)}b_k(\zeta)(1-\Delta_\varphi)^{\frac{l}{2}}f(t, \varphi, v)\dd\varphi\dd\zeta,
\label{eq:2.4}
\eal
where $a_k(\zeta) := \langle\zeta\rangle^\sigma\phi_k$ and $b_k(\zeta) :=\langle\zeta\rangle^{\sigma -l}\phi_k$ with $\phi_k(\zeta) := [\eta(2^{-k}\zeta)-\eta(2^{-k+1}\zeta)]$ for $k\geq 0$ where $\eta$ is a smooth function such that $0 \leq \eta \leq 1$ and $\eta = 1$ in $B_1(0)$ and $0$ outside $B_2(0)$. For $k=-1$ we define $\phi_{-1}(\zeta) := \sum_{k< -1}[\eta(2^{-k}\zeta)-\eta(2^{-k+1}\zeta)]$. Note that for a given function $F = F(\varphi)$ there holds
\bals
	\int_{\R^d}\Big\lvert\int_{\R^{2d}}e^{i\zeta\cdot(x-\varphi)}a_k(\zeta)F(y)\dd\varphi\dd\zeta\Big\rvert\dd x&\lesssim 2^{k\sigma}\norm{F}_{L^1},\\
	\int_{\R^d}\Big\lvert\int_{\R^{2d}}e^{i\zeta\cdot(x-\varphi)}b_k(\zeta)(1-\Delta_\varphi)^{\frac{l}{2}}F(y)\dd\varphi\dd\zeta\Big\rvert\dd x&\lesssim 2^{k(\sigma-l)}\norm{F}_{W^{l,1}}.
\eals
Therefore we find for the solution $f$ to the fractional Kolmogorov equation
\bals
	&\norm{f}_{L^1_{t, v}W^{\sigma, 1}_x([-\tau,0]\times\R^{2d})} \\
	&\quad= \norm{(1-\Delta)_x^{\frac{\sigma}{2}}f}_{L^1([-\tau,0]\times\R^{2d})}\\
	&\quad= \int_{[-\tau,0]\times\R^{2d}}\Bigg\lvert \sum_{k\geq-1}\int_{\R^{2d}}\Big[e^{i\zeta\cdot(x-\varphi)}a_k(\zeta) f_\varepsilon(t, \varphi, v) \\
	&\qquad\qquad\qquad+ e^{i\zeta\cdot(x-\varphi)}b_k(\zeta)(1-\Delta_\varphi)^{\frac{l}{2}}f_\varepsilon^\perp(t, \varphi, v)\Big]\dd\varphi\dd\zeta\Bigg\rvert\dd t\dd x\dd v\\
	&\quad\lesssim \sum_{k\geq-1}\tau\left(2^{k\sigma}\varepsilon^{\frac{1}{2s}} + 2^{k(\sigma-l)}\varepsilon^{-l(1+\frac{1}{2s})-\frac{1}{2s}}\right)\left(\norm{h_1}_{L^1} + \norm{h_2}_{L^1}+\norm{f_0}_{L^1} +\norm{m}_{M^1}\right)\\
	&\quad\lesssim \frac{\tau}{\delta}\Big(\norm{h_1}_{L^1([-\tau,0]\times\R^{2d})} + \norm{h_2}_{L^1([-\tau,0]\times\R^{2d})}+\norm{f_0}_{L^1(\R^{2d})} + \norm{m}_{M^1([-\tau,0]\times\R^{2d})}\Big),
\eals
where we choose $\sigma = \frac{s}{1+2s} -\delta \in \big[0, \frac{s}{1+2s}\big)$, $\varepsilon = 2^{-2ks(\frac{s}{1+2s}-\frac{\delta}{2})}$ and $l > \frac{4s-3\delta(1+2s)}{(1+2s)(\delta(1+2s)+2-2s)}$ for some small $\delta > 0$. This concludes the proof. 
\end{proof}

\subsection{Energy estimates}
\label{subsec:enestim} The following two lemmas are used to extract information from the equation.
\begin{lemma}[Local energy estimate]\label{prop:9}
Let $f$ be a non-negative sub-solution to \eqref{eq:1.1}-\eqref{eq:1.2} in $Q_{\frac{\bar R}{2}}$ for some $\bar R \geq 1$ where $0\leq K$ satisfies \eqref{eq:coercivity}-\eqref{eq:cancellation2}. Let $Q_{r}(z_0)\subset Q_R(z_0)\subset Q_{\frac{\bar R}{2}}$ with $0 < r < R< \min\Big\{1, \frac{\bar R}{2}\Big\}$ be given. Assume $0 \leq f \leq 1$ a.e. in $\big(-R^{2s}+t_0, t_0\big] \times B_{R^{1+2s}}(x_0 + t_0v_0) \times \R^d$. Then there holds
\bals
	\sup_{\tau\in(-r^{2s}+t_0, t_0)}&\int_{Q^\tau_r(z_0)}f^2(\tau, x, v)\dd x\dd v+\int_{Q_r(z_0)}\int_{B_r(z_0)}\frac{\abs{f(w)-f(v)}^2}{\abs{v-w}^{d+2s}}\dd w\dd z\\
	&\lesssim_{\lambda, \Lambda, v_0}\int_{Q^{-r^{2s}+t_0}_R(z_0)}f_0^2(x, v)\dd x\dd v + (R-r)^{-2}\Abs{\{f > 0\} \cap Q_R(z_0)} + \int_{Q_R(z_0)} h^2\dd z,
\eals
where $z_0 = (t_0, x_0, v_0)$, $f_0(x, v) = f\left(-r^{2s}+t_0, x, v\right)$ and $$Q^\tau_r(z_0) = \left\{(x, v)\in \R^{2d} : (\tau, x, v) \in Q_r(z_0)\right\}.$$
\end{lemma}
\begin{proof}
For $t \in \big(-r^{2s}+t_0, t_0\big]$ we let $\psi \in C^\infty(\R^{2d})$ be a cut-off function with $0 \leq \psi \leq 1$, such that $\psi = 0$ in $Q_r^t(z_0)$ and $\psi = 1$ outside $Q_R^t(z_0)$. 
We test \eqref{eq:1.1} with the truncated function $(f - \psi)_+ = \sup\big(0, f - \psi\big)$. We notice that 
$f = (f - \psi)_+ + (f - \psi)_- + \psi$ and $\nabla_x (f - \psi)^2_+ = 2 (f - \psi)_+ \big(\nabla_x f - \nabla_x\psi\big)$. This yields for a.e. $t \in (-r^{2s}+t_0, t_0]$
\bals
	\int_{\R^{2d}} h (f - \psi)_+ \dd x\dd v &\geq \int_{\R^{2d}} \mathcal T f (f - \psi)_+ \dd x \dd v + \int_{\R^d} \mathcal E(f, (f - \psi)_+)\dd x\\
	&= \frac{1}{2}\frac{\dd}{\dd t}\int_{\R^{2d}} \mathcal  (f - \psi)_+^2 \dd x\dd v + \int_{\R^d}\mathcal E((f - \psi)_+, (f - \psi)_+) \dd x\\
	&\qquad + \int_{\R^d}\mathcal E((f - \psi)_-, (f - \psi)_+)\dd x + \int_{\R^d}\mathcal E(\psi, (f - \psi)_+)\dd x \\
	&\qquad+ \int_{\R^{2d}} (f - \psi)_+ \mathcal T\psi \dd x \dd v.
\eals
Due to the coercivity of $K$ \eqref{eq:coercivity-full} we know that 
\bals
	\mathcal E\big((f - \psi)_+, (f - \psi)_+\big) + \Lambda\norm{(f - \psi)_+}_{L^2}^2 \geq \lambda \norm{(f - \psi)_+}_{H^s_v}^2.
\eals
Moreover we notice that 
\bals
	\mathcal E\big((f - \psi)_-, (f - \psi)_+\big) &= \int_{\R^d} \int_{\R^d} \big[(f - \psi)_-(v) - (f - \psi)_-(w)\big](f - \psi)_+(v) K(v, w) \dd w\dd v\\
	&= - \int_{\R^d} \int_{\R^d}(f - \psi)_-(w)(f - \psi)_+(v) K(v, w) \dd w\dd v \geq 0,
\eals
since $(f-\psi)_+ \cdot (f-\psi)_- = 0$. This term will be fully exploited for the Second De Giorgi Lemma, Theorem \ref{thm:IVL}. Finally to bound $\mathcal E(\psi, (f - \psi)_+)$ we distinguish the symmetric and the skew-symmetric part
\bals
	-\mathcal E\big(\psi, (f-\psi)_+\big) &= \int_{\R^d} \int_{\R^d} (f - \psi)_+(v)\big[\psi(w)-\psi(v)\big] K(v, w) \dd w\dd v\\
	&= \frac{1}{2}\int_{\R^d} \int_{\R^d} \big[(f - \psi)_+(v) - (f - \psi)_+(w)\big]\big[\psi(w)-\psi(v)\big] K(v, w) \dd w\dd v \\
	&\quad+  \frac{1}{2}\int_{\R^d} \int_{\R^d} \big[(f - \psi)_+(v) + (f - \psi)_+(w)\big]\big[\psi(w)-\psi(v)\big] K(v, w) \dd w\dd v.
\eals
Realising that $\Abs{g_+(v) - g_+(w)} \leq \big(\chi_{\{g > 0 \}}(v) +  \chi_{\{g> 0 \}}(w)\big)\Abs{g_+(v) - g_+(w)}$ we find
\bals
	-\mathcal E^{\textrm{sym}}\big(\psi, (f - \psi)_+\big) &:= \frac{1}{2} \int_{\R^d} \int_{\R^d} \big[(f - \psi)_+(v) - (f - \psi)_+(w)\big]\big[\psi(w)-\psi(v)\big] K(v, w) \dd w\dd v\\
	&\leq \frac{1}{8} \mathcal E^{\textrm{sym}}\big((f - \psi)_+, (f - \psi)_+\big) \\
	&\quad+ \int\int \big[\chi_{\{ f - \psi > 0 \}}(v) + \chi_{\{ f - \psi > 0 \}}(w)\big]
\abs{\psi(w) - \psi(v)}^2K(v, w) \dd w \dd v\\
	&\leq  \frac{1}{8}\norm{(f-\psi)_+}^2_{H^s} + \norm{\psi}_{C^1_v}^2 \int \chi_{\{ f - \psi > 0 \}} (v) \dd v,
\eals 
where we used Theorem \ref{thm:4.1} and the upper bound \eqref{eq:upperbound}. For the skew-symmetric part we have
\bals
	-\mathcal E^{\textrm{skew}}&(\psi, (f - \psi)_+) \\
	&:=  \frac{1}{2}\int_{\R^d} \int_{\R^d} \big[(f - \psi)_+(v) + (f - \psi)_+(w)\big]\big[\psi(w)-\psi(v)\big] K(v, w) \dd w\dd v\\
	&= \frac{1}{4}\int\int\big[\psi(w) - \psi(v)\big]\big[(f - \psi)_+(v) + (f - \psi)_+(w)\big]\big[K(v, w) - K(w, v)\big]\dd w \dd v\\
	&=\frac{1}{2}\int\int\big[\psi(w) - \psi(v)\big](f - \psi)_+(v)\big[K(v, w) - K(w, v)\big] \dd w \dd v\\
	&=\frac{1}{2}\int(f - \psi)_+(v)\Bigg(\int\big[\psi(w) - \psi(v)\big]\big[K(v, w) - K(w, v)\big] \dd w \Bigg)\dd v.
\eals
We estimate
\bals
	\textrm{PV} \int\big[\psi(w) - \psi(v)\big]&\big[K(v, w) - K(w, v)\big] \dd w = \int_{B_r(v)} \dots + \int_{\R^d\setminus B_r(v)}\dots\\
	&\leq \int_{B_r(v)}(w - v)\cdot\nabla_v \psi(v)\big[K(v, w) - K(w, v)\big] \dd w \\
	&\quad+ \int_{B_r(v)}\abs{w - v}^2\norm{D_v^2 \psi}_{L^\infty}\abs{K(v, w) - K(w, v)} \dd w + \Lambda r^{-2s} \norm{\psi}_{L^\infty}\\
	&\leq \Lambda r^{1-2s} \norm{\psi}_{C^1_v} + \Lambda r^{2-2s} \norm{\psi}_{C^2_v} + \Lambda r^{-2s} \norm{\psi}_{L^\infty}.
\eals	
For the tail we used the upper bound \eqref{eq:upperbound}. For the near part we used \eqref{eq:upperbound-2} to bound 
\bals
	 \int_{B_r(v)}\abs{v - w}^2\Abs{K(v, w) - K(w, v)} \dd w \lesssim \Lambda r^{2-2s}.
\eals
Moreover, in case that $s \in (0, 1/2)$ we again get by \eqref{eq:upperbound-2}
\beqs
	\int_{B_r(v)} \abs{v-w} \Abs{K(v, w) - K(w, v)} \dd w  \lesssim \Lambda r^{1-2s},
\eeqs
and in case that $s \in [1/2, 1)$ we use the cancellation assumptions \eqref{eq:cancellation2} 
\bals
	\Bigg\vert\textrm{PV}\int_{B_r(v)}(v - w)\big(K(v, w) - K(w, v)\big) \dd w\Bigg\vert \lesssim \Lambda r^{1-2s}.
\eals
Thus we find
\bals
	-\mathcal E^{\textrm{skew}}(\psi, (f - \psi)_+) &\leq C  \Big( r^{1-2s} \norm{\psi}_{C^1_v} +  r^{2-2s} \norm{\psi}_{C^2_v} +  r^{-2s} \norm{\psi}_{L^\infty}\Big)\int (f - \psi)_+(v)\dd v \\
	&\leq C \norm{\psi}_{C^2_v} \int (f - \psi)_+(v)\dd v,
\eals
so that
\bals
	-\mathcal E(\psi, (f - \psi)_+) \leq \frac{1}{8}\norm{(f-\psi)_+}^2_{H^s} + \norm{\psi}_{C^1_v}^2 \int \chi_{\{ f - \psi > 0 \}} (v) \dd v + C \norm{\psi}_{C^2_v} \int (f - \psi)_+(v)\dd v.
\eals
Gathering the estimates we obtained so far we get for any $t \in (-r^{2s} +t_0, t_0]$
\bals
	\int_{\R^{2d}} &\mathcal  (f - \psi)_+^2(t, x, v) \dd x\dd v +\lambda \int_{-r^2s+t_0}^{t_0}\int_{\R^d} \norm{(f - \psi)_+}_{H^s_v}^2\dd x\dd t \\
	&\leq \int_{\R^{2d}} \mathcal  (f - \psi)_+^2(r^{2s}+t_0, x, v) \dd x\dd v+  C \norm{\psi}_{C_v^2} \int_{-r^2s+t_0}^{t_0}\int_{\R^{2d}} (f - \psi)_+(v)\dd z \\
	&\quad+ \norm{\psi}_{C_v^1}^2\Abs{\{f - \psi > 0\} \cap (-r^{2s} +t_0, t_0]} + C\frac{v_0 + R}{(R-r)^{1+2s}} \int_{-r^2s+t_0}^{t_0}\int_{\R^{2+}} (f - \psi)_+(v)\dd z\\
	&\quad+ C\Lambda\int_{-r^2s+t_0}^{t_0}\int_{\R^{2d}} (f - \psi)^2_+(v)\dd z + \int_{Q_R(z_0)} h^2\dd z.
\eals
Since $0 \leq (f-\psi)_+ \leq 1$ the $L^1$ norm and $L^2$ norm are bounded by $\Abs{\{f - \psi > 0\} \cap (-r^{2s} +t_0, t_0]} \leq \Abs{\{f > 0\} \cap Q_R(z_0)}$. Moreover, we chose $\psi$ such that $\norm{D_v\psi}_{L^\infty} \lesssim \frac{1}{R-r}$, $\norm{D^2_v\psi}_{L^\infty} \lesssim \frac{1}{(R-r)^2}$ and $\norm{D_x \psi}_{L^\infty} \lesssim \frac{1}{(R-r)^{1+2s}}$.
Finally we note that $(f-\psi)_+ = f$ in $Q_r(z_0)$, $(f-\psi)_+ \leq f$ everywhere and $(f-\psi)_+ = 0$ outside $Q_R(z_0)$ and conclude.
\end{proof}

\begin{lemma}[Local gain of integrability]\label{prop:11}
Let $f$ be a non-negative sub-solution of \eqref{eq:1.1}-\eqref{eq:1.2} in $Q_R(z_0)$ with a non-negative kernel satisfying \eqref{eq:coercivity}-\eqref{eq:cancellation2} for some $\bar R \geq 1$. Let $Q_r(z_0)\subset Q_R(z_0)\subset Q_{\frac{\bar R}{2}}$ for some $0 < r < R \leq 1$.  Assume $0 \leq f \leq 1$ a.e. in $\big(-R^{2s}+t_0, t_0\big] \times B_{R^{1+2s}}(x_0 + t_0v_0)\times \R^d$. Given $p \in \Big[2, 2+ \frac{2s}{d(1+s)}\Big)$, $f$ satisfies
\bal
	\norm{f}^2_{L^p(Q_r(z_0))} \lesssim_{\lambda, \Lambda, v_0} &(R-r)^{-2}\int_{Q_R^t(z_0)} f^2_0(x, v) \dd x\dd v \\
	&+ (R-r)^{-4}\Abs{\{f > 0\} \cap Q_R(z_0)} + (R-r)^{-2}\int_{Q_R(z_0)} h^2\dd z,
\label{eq:2.5}
\eal
and
\bal
	\norm{f}^2_{L^1_{t, v}W^{\sigma, 1}_x(Q_r(z_0))} \lesssim_{\lambda, \Lambda, v_0} &(R-r)^{-2}\int_{Q_R^t(z_0)} f^2_0(x, v) \dd x\dd v\\
	&+(R-r)^{-4}\Abs{\{f > 0\} \cap Q_R(z_0)} +(R-r)^{-2}\int_{Q_R(z_0)} h^2\dd z,
\label{eq:2.6}
\eal
where $z_0 = (t_0, x_0, v_0)$, $f_0(x, v) = f\big(-r^{2s}+t_0, x, v\big)$ and $Q^\tau_r(z_0) = \left\{(x, v)\in \R^{2d} : (\tau, x, v) \in Q_r(z_0)\right\}$.
\end{lemma}
\begin{proof}
We first recall a clever observation made by Imbert and Silvestre \cite{IS}. If $f(t, x, v) = 0$ then $\mathcal Tf = 0$ a.e. and $\mathcal Lf(v) \geq 0$. In particular any sub-solution $f$ of \eqref{eq:1.1} satisfies
\beqs
	\mathcal T f \leq \mathcal L f - \Bigg(\int_{\R^d} f(w) K(v, w) \dd w\Bigg) \chi_{\{f = 0\}} + h = \Big(\mathcal L f\Big)\chi_{\{f > 0\}} + h.
\eeqs
Now we write $\mathcal I := (-{R}^{2s} + t_0, t_0]$ and for $t \in \mathcal I$ we let $\varphi_1: \R^{2d} \to [0, 1]$ be a smooth function such that $\varphi_1 = 1$ on $Q_r^{t}(z_0):= B_{r^{1+2s}}(x_0+(t-t_0)v_0)\times B_r(v_0)$ and $\varphi_1 = 0$ outside $Q_{\frac{R+r}{4}}^t(z_0)$. Consider $g := f\varphi_1$. There holds
\beqs
	\partial_t g + v\cdot \nabla_x g + (-\Delta)_v^s g =  \big(\mathcal T f\big) \varphi_1 + f \big(\mathcal T \varphi_1\big) + (-\Delta)_v^s g \leq\big(\mathcal L f\big)\chi_{\{f>0\}}\varphi_1 + h\varphi_1 + f \big(\mathcal T \varphi_1\big) + (-\Delta)_v^s g.
\eeqs
Note that we have by the energy estimate \ref{prop:9}
\bals
	\norm{f}^2_{L^2_{t, x}\dot H^s_v(Q_{\frac{R+r}{4}}(z_0))} \leq CN,
\eals	
where $$N = \int_{Q_R^t(z_0)}f_0^2(x, v)\dd x \dd v  + \mathcal C\Bigg(\frac{R+r}{4}, R\Bigg) \Abs{\{f > 0\} \cap Q_R(z_0)} + \int_{Q_R(z_0)} h^2\dd z$$ with $\mathcal C(r, R) = (R-r)^{-2}$. 
Moreover, since $f \in H^s_v$ with support in $Q_R(z_0)$ we know by Theorem \ref{thm:4.1} that $\mathcal L f \in H^{-s}_v(\R^d)$ and
\bals
	\int_{B_{\frac{R+r}{4}(v_0)}} \big(\mathcal L f\big)\chi_{\{f>0\}}\varphi_1(v) \dd v &= \int_{B_{\frac{R+r}{4}(v_0)}}\int_{B_{\frac{r}{4}(v)}} \dots \dd w \dd v + \int_{B_{\frac{R+r}{4}(v_0)}}\int_{\R^d \setminus B_{\frac{r}{4}(v)}} \dots\dd w \dd v\\
	&\leq C\mathcal C^{1/2}\Bigg(r, \frac{R+r}{4}\Bigg)\norm{f}_{H^s_v(B_{\frac{R+r}{2}(v_0)})}\\
	&\quad + C \Lambda \Bigg(\frac{r}{4} + v_0\Bigg)^{-2s}\int_{B_{\frac{R+r}{4}(v_0)}}\chi_{\{f>0\}}\varphi_1(v)\dd v.
\eals
Further we have
\bals
	\norm{g}^2_{L^2_{t, x}\dot H^s_v(\mathcal I\times\R^{2d})} &\leq C\norm{f}^2_{L^2_{t, x}\dot H^s_v(Q_{\frac{R+r}{4}}(z_0))} + C \Big(\frac{r}{4} + v_0\Big)^{-2s}\norm{g}_{L^2}^2\\
	&\leq C\norm{f}^2_{L^2_{t, x}\dot H^s_v(Q_{\frac{R+r}{4}}(z_0))} + C \Big(\frac{r}{4} + v_0\Big)^{-2s}\norm{f}^2_{L^2(Q_R(z_0)}.
\eals	
Therefore $g$ satisfies 
\bals
	\partial_t g + v\cdot \nabla_x g + (-\Delta)_v^s g &\leq \big(\mathcal L f\big)\chi_{\{f>0\}}\varphi_1 + h\varphi_1 + f \big(\mathcal T \varphi_1\big) + (-\Delta)_v^{s}g =: H,
\eals
with right hand side $H = \big(\mathcal L f\big)\chi_{\{f>0\}}\varphi_1+ h\varphi_1+ f \big(\mathcal T \varphi_1\big) + (-\Delta)_v^{s} g \in L^2_{t, x}H^{-s}_v(\mathcal I\times \R^{2d})$. 
The energy estimate \ref{prop:9} implies
\bals
	\norm{H}^2_{L^2_{t, x}H^{-s}_v(\mathcal I\times \R^{2d})} &\leq C \mathcal C\Big(r, \frac{R+r}{4}\Big)\norm{f}^2_{L^2_{t, x}(\dot H^s_v(Q_{\frac{R+r}{4}}(z_0)))}+ C\Big(\frac{r}{4} + v_0\Big)^{-4s}\Abs{\{f > 0\} \cap Q_R(z_0)} \\
	&\quad+ \norm{h}^2_{L^2(Q_R(z_0))}+ \norm{g}^2_{L^2_{t, x}(\dot H^s_v(\mathcal I\times\R^{2d}))} + \frac{(\abs{v_0} + 1)^2}{(R- (\frac{R+r}{4}))^{4s}}\norm{f}^2_{L^2(Q_R(z_0))}\\
	&\lesssim_{\lambda,\Lambda, v_0} \mathcal C\Big(r, \frac{R+r}{4}\Big)N + \mathcal C^2\Big(r,  \frac{R+r}{4}\Big)\norm{f}^2_{L^2(Q_R(z_0))}.
\eals

It remains to quantify the measure of the difference between solution and sub-solution. 
Note that for $(x, v)$ outside $Q_{\frac{R+r}{4}}^{t}(z_0)$ there holds $g = 0$ and thus $\mathcal T g = 0$. Moreover, outside $Q_{\frac{R+r}{4}}^{t}(z_0)$
\bals
	\Abs{(-\Delta)_v^s g(t, x, v)} &= C \Bigg\lvert{\int_{B_{r+\frac{R-r}{2}}(v_0)}} g(t, x, w)\abs{w - v}^{-(d+2s)}\dd w\Bigg\rvert \\
	&\leq C \Bigg(\abs{v} - v_0 - \frac{R+r}{2}\Bigg)^{-d - 2s} \abs{B_{\frac{R+r}{4}}(v_0)} \chi_{\Big\{\abs{x} \leq \big(\frac{R+r}{4}\big)^{1+2s} + x_0 + (t-t_0)v_0\Big\}}.
\eals
Thus $\mathcal T g + (-\Delta)_v^s g =  (-\Delta)_v^s g \in L^2$ outside $B_{(\frac{R+r}{4})^{1+2s}}(x_0 + (t-t_0)v_0) \times B_{\frac{R+r}{2}}(v_0)$. 
Consider now a second smooth cut-off function $0 \leq \varphi_2\leq 1$ such that $\varphi_2 = 1$ in $B_{(\frac{R+r}{4})^{1+2s}}(x_0+(t-t_0)v_0)\times B_{\frac{R+r}{2}}(v_0)$ and $ \varphi_2 = 0$ outside $Q_{\frac{3(R+r)}{4}}^{t}(z_0)$. 
We let 
\beqs
	\tilde H = H \varphi_2 + (1- \varphi_2)(-\Delta)_v^s g.
\eeqs
We claim $\tilde H \in L^2_{t, x}H^{-s}_v(\mathcal I\times \R^{2d})$. Indeed $(1- \varphi_2)(-\Delta)_v^s g \in L^2_{t, x}H^{-s}_v$. Moreover, if a function $a\in H^s$ then also $a \varphi_2 \in H^s$. Thus by duality it also follows that $H\varphi_2 \in L^2_{t, x}H^{-s}_v$. 
With this definition there holds in $\mathcal I \times \R^{2d}$ 
\beqs
	\mathcal T g + (-\Delta)_v^s g \leq \tilde H
\eeqs
with equality outside $Q_{\frac{3(R+r)}{4}}^{t}(z_0)$.

Let $m$ be a non-negative measure whose support is contained in $\mathcal I \times Q_{\frac{3(R+r)}{4}}^{t}(z_0)$ given by 
\beqs
	m := \tilde H - \mathcal T g - (-\Delta)_v^sg.
\eeqs
To control this measure, we test it against a test function $\varphi_3$ that is equal to $1$ in the support of $m$ such that $\varphi_3$ vanishes outside $Q_R^{t}(z_0)$. We get
\bals
	\norm{m}_{M^1(\mathcal I\times \R^{2d})} &= \int_{\mathcal I\times \R^{2d}} \varphi_3 \dd m = \int_{\mathcal I\times \R^{2d}} \varphi_3 \big(\tilde H - \mathcal T g - (-\Delta)_v^sg\big)\dd z\\
	&= \int_{\R^{2d}}\Big(g(t_0, x, v) - g\big(t_0-R^{2s}, x, v\big)\Big)\dd x \dd v \\
	&\qquad+ \int_{\mathcal I\times \R^{2d}} \Big( v\cdot \nabla_x \varphi_3  - (-\Delta)_v^s \varphi_3\Big) g + \tilde H\varphi_3 \dd z\\
	&\lesssim \mathcal C\Bigg(\frac{3(r+R)}{4}, R\Bigg)\norm{g}_{L^1(\mathcal I\times \R^{2d})} + \norm{\tilde H}_{L^2_{t, x}H^{-s}_v}\\
	&\lesssim  \mathcal C\Bigg(\frac{3(r+R)}{4}, R\Bigg) \norm{f}_{L^1(Q_R(z_0))} + \norm{H}_{L^2_{t, x}H^{-s}_v(\mathcal I\times \R^{2d})}.
\eals

Thus $g$ solves the fractional Kolmogorov equation with right hand side
\beqs
	\mathcal T g + (-\Delta)_v^sg = \tilde H - m.
\eeqs
Using Lemma \ref{prop:2.2} we deduce
\bals
	\norm{g}^2_{L^p(\mathcal I\times \R^{2d})} \lesssim \norm{\tilde H}^2_{L^2_{t,x}H^{-s}_v(\mathcal I\times \R^{2d})}\leq \norm{H}^2_{L^2_{t,x}H^{-s}_v(\mathcal I\times \R^{2d})}.
\eals
and
\bals
	\norm{g}^2_{L^1_{t, v}W_x^{\sigma, 1}(\mathcal I\times \R^{2d})} &\lesssim \norm{\tilde H}^2_{L^2_{t,x}H^{-s}_v(\mathcal I\times \R^{2d})} + \norm{m}^2_{M^1(\mathcal I\times \R^{2d})}\\
	&\lesssim \mathcal C^2\Bigg(\frac{3(r+R)}{4}, R\Bigg) \norm{f}^2_{L^1(Q_R(z_0))} + \norm{H}^2_{L^2_{t, x}H^{-s}_v(\mathcal I\times \R^{2d})}.
\eals	
We conclude by noting that $0 \leq f \leq 1$ a.e. so that the $L^1$ and the $L^2$ norm are bounded by $\Abs{\{f > 0\} \cap Q_R(z_0)}$.
\end{proof}

\section{First Lemma of De Giorgi}
\label{subsec:firstDG}
\begin{lemma}\label{lem:6.6}
Let $Q_r(z_0) \subset Q_R(z_0)$ with $0 < r < R \leq 1$ and $t_0 < r^{2s}$. Let $f$ be a non-negative sub-solution of \eqref{eq:1.1}-\eqref{eq:1.2} in $Q_R(z_0)$ with a non-negative kernel $K$ satisfying \eqref{eq:coercivity}-\eqref{eq:cancellation2} for $\bar R \geq 1$. Assume $0 \leq f \leq 1$ a.e. in $(-R^{2s} + t_0, t_0] \times B_{R^{1+2s}}(x_0 + (t-t_0)v_0)\times \R^d$. Then there exists $\varepsilon_0 > 0$ so that for all $\varepsilon < \varepsilon_0$ there holds: if
\beqs
	\int_{Q_R(z_0)} f^2 \dd z\leq \varepsilon,
\eeqs
then for any $2 < p < 2+\frac{2s}{d(1+s)}$ there exists a large constant $C > 1$ depending only on $s, d$ and $p$ such that 
\beqs
	f \leq \min\Bigg\{C\Big[1 + (R-r)\norm{h}_{L^\infty(Q_R(z_0))}\Big]^{\frac{\beta_1}{2}}(R-r)^{-\beta_1}\varepsilon^{\beta_2}, \frac{1}{2}\Bigg\}\quad a.e. \textrm{ in } Q_r(z_0),
\eeqs
where $\beta_1 =  \frac{2p}{2p-2} > 0$ and $\beta_2 = \frac{p-2}{2(2p-2)} > 0$.
\end{lemma}
\begin{proof}[Proof of Lemma \ref{lem:6.6}]
We do a De Giorgi iteration, as was done in \cite[Lemma 6.6]{IS} and \cite[Lemma 3.8]{JG}. For an intuitive description of this technique, we refer the reader to \cite{cafvas}. Let $\tilde \tau := r^{2s} - t_0$ and $\hat \tau := R^{2s} - t_0$. Then $0 < \tilde \tau < \hat \tau$. Let $L > 0$ to be determined and consider 
\bals
	&l_k = L(1 - 2^{-k}),\\
	&r_k = r + 2^{-k}(R - r),\\
	&t_k = -\tilde \tau - 2^{-k}(\hat \tau - \tilde \tau).
\eals
Define $Q_k(z_0) = (t_k, t_0] \times B_{r_k^{1 + 2s}}(x_0 + (t_k-t_0)v_0) \times B_{r_k}(v_0)$ and
\beqs
	A_k := \int_{Q_k(z_0)} (f - l_k)^2_+ \dd z.
\eeqs
Then $A_0 \leq \varepsilon_0$ by assumption. We want to prove that $A_k \to 0$ as $k \to +\infty$. This then proves the lemma.

We observe that $(f - l_{k+1})_+$ is a sub-solution of \eqref{eq:1.1}-\eqref{eq:1.2} with source term 
\beqs
	\tilde h(t, x, v) = \chi_{\{f(v) > l_{k+1}\}}h(t, x, v),
\eeqs
and such that $0 \leq (f - l_{k+1})_+ \leq 1$ a.e. in $(-R^{2s} + t_0, t_0] \times B_{R^{1+2s}}(x_0 + (t-t_0)v_0)\times \R^d$. 
With Chebyschev's inequality we can estimate
\bal
	\Abs{\{f > l_{k+1}\} \cap Q_{k}(z_0)} = \Abs{\big\{(f - l_k)_+> 2^{-k-2}L\big\} \cap Q_{k}(z_0)} \leq 2^{2k+4}L^{-2} A_k.
\label{eq:chebyschevaux}
\eal
Thus
\beqs
	\norm{\tilde h}^2_{L^2(Q_k(z_0))} \leq C\norm{h}^2_{L^\infty(Q_R(z_0))} 2^{2k+4}L^{-2}A_k.
\eeqs
We now pick $t_{k+\frac{1}{2}} \in [t_k, t_{k+1}]$ such that
\bals
	\int_{Q_k^{t_{k+\frac{1}{2}}}(z_0)} (f-l_{k})_+^2 \dd x\dd v \leq \frac{1}{t_{k+1}-t_k}\int_{Q_k(z_0)} (f-l_{k})_+^2\dd z \leq C2^kA_k.
\eals
Then we apply Lemma \ref{prop:11} and \eqref{eq:chebyschevaux} to find
\bal\label{eq:DG_aux}
	\Bigg(\int_{t_{k+\frac{1}{2}+t_0}}^{t_0}&\int_{Q^t_{k+1}(z_0)}(f - l_{k+1})_+^p \dd z\Bigg)^{\frac{2}{p}} \\
	&\leq \frac{2^{4k+4}}{(R-r)^{4}}\Abs{\{f > l_{k+1}\} \cap Q_{k}} +C \frac{2^k}{(R-r)^2} A_k+ C\norm{h}^2_{L^\infty(Q_R(z_0))} \frac{2^{2k+4}}{(R-r)^{2}L^2}A_k\\
	&\leq \frac{2^{6k+8}}{(R-r)^{4}L^{2}}A_k +C \frac{2^k}{(R-r)^2} A_k + C\norm{h}^2_{L^\infty(Q_R(z_0))} \frac{2^{2k+4}}{(R-r)^{2}L^2}A_k\\
	&\leq C2^{6k}(R-r)^{-4}L^{-2}\left[1 + (R-r)^2\norm{h}^2_{L^\infty(Q_R(z_0))}\right]A_k,
\eal
where we used $l_{k+1} \geq l_k$. 

Using $Q_{k+1} \subset Q_k$ and Chebyschev's inequality we get
\bals
	A_{k+1} &\leq \Bigg(\int_{Q_{k+1}(z_0)} (f - l_{k+1})^p_+ \dd z\Bigg)^{\frac{2}{p}} \abs{\{f > l_{k+1}\}\cap Q_{k+1}(z_0)}^{1-\frac{2}{p}}  \\
	&\lesssim_{s, v_0} 2^{6k} (R-r)^{-4}L^{-2}\left[1 + (R-r)^2\norm{h}^2_{L^\infty(Q_R(z_0))}\right]A_k\Abs{\big\{f > l_{k+1}\big\}\cap Q_{k}(z_0)}^{1-\frac{2}{p}}  \\
	&\lesssim_{s, v_0}2^{6k} (R-r)^{-4}L^{-2}\left[1 + (R-r)^2\norm{h}^2_{L^\infty(Q_R(z_0))}\right]A_k\Big(2^{2k + 4} L^{-2}A_k\Big)^{1-\frac{2}{p}}\\
	&\lesssim_{s, v_0} 2^{8k}(R-r)^{-4} \left[1 + (R-r)^2\norm{h}^2_{L^\infty(Q_R(z_0))}\right]L^{-2(1+\frac{p-2}{p})}A_k^{1 + \frac{p-2}{p}}.
\eals
To determine how large we can pick $L$, we consider the sequence $A_k^* := A_0 Q^{-k}$ for some $Q > 1$ to be determined and for $k \geq 0$. Then we require $A_k^*$ to satisfy the reverse non-linear recurrence, that is
\beq
	A_{k+1}^* \geq2^{8k}(R-r)^{-4} \left[1 + (R-r)^2\norm{h}^2_{L^\infty(Q_R(z_0))}\right]L^{-2(1+\frac{p-2}{p})}A_k^{1 + \frac{p-2}{p}}.
\label{eq:reverse}
\eeq
Equivalently,
\beqs
	1 \geq  2^{8k}(R-r)^{-4} \left[1 + (R-r)^2\norm{h}^2_{L^\infty(Q_R(z_0))}\right]L^{-2(1+\frac{p-2}{p})}A_0^{\frac{p-2}{p}} Q^{-k(\frac{p-2}{p})}Q.
\eeqs
Choose $Q$ such that $2^8Q^{-\frac{p-2}{p}} \leq 1$ or 
\beqs
	2^{8} \leq Q^{\frac{p-2}{p}}.
\eeqs
Then \eqref{eq:reverse} holds if 
\beqs
	L^{2(1+\frac{p-2}{p})}\geq (R-r)^{-4}\left[1 + (R-r)^2\norm{h}^2_{L^\infty(Q_R(z_0))}\right]A_0^{\frac{p-2}{p}} Q.
\eeqs
Thus, if we choose
\beqs
	L=   A_0^{\frac{p-2}{2(2p-2)}} (R-r)^{-\frac{2p}{2p-2}}2^{\frac{4p^2}{(p-2)(2p-2)}}\left[1 + (R-r)\norm{h}_{L^\infty(Q_R(z_0))}\right]^{\frac{p}{2p-2}},
\eeqs
then, since $\frac{p-2}{p} > 0$ and $A_0 \leq \varepsilon$, we deduce $A_k \to 0$ as $k \to +\infty$.
In particular, for almost every $z \in Q_r(z_0)$
\beqs
	f \leq L.
\eeqs
And therefore for almost every $z \in Q_r(z_0)$ there holds
\bals
	f &\leq   A_0^{\frac{p-2}{2(2p-2)}} (R-r)^{-\frac{2p}{2p-2}}2^{\frac{4p^2}{(p-2)(2p-2)}}\left[1 + (R-r)\norm{h}_{L^\infty(Q_R(z_0))}\right]^{\frac{p}{2p-2}}\\
	&=  \norm{f}_{L^2(Q_R(z_0))}^{\frac{p-2}{2p-2}} (R-r)^{-\frac{2p}{2p-2}}2^{\frac{4p^2}{(p-2)(2p-2)}}\left[1 + (R-r)\norm{h}_{L^\infty(Q_R(z_0))}\right]^{\frac{p}{2p-2}}.
\eals
\end{proof}
\begin{remark}
First we remark that if instead of $f \leq 1$ a.e. we just assume $f$ to be essentially bounded, then the gain of integrability in Lemma \ref{prop:11} would yield
\bals
	\norm{f}^2_{L^p(Q_r(z_0))} &\lesssim_{\lambda, \Lambda, v_0}(R-r)^{-2}\int_{Q_R^t(z_0)} f^2_0(x, v) \dd x\dd v \\
	&\qquad+ (R-r)^{-4}\left(\sup_{z \in \Omega_{t, x}(R, z_0)\times \R^d} f\right)^2 \Abs{\{f > 0\} \cap Q_R(z_0)} + (R-r)^{-2}\int_{Q_R(z_0)} h^2\dd z,
\eals
where $\Omega_{t, x}(R, z_0) := (-R^{2s} +t_0, t_0]\times B_{R^{1+2s}}(z_0)$.
Using this we can further note that the nonlinear relation between the $L^\infty$ and the $L^2$ norm of the function can be linearised at the price of a non-local tail term. Indeed, going back to \eqref{eq:DG_aux}, if for some $\delta \in (0, 1]$ we choose $L$ such that 
\beqs
	L^2 \geq \delta\left(\sup_{z \in \Omega_{t, x}(R, z_0)\times \R^d} f + (R-r)\norm{h}_{L^\infty(Q_R(z_0))}\right)^2,
\eeqs
then
\bals
	\Bigg(\int_{t_{k+\frac{1}{2}+t_0}}^{t_0}&\int_{Q^t_{k+1}(z_0)}(f - l_{k+1})_+^p \dd z\Bigg)^{\frac{2}{p}} \leq C 2^{6k} (R-r)^{-4}\delta^{-1}A_k.
\eals
Thus we deduce
\bals
	A_{k+1} &\leq \Bigg(\int_{Q_{k+1}(z_0)} (f - l_{k+1})^p_+ \dd z\Bigg)^{\frac{2}{p}} \abs{\{f > l_{k+1}\}\cap Q_{k+1}(z_0)}^{1-\frac{2}{p}}  \\
	&\lesssim_{s, v_0}2^{6k} (R-r)^{-4}\delta^{-1}A_k\big(2^{2k + 4} L^{-2}A_k\big)^{1-\frac{2}{p}}\\
	&\lesssim_{s, v_0} 2^{8k}(R-r)^{-4} \delta^{-1}L^{-\frac{2(p-2)}{p}}A_k^{1 + \frac{p-2}{p}}.
\eals
The same barrier argument as before shows that we have to choose $L$ such that
\beqs
	L^{2(\frac{p-2}{p})}\geq (R-r)^{-4}\delta^{-1}A_0^{\frac{p-2}{p}} Q.
\eeqs
Therefore, we pick
\beqs
	L = \delta^{\frac{1}{2}}\Bigg(\sup_{z \in \Omega_{t, x}(R, z_0)\times \R^d} f + (R-r)\norm{h}_{L^\infty(Q_R(z_0))}\Bigg) +  A_0^{\frac{1}{2}} (R-r)^{-\frac{2p}{p-2}}2^{\frac{4p^2}{(p-2)^2}}\delta^{-\frac{p}{2(p-2)}},
\eeqs
and we deduce that
for almost every $z \in Q_r(z_0)$ there holds
\bal\label{eq:aux-proper-SH}
	f \leq C(R-r)^{-\frac{2p}{p-2}}\delta^{-\frac{p}{2(p-2)}}\norm{f}_{L^2(Q_R(z_0))}+\delta^{\frac{1}{2}}\sup_{z \in \Omega_{t, x}(R, z_0)\times \R^d} f + \delta^{\frac{1}{2}}(R-r)\norm{h}_{L^\infty(Q_R(z_0))}.
\eal
To derive a linear relation between the $L^\infty$ and the $L^2$ norm of $f$, we would need to absorb the non-local tail term on the left hand side. Thus we require an estimate of the tail in terms of local quantities. If we were able to replace the tail by the essential supremum on any finite ball in the velocity domain, then using a standard covering argument, see e.g. \cite[Lemma 8.18]{GiaquintaMartinazzi}, we would have obtained 
for almost every $z \in Q_r(z_0)$ 
\bals
	f \leq  C(R-r)^{-\frac{2p}{p-2}}\norm{f}_{L^2(Q_R(z_0))} +C\norm{h}_{L^\infty(Q_R(z_0))}.
\eals
Combined with the Weak Harnack inequality \eqref{eq:weakH}, this would allow us to derive the (proper) Strong Harnack inequality, in the sense that
\beqs
	\sup_{\tilde Q^-_{\frac{r_0}{4}}(z_0)} f \leq C \inf_{Q_{\frac{r_0}{4}}(z_0)} f +C\norm{h}_{L^\infty(Q_R(z_0))}.
\eeqs
We refer to Section \ref{subsec:harnack} below. Note, however, that it is not clear that the non-local tail in \eqref{eq:aux-proper-SH} can be bounded by local quantities.
\end{remark}

\section{Second Lemma of De Giorgi}
\label{subsec:secondDG}
\subsection{Weak Poincaré} This is where we move along trajectories in order to obtain a hypoelliptic Poincaré-type inequality with an error term. The idea comes from Guerand and Mouhot \cite{JGCM}. This inequality allows us to pick up an intermediate set in De Giorgi's Second Lemma, Theorem \ref{thm:IVL}. The proof of the inequality \eqref{eq:3.1} does not use De Giorgi's First Lemma. 
\begin{proposition}\label{prop:13}
Let $f$ be a non-negative sub-solution to \eqref{eq:1.1}-\eqref{eq:1.2} in $(-9, 0] \times Q^t_3$ with a non-negative kernel $K$ satisfying \eqref{eq:coercivity}-\eqref{eq:cancellation2} for $\bar R = 6$. Given any $\varepsilon \in (0, 1)$ and $0 <\sigma < \frac{s}{1+2s}$ there holds 
\bal
	\big\Vert\big(f - \langle f\rangle_{Q^-_1}\big)_+\big\Vert_{L^1(Q_1)} &\lesssim\frac{1}{\varepsilon^{d+2}} \int_{(-3^{2s}, 0]\times B_{3^{1+2s}} \times \R^d}\Bigg(\int_{\R^d} \abs{f(v) - f(w)}^2 K(v, w)\dd w\Bigg)^{\frac{1}{2}}\dd z  \\
	&\quad +\frac{1}{\varepsilon^{d}} \abs{ \textrm{PV}\int_{(-3^{2s}, 0]\times B_{3^{1+2s}} \times \R^d}\int_{\R^d} f(v)\big(K(v, w) - K(w, v)\big)\dd w \dd z}\\
	&\quad  + \varepsilon^\sigma \norm{f}_{L^1_{t, v}W^{\sigma, 1}_x(Q_2)} + \norm{h}_{L^1(Q_3)},
\label{eq:3.1}
\eal
where $Q^-_1:= Q_1(-2, 0, 0) = (-3, -2] \times Q_1^t$ and $\langle f\rangle_{Q^-_1} = \fint_{Q^-_1} f = \frac{1}{\abs{Q^-_1}}\int_{Q^-_1} f$.
\end{proposition}
\begin{remark}
Note that the right hand side is finite for a solution $f \in L^2\big((-9, 0] \times B_{3^{1+2s}}, L^1_{v, \textrm{loc}}\cap H^s_{v, \textrm{loc}}(\R^d)\big)$ due to Lemma \ref{prop:9} and Lemma \ref{prop:11}. In particular, if $f$ is assumed to be essentially bounded then the right hand side is finite. Indeed, the symmetric non-local term is bounded by the $H^s$ norm of $f$, which is bounded by Lemma \ref{prop:9}, the anti-symmetric part appearing on the right hand side can be bounded by the $L^1$ norm of $f$ due to the cancellation \eqref{eq:cancellation1}, and finally, Lemma \ref{prop:11} bounds the term $\norm{f}_{L^1_{t, v}W^{\sigma, 1}_x(Q_2)}$ by the right hand side of the energy estimate, Lemma \ref{prop:9}. 
\end{remark}

\begin{proof}
Consider for $\varepsilon \in (0, 1)$ a small function $\varphi_\varepsilon = \varphi_\varepsilon(y, w)$, $0 \leq \varphi_\varepsilon \leq 1$ with support in $B_{1^{1+2s}} \times B_1$, so that $\varphi_\varepsilon = 1$ in $B_{(1-\varepsilon)^{1+2s}} \times B_{(1-\varepsilon)}$ and with $\abs{\nabla_y \varphi_\varepsilon}\lesssim \frac{1}{\varepsilon}$,$\abs{\nabla_w \varphi_\varepsilon}\lesssim \frac{1}{\varepsilon}$. We then split
\bals
	\norm{(f - \langle f\rangle_{Q^-_1})_+}_{L^1(Q_1)} &\lesssim \norm{(f - \langle f \varphi_\varepsilon \rangle_{Q^-_1})_+}_{L^1(Q_1)} \\
	&\lesssim \int_{Q_1} \Bigg\{\fint_{Q^-_1} \big[f(t, x, v) - f(s, y, w)\big]\varphi_\varepsilon(y, w)\dd s\dd y\dd w\Bigg\}_+\dd t\dd x\dd v \\
	&\quad + \norm{f}_{L^1(Q_1)}\fint_{Q^-_1} \big(1 - \varphi_\varepsilon(y, w)\big)\dd s\dd y\dd w \\
	&\lesssim  \int_{Q_1} \Bigg\{\fint_{Q^-_1} \big[f(t, x, v) - f(s, y, w)\big]\varphi_\varepsilon(y, w)\dd s\dd y\dd w\Bigg\}_+\dd t\dd x\dd v  \\
	&\quad+ \varepsilon^{d(1+s)}\norm{f}_{L^2(Q_1)},
\eals
where we used $ \langle f \varphi_\varepsilon \rangle_{Q^-_1} \leq \langle f \rangle_{Q^-_1}$ and Cauchy-Schwarz inequality.

Consider the first term. For fixed $t, x, v$ we decompose the trajectory $(t, x, v) \rightarrow (s, y, w)$ into four sub-trajectories in $Q_3$ as follows:
\bals
	(t, x, v) \xrightarrow{\nabla_x} (t, x + \varepsilon w, v) \xrightarrow{\nabla_v^s} \Big(t, x+\varepsilon w, \frac{x+\varepsilon w - y}{t -s}\Big) \xrightarrow{\mathcal T} \Big(s, y,\frac{x+\varepsilon w - y}{t -s}\Big) \xrightarrow{\nabla_v^s} (s, y, w).
\eals
We split the integrand along these trajectories
\bals
	f(t, x, v) - f(s, y, w)&=  \left[f(t, x, v) - f(t, x + \varepsilon w, v)\right]\\
	&\quad+ \left[f(t, x + \varepsilon w, v) - f\left(t, x + \varepsilon w, \frac{x + \varepsilon w - y}{t -s}\right)\right] \\
	&\quad+ \left[f\left(t, x + \varepsilon w, \frac{x + \varepsilon w - y}{t -s}\right) - f\left(s, y,\frac{x + \varepsilon w - y}{t -s}\right)\right] \\
	&\quad+\left[f\left(s, y, \frac{x + \varepsilon w - y}{t -s}\right) - f(s, y, w)\right].
\eals
We integrate against $\varphi_\varepsilon(y, w)$ on $Q^-_1$ yielding
\bals
	I_1(t, x, v) &:= \int_{Q^-_1} \left[f(t, x, v) - f(t, x+\varepsilon w, v)\right] \varphi_\varepsilon(y, w)\dd s\dd y\dd w, \\
	I_2(t, x, v) &:= \int_{Q^-_1} \left[f(t, x+\varepsilon w, v) - f\left(t, x+\varepsilon w, \frac{x+ \varepsilon w - y}{t -s}\right)\right] \varphi_\varepsilon(y, w)\dd s\dd y\dd w, \\
	I_3(t, x, v) &:= \int_{Q^-_1} \left[f\left(t, x+ \varepsilon w, \frac{x+ \varepsilon w - y}{t -s}\right) - f\left(s, y,\frac{x+ \varepsilon w - y}{t -s}\right)\right] \varphi_\varepsilon(y, w)\dd s\dd y\dd w, \\
	I_4(t, x, v) &:= \int_{Q^-_1}\left[f\left(s, y, \frac{x+ \varepsilon w - y}{t -s}\right) - f(s, y, w)\right]\varphi_\varepsilon(y, w)\dd s\dd y\dd w.
\eals

For $I_2$ we find
\bals
	\fint_{Q_1} &\abs{I_2} \dd t \dd x \dd v \\
	&= \fint_{Q_1} \Bigg\vert \int_{Q^-_1} \left[f(t, x+\varepsilon w, v) - f\left(t, x+\varepsilon w, \frac{x+\varepsilon w - y}{t -s}\right)\right]\varphi_\varepsilon(y, w)\dd s\dd y\dd w \Bigg\vert \dd t\dd x\dd v \\
	&\leq \fint_{Q_1}  \int_{Q^-_1} \frac{\Abs{f(t, x+\varepsilon w, v) - f\big(t, x+\varepsilon w, \frac{x+\varepsilon w - y}{t -s}\big)}}{\Abs{v - \big(\frac{x+\varepsilon w - y}{t-s}\big)}^{\frac{d+2s}{2}}} \Big\vert v - \Big(\frac{x+\varepsilon w - y}{t-s}\Big)\Big\vert^{\frac{d+2s}{2}}\dd s\dd y\dd w \dd t\dd x\dd v \\
	&\leq \int_{(-1, 0)\times B_2\times B_1}  \int_{Q^-_1} \frac{\Abs{f(t, X, v) - f(t, X, \frac{X - y}{t -s})}}{\Abs{v - \Big(\frac{X - y}{t-s}\Big)}^{\frac{d+2s}{2}}} \Big\lvert v - \left(\frac{X - y}{t-s}\right)\Big\rvert^{\frac{d+2s}{2}} \dd s\dd y\dd w \dd t\dd X\dd v \\
	&\leq \int_{(-1, 0)\times B_2\times B_1}  \int_{(-3, -2)\times B_3} \frac{\Abs{f(t, X, v) - f(t, X, V)}}{\abs{v - V}^{\frac{d+2s}{2}}}\abs{v - V}^{\frac{d+2s}{2}} (t - s)^{d} \dd s\dd V \dd t\dd X\dd v \\
	&\lesssim \int_{(-1, 0)\times B_2\times B_1} \int_{B_3} \Abs{f(t, X, v) - f(t, X, V)} K(v, V)^{\frac{1}{2}} \dd V \dd t\dd X\dd v\\
	&\lesssim \int_{Q_3}   \Bigg(\int_{B_3} \Abs{f(t, X, v) - f(t, X, V)}^2K(v, V)\dd V\Bigg)^{1/2} \dd t\dd X\dd v,
\eals
where we used $0 \leq \varphi_\varepsilon \leq 1$, the change of variables $x \to X = x+ \varepsilon w$ and $y \to V = \frac{X - y}{t -s}$, the coercivity estimate \eqref{eq:coercivity-sqrt} and the Cauchy-Schwarz inequality.

For $I_4$ we can proceed similarly with the change of variables $x \to \tilde V = \frac{x +\varepsilon w- y}{t -s}$, Fubini, \eqref{eq:coercivity-sqrt} and Cauchy-Schwarz:
\bals
	\fint_{Q_1} &\abs{I_4} \dd t\dd x\dd v \\
	&= \fint_{Q_1} \Bigg\vert \int_{Q^-_1}\left[f\left(s, y, \frac{x +\varepsilon w - y}{t -s}\right) - f(s, y, w)\right]\varphi_\varepsilon(y, w)\dd s\dd y\dd w\Bigg\vert \dd t\dd x\dd v \\
	&\leq \fint_{Q_1}  \int_{Q^-_1} \frac{\Abs{f\big(s, y, \frac{x +\varepsilon w - y}{t -s}\big) - f(s, y, w)}}{\Abs{\Big(\frac{x +\varepsilon w - y}{t-s}\Big) - w}^{\frac{d+2s}{2}}} \Big\lvert\left(\frac{x +\varepsilon w - y}{t-s}\right) - w\Big\rvert^{\frac{d+2s}{2}}\dd s\dd y\dd w \dd t\dd x\dd v \\	
	&\lesssim \int_{B_3}  \int_{Q^-_1} \frac{\abs{f(s, y, \tilde V) - f(s, y, w)}}{\abs{\tilde V - w}^{\frac{d+2s}{2}}}\abs{\tilde V - w}^{\frac{d+2s}{2}} \dd s\dd y\dd w \dd \tilde V \\
	&\lesssim \int_{Q_3} \Bigg(\int_{B_3} \Abs{f(s, y, \tilde V) - f(s, y, w)}^2K(\tilde V, w) \dd \tilde V\Bigg)^{\frac{1}{2}} \dd s\dd y\dd w.
\eals

For $I_3$ we use a Taylor formula in weak form against $\varphi_\varepsilon$ between $(t, x +\varepsilon w)$ and $(s, y)$ along $\mathcal T$ and the equation \eqref{eq:1.1} satisfied by $f$. We obtain
\bals
	I_3&(t, x, v) \\
	&= \int_{Q^-_1} \left[f\left(t, x+\varepsilon w, \frac{x+\varepsilon w - y}{t -s}\right) - f\left(s, y,\frac{x+\varepsilon w - y}{t -s}\right)\right] \varphi_\varepsilon(y, w)\dd s\dd y\dd w \\
	&\lesssim \int_{Q^-_1} \int^1_0 (t-s)\mathcal Tf\left(\tau t + (1 - \tau)s, \tau(x+\varepsilon w)+ (1 - \tau)y, \frac{x+\varepsilon w  - y}{t -s}\right)\varphi_\varepsilon(y, w)\dd\tau \dd s\dd y\dd w \\
	&\lesssim  \int_{Q^-_1} \int^1_0 (t-s) \mathcal L f \left(\tau t + (1 - \tau)s, \tau(x+\varepsilon w) + (1 - \tau)y, \frac{x+\varepsilon w  - y}{t -s}\right)\varphi_\varepsilon(y, w)\dd\tau\dd s\dd y\dd w\\
	&\quad+  \int_{Q^-_1} \int^1_0 (t-s) h\left(\tau t + (1 - \tau)s, \tau(x+\varepsilon w) + (1 - \tau)y, \frac{x +\varepsilon w - y}{t -s}\right)\varphi_\varepsilon(y, w)\dd\tau \dd s\dd y\dd w\\
	&=: I_{31}(t, x, v) +I_{32}(t, x, v).
\eals
For $I_{31}$ we then perform the change of variables
\beqs
	(y, w) \to (Y, W) := \left(\tau(x+\varepsilon w) + (1-\tau)y, \frac{x+\varepsilon w - y}{t-s}\right),
\eeqs
such that $(y, w) \to (Y, W)$ is a bijection from $(B_1)^2$ to the set
\beqs
	E \subset B(\tau x, (1-\tau) + \tau \varepsilon) \times B\left(\frac{x}{t-s}, \frac{1+\varepsilon}{t-s}\right) \subset B_2\times B_3,
\eeqs
with Jacobian $\big(\frac{\varepsilon}{t-s}\big)^d$. Then we write for brevity $T = \tau t + (1-\tau)s$ so that
\bals
	\fint_{Q_1}&\Abs{I_{31}(t, x, v)}\dd t\dd x\dd v \\
	&\lesssim \varepsilon^{-d}\int_{(-1, 0]\times B_3\times B_1}\Bigg\lvert\int_0^1\int_{(-3, -2] \times E} \mathcal L f (T, Y,  W)\\
	&\qquad\qquad\qquad\qquad\times\varphi_\varepsilon\left(Y - \tau(t-s)W, \frac{Y-x+(1-\tau)(t-s)W}{\varepsilon}\right)\dd s\dd Y\dd W \dd \tau\Bigg\rvert\dd t\dd x\dd v.
\eals
We distinguish the symmetric and skew symmetric part: write 
\bals
&\varphi_\varepsilon := \varphi_\varepsilon\left(Y - \tau(t-s)W, \frac{Y-x+(1-\tau)(t-s)W}{\varepsilon}\right),\\
&\varphi_\varepsilon' := \varphi_\varepsilon\left(Y - \tau(t-s)W', \frac{Y-x+(1-\tau)(t-s)W'}{\varepsilon}\right).
\eals
Then, if we write $E =: E_x \times E_v$
\bals
	 \int_{E}&\int_{\R^d}\big[f (W') - f(W)\big]K(W, W')\dd W'\varphi_\varepsilon\Big(Y - \tau(t-s)W, \frac{Y-x+(1-\tau)(t-s)W}{\varepsilon}\Big)\dd W\dd Y \\
	 &= \int_{E_x \times \R^d}\int_{\R^d}\big[f (W') - f(W)\big]K(W, W')\dd W'\varphi_\varepsilon\dd W\dd Y \\
	 &= \frac{1}{2}\int_{E_x \times \R^d}\int_{\R^d}\big[f (W') - f(W)\big]K(W, W')\big[\varphi_\varepsilon -\varphi_\varepsilon'\big] \dd W'\dd W\dd Y\\
	 &\quad +\frac{1}{2} \int_{E_x \times \R^d}\int_{\R^d}\big[f (W') - f(W)\big]K(W, W')\big[\varphi_\varepsilon +\varphi_\varepsilon'\big] \dd W'\dd W\dd Y.
\eals
For the symmetric part, we simply use Cauchy-Schwarz inequality and the regularity of $\varphi_\varepsilon$:
\bals
	 \frac{1}{2}&\int_{E_x \times \R^d}\int_{\R^d}\big[f (W') - f(W)\big]K(W, W')\big[\varphi_\varepsilon -\varphi_\varepsilon'\big] \dd W'\dd W\dd Y\\
	 &\leq  \frac{1}{2}\int_{E_x \times \R^d}\Bigg\{\int_{\R^d}\big[f (W') - f(W)\big]^2K(W, W')\dd W'\Bigg\}^{\frac{1}{2}}\\
	 &\qquad\qquad\quad\times\Bigg\{\int_{\R^d}\big[\varphi_\varepsilon -\varphi_\varepsilon'\big]^2K(W, W') \dd W'\Bigg\}^{\frac{1}{2}}\dd W\dd Y\\
	 &\leq C[\varphi_\varepsilon]_{C^2}\int_{E_x \times \R^d}\Bigg(\int_{\R^d}\big[f (W') - f(W)\big]^2K(W, W')\dd W'\Bigg)^{\frac{1}{2}}\dd W\dd Y.
\eals
Let us remark that we can also prove that the symmetric part is bounded by the $H^s_v$ norm by doing a dyadic decomposition similar to the proof of Theorem \ref{thm:4.1}. However, it will be more convenient in the Intermediate Value Lemma, Theorem \ref{thm:IVL}, to work with the kernel $K$ directly.

For the skew-symmetric part, we write
\bals
	\frac{1}{2}\int_{E_x \times \R^d}\int_{\R^d}&\big[f (W') - f(W)\big]\big[\varphi_\varepsilon + \varphi_\varepsilon'\big]K(W, W')\dd W'\dd W\dd Y\\
	&=\frac{1}{2} \int_{E_x \times  \R^d}\int_{\R^d}f(W)\big[\varphi_\varepsilon + \varphi_\varepsilon'\big]\big[K(W', W)- K(W, W')\big]\dd W'\dd W\dd Y\\
	&\leq C\int_{E_x \times \R^d}f(W)\Bigg\lvert\textrm{PV}\int_{\R^d}\big[K(W',W)- K(W, W')\big]\dd W'\Bigg\rvert\dd W\dd Y.
\eals
Note that the cancellation assumptions \eqref{eq:cancellation1} would allow us to bound the last line by the $L^1$ norm of $f$. But again, we will keep it as it is for later. 
Thus we obtained
\bals
	\fint_{Q_1} &\abs{I_{31}(t, x, v)}\dd t\dd x \dd v \\
	&\lesssim \varepsilon^{-{d+2}} \int_{(-3^{-2s}, 0]\times B_{3^{1+2s}} \times \R^d}\Bigg(\int_{\R^d}\Abs{f (W) - f(W')}^2K(W,W')\dd W'\Bigg)^{\frac{1}{2}} \dd t\dd Y\dd W\\
	&\quad+ \varepsilon^{-d}\int_{(-3^{-2s}, 0]\times B_{3^{1+2s}} \times \R^d}f(W)\Bigg\lvert\textrm{PV}\int_{\R^d}\big[K(W', W)- K(W, W')\big]\dd W'\Bigg\rvert\dd t\dd Y\dd W.
\eals

For the source term we perform the change of variable $y \to V = \frac{x+\varepsilon w - y}{t-s}$, $x \to X =x+\varepsilon w -(1-\tau)(t-s)V, s \to s' = t-s$ and $t' \to t - (1-\tau)s'$. 
We deduce
\beqs
	\fint_{Q_1}\abs{I_{32}(t, x, v)}\dd t\dd x\dd v \lesssim \int_{(-3,0]\times B_3\times B_3}\abs{h(t', X, V)}\dd t'\dd X\dd V\lesssim \int_{Q_3} \abs{h} \dd t'\dd X\dd V.
\eeqs

Finally, for $I_1$ we estimate
\bals
	\fint_{Q_1} \abs{I_1(t, x, v)} \dd t\dd x\dd v &\lesssim \fint_{Q_1} \int_{Q^-_1}   \Abs{f(t, x, v) - f(t, x+\varepsilon w, v)} \varphi_\varepsilon(y, w)\dd s\dd y\dd w\dd t\dd x\dd v \\
	&\lesssim \int_{Q_1}\int_{B_1} \frac{\abs{f(t, x, v) - f(t, x+\varepsilon w, v)}}{\abs{\varepsilon w}^{d + \sigma}}\abs{\varepsilon w}^{d + \sigma} \dd w\dd t\dd x\dd v \\
	&\lesssim \varepsilon^\sigma \int_{Q_1}\int_{B_2}  \frac{\abs{f(t, x, v) - f(t, x', v)}}{\abs{x' - x}^{d + \sigma}} \dd x'\dd t\dd x\dd v \\
	&\lesssim \varepsilon^\sigma \norm{f}_{L^1_{t, v}W^{\sigma, 1}_x(Q_2)},
\eals
where we use $0 \leq \varphi_\varepsilon \leq 1$ and the change of variables $w \to x' = x + \varepsilon w$.

Combining all these estimates yields the claim. 
\end{proof}

\subsection{Intermediate Value Lemma}\label{subsec:ivl}
In this section we prove the Second De Giorgi Lemma, also known as the Intermediate Value Lemma, stated in Theorem \ref{thm:IVL}. The proof exploits the cross-term arising in the energy estimate. The idea originates from Bass and Kassmann \cite{kassbass} and was further employed by Caffarelli, Chan and Vasseur \cite{CCV}. However, this is the first time it is used in the context of kinetic equations and for kernels as general as the non-cutoff Boltzmann kernel. 

Let $\mu < 1$. We take a cut-off $\varphi \in C^\infty_c(\R^{d})$ in $x$ such that $0 \leq \varphi \leq 1$ and where $\varphi = 1$ in $B_{(3r_0)^{1+2s}}$ and $\varphi = 0$ outside $B_{(9r_0)^{1+2s}}$. Denote $\psi = 1 - \varphi$. We further define 
\bals
	&F_0(v) = \frac{1}{r_0}\sup\Bigg(-r_0, \frac{\inf(0, \abs{v}^2 - 10r_0^2)}{r_0}\Bigg),\\
	&F_1(v) = \frac{1}{r_0}\sup\Bigg(-r_0, \frac{\inf(0, \abs{v}^2 - 9r_0^2)}{r_0}\Bigg),\\
	&F_2(v) = \frac{1}{r_0}\sup\Bigg(-r_0, \frac{\inf(0, \abs{v}^2 - 8r_0^2)}{r_0}\Bigg). 
\eals
These functions are Lipschitz, equal to $0$ outside $B_{\sqrt{10}r_0},  B_{3r_0}$ and $B_{2\sqrt{2}r_0}$ respectively and equal to $-1$ in $B_{3r_0}, B_{2\sqrt2r_0}$ and $B_{\sqrt{7}r_0}$. Then we use the following three cut-offs:
\bal\label{eq:barriers}
	&\varphi_0(x, v) = \psi(x) + 1+ F_0(v),\\
	&\varphi_1(x, v) = \psi(x) + 1+ \mu F_1(v),\\
	&\varphi_2(x, v) = \psi(x) + 1+ \mu^2 F_2(v).
\eal
We refer the reader to Figure \ref{fig:barriers} to visualise \eqref{eq:barriers}. 
\begin{proof}[Proof of Theorem \ref{thm:IVL}]
Let $f$ be a sub-solution to \eqref{eq:1.1}-\eqref{eq:1.2} in $[-3, 0] \times Q_1^{t}$ satisfying \eqref{eq:IVLc}. 

\textit{Step 1: Improved energy estimate.}
We test \eqref{eq:1.1} with $(f-\varphi_1)_+$ similar to the proof of the energy estimate that we had before. We get for $-3 < T_1 < T_2 < 0$
\bal
	\int\int (f - \varphi_1)_+^2& \dd x \dd v \Big\lvert^{T_2}_{T_1} + \int_{T_1}^{T_2}\int \mathcal E\big((f-\varphi_1)_+, (f-\varphi_1)_+\big) \dd x \dd t \\
	&\leq - \int_{T_1}^{T_2}\int \mathcal E\big(\varphi_1, (f-\varphi_1)_+\big) \dd x \dd t - \int_{T_1}^{T_2}\int \mathcal E\big((f-\varphi_1)_-, (f-\varphi_1)_+\big) \dd x \dd t \\
	&\qquad + \int_{T_1}^{T_2} \int\int (f-\varphi_1)_+ h\dd z.
\label{eq:3.4}
\eal
Note that here holds
\bals
	(f - \varphi_1)_+(v) \leq f - 1 + \mu \leq \mu.
\eals
Therefore by our assumption on the source term we have
\bals
	 \int_{T_1}^{T_2} \int\int (f-\varphi_1)_+ h\dd z \leq C\mu^3.
\eals	
Furthermore we estimate
\bals
	-\mathcal E^{\textrm{sym}}\big(\varphi_1, (f-\varphi_1)_+\big)\leq &\frac{1}{8} \mathcal E^{\textrm{sym}}\big((f-\varphi_1)_+, (f-\varphi_1)_+\big) \\
	&+ \int \int \Abs{\varphi_1(v) - \varphi_1(w)}^2K(v, w) \big(\chi_{B_{3r_0}}(v)+\chi_{B_{3r_0}}(w)\big) \dd w \dd v.
\eals	
The first term is absorbed on the left in \eqref{eq:3.4}. For the second one there holds 
\bals
	2 \int \int \Abs{\varphi_1(v) - \varphi_1(w)}^2K(v, w) \chi_{B_{3r_0}}(v) \dd w \dd v &\leq 4\mu^2\int\int\Abs{F_1(v) - F_1(w)}^2K(v, w) \chi_{B_{3r_0}}(v)\dd w \dd v \\
	&\leq C \mu^2\int\int\abs{v -w}^2K(v, w)\chi_{B_{3r_0}}(v) \dd w \dd v \\
	&\leq C\Lambda \mu^2,
\eals
using that $F_1$ is Lipschitz and the upper bound \eqref{eq:upperbound-2} on $K$. For the skew-symmetric part we note that $\mathcal E^{\textrm{skew}} = 0$ for $f \leq \varphi_1$. Thus 
\bals
	-&\mathcal E^{\textrm{skew}}\big(\varphi_1, (f-\varphi_1)_+\big)\\
	&= \frac{1}{2}\int\int\big(\varphi_1(w) - \varphi_1(v)\big)(f - \varphi_1)_+(v)\big(K(v, w) - K(w, v)\big) \dd w \dd v\\
	&= \frac{1}{2}\int\int\big(\varphi_1(w) - \varphi_1(v)\big)(f - \varphi_1)_+(v)\big(K(v, w) - K(w, v)\big)\chi_{B_{3r_0}}(v) \dd w \dd v\\
	&=\int\int\mu\big(F_1(w) - F_1(v)\big)(f - \varphi_1)_+(v)\big(K(v, w) - K(w, v)\big)\chi_{B_{3r_0}}(v) \dd w \dd v\\
	&\leq \mu \int \Bigg\lvert\textrm{PV} \int(v-w)\big(K(v, w) - K(w, v)\big) \dd w\Bigg\rvert (f - \varphi_1)_+(v)\chi_{B_{3r_0}}(v)\dd v\\
	&\leq C \mu \Lambda  \int(f - \varphi_1)_+(v)\chi_{B_{3r_0}}(v) \dd v.
\eals
We used the cancellation assumptions \eqref{eq:cancellation2} and the definition of $F_1$. Thus 
\bals
	-\mathcal E^{\textrm{skew}} \big( \varphi_1, (f-\varphi_1)_+\big)\leq C \mu^2.
\eals
We deduce
\bals
	-\mathcal E \big( \varphi_1, (f-\varphi_1)_+\big) \leq C \mu^2.
\eals
Finally we write due to the cancellation assumptions \eqref{eq:cancellation1}
\bals
	 \mathcal E\big((f-\varphi_1)_+, (f-\varphi_1)_+\big)  &\geq  \int\int \big((f-\varphi_1)_+(v) -  (f-\varphi_1)_+(w)\big)^2 K(v, w) \dd w \dd v \\
	 &\quad- \int\int (f-\varphi_1)_+^2(v) \Bigg\lvert\int\big(K(v, w) -K(w, v)\big)\dd w\Bigg\rvert \dd v\\
	 &\geq\int\int \big((f-\varphi_1)_+(v) -  (f-\varphi_1)_+(w)\big)^2 K(v, w) \dd w \dd v -\Lambda\mu^2.
\eals
Thus we obtain from \eqref{eq:3.4} for $t \in (T_1, T_2)$
\bal
	\int\int_{Q_{3r_0}^t} &(f - \varphi_1)_+^2 \dd x \dd v \Big\lvert^{T_2}_{T_1}\\
	&+ \int_{T_1}^{T_2}\int_{Q_{3r_0}^t} \int_{\R^d} \ \big((f-\varphi_1)_+(v) -  (f-\varphi_1)_+(w)\big)^2 K(v, w) \dd w\dd z \\
	&+ \int_{T_1}^{T_2}\int_{B_{(3r_0)^{1+2s}}} \mathcal E\big((f-\varphi_1)_-, (f-\varphi_1)_+\big) \dd x \dd t \leq C \mu^2(T_2 - T_1).
\label{eq:enestimimprov}
\eal
In particular, we deduce for $-3 < T_1 < T_2 < 0$:
\bals
	- \int_{T_1}^{T_2}\int \int\int(f-\varphi_1)_+(v) (f-\varphi_1)_-(w) K(v, w) \dd w\dd v\dd x \dd t \leq C\mu^2\big(T_2 - T_1\big).
\eals

\textit{Step 2: Weak Poincaré for the intermediate cut-off.}
We want to use the hypoelliptic Poincaré inequality for the intermediate cut-off $(f-\varphi_1)_+$. We note that in $(-3, 0] \times Q^t_{3r_0}$ there holds $\nabla_x \varphi_1 = 0$, thus
\bals
	\mathcal T(f-\varphi_1)_+ = \mathcal Tf (v) \chi_{f> \varphi_1}(v) \leq \big[\mathcal Lf(v) + h\big]\chi_{f> \varphi_1}(v) \leq \mathcal L(f-\varphi_1)_+ + \underbrace{\left(\mu\mathcal LF(v)+h\right)\chi_{f > \varphi_1}(v)}_{=: \tilde h}.
\eals
Therefore $(f-\varphi_1)_+$ is a sub-solution of \eqref{eq:1.1} in $(-3, 0]\times Q^t_{3r_0}$ with source term $\tilde h$. Then, if we denote by $\Omega_{3r_0}$ the domain $\Omega_{3r_0} := (-(3r_0)^{2s}, 0] \times B_{(3r_0)^{1+2s}} \times \R^d$, the hypoelliptic Poincaré inequality \eqref{eq:3.1} gives for some $\varepsilon \in (0, 1)$ and $\sigma$ from Proposition \ref{prop:13}
\bal
	\big\Vert\big[(f-\varphi_1)_+ &- \langle(f-\varphi_1)_+\rangle_{Q_{r_0}^-}\big]_+\big\Vert_{L^1(Q_{r_0})}\\
	 &\leq \frac{1}{\varepsilon^{{d+2}}}\int_{\Omega_{3r_0}}\Bigg(\int_{\R^d}\Abs{(f-\varphi_1)_+(v) - (f-\varphi_1)_+(w)}^2K(v, w)\dd w\Bigg)^{\frac{1}{2}} \dd z\\
	&\quad+ \frac{1}{\varepsilon^d}\int_{\Omega_{3r_0}}\int_{\R^d} (f-\varphi_1)_+(v)\big(K(v, w) - K(w, v)\big) \dd w \dd z\\
	&\quad+ \varepsilon^{\sigma} \norm{(f-\varphi_1)_+}_{L^1_{t, v}W_x^{\sigma, 1}(Q_{2r_0})} + \norm{\tilde h}_{L^1((Q_{3r_0})}.
\label{eq:hypopc}
\eal
In the sequel we write $n = 2s + d(2+2s)$ for the total dimension of the kinetic cylinder such that $\abs{Q_{r_0}} \sim r_0^n$.
On the one hand, we have by assumption \eqref{eq:IVLc}
\bals
	\fint_{Q_{r_0}^-} (f-\varphi_1)_+ \dd z &\leq \mu \frac{\abs{\{f > \varphi_1\} \cap Q_{r_0}^-}}{\abs{Q_{r_0}^-}} \leq \mu\frac{\abs{\{f > \varphi_0\} \cap Q_{r_0}^-}}{\abs{Q_{r_0}^-}}\\
	&\leq \mu\left(1 - \frac{\abs{\{f \leq \varphi_0\} \cap Q_{r_0}^-}}{\abs{Q_{r_0}^-}}\right)\leq \mu - \mu\delta_1,
\eals	
so that with \eqref{eq:IVLc}
\bal
	\big\Vert \big[(f-\varphi_1)_+ - \langle(f-\varphi_1)_+\rangle_{Q_{r_0}^-}\big]_+\big\Vert_{L^1(Q_{r_0})} &\geq \int_{Q_{r_0}\cap \{f > \varphi_2\}}\big[\big(f-\varphi_1\big)_+ - \mu(1 - \delta_1)\big]_+\dd z\\
	&\geq \int_{\{f > \varphi_2\} \cap Q_{r_0}}\big[\big(\varphi_2 -\varphi_1\big)_+ - \mu(1 - \delta_1)\big]_+\dd x\dd v\\
	&\geq\big(\mu \delta_1 - \mu^2\big)\delta_2 r_0^n.
\label{eq:3.10}
\eal
Note that $\varphi_2 - \varphi_1 = \mu - \mu^2$ in $Q_{r_0}$.

\textit{Step 3-(i): Upper bounds on the symmetric fractional gradient.} On the other hand, we find 
\bals
	\int_{\Omega_{3r_0}}&\Bigg(\int_{\R^d} \abs{(f-\varphi_1)_+(v) - (f-\varphi_1)_+(w)}^2K(v, w)\dd w\Bigg)^{\frac{1}{2}} \dd z\\
	&= \underbrace{\int_{\Omega_{3r_0} \cap \{ f< \varphi_0\}} \dots}_{=: I_1} + \underbrace{\int_{ \Omega_{3r_0}\cap\{\varphi_0 < f < \varphi_2\}}\dots}_{=: I_2} + \underbrace{\int_{\Omega_{3r_0} \cap \{f > \varphi_2\}}\dots}_{=: I_3}.
\eals
For $I_2$ we get 
\bal
	I_2 &\leq \Abs{\{\varphi_0 < f < \varphi_2\} \cap Q_{3r_0}}^{\frac{1}{2}}\Bigg(\int_{\Omega_{3r_0}}\int_{\R^d}\abs{(f-\varphi_1)_+(v) - (f-\varphi_1)_+(w)}^2K(v, w)\dd w \dd z\Bigg)^{\frac{1}{2}}\\
	&\leq C\mu r_0^{\frac{n}{2}}\Abs{\{\varphi_0 < f < \varphi_2\} \cap (-3, 0]\times Q^t_{\frac{1}{2}}}^{\frac{1}{2}}\leq Cr_0^{\frac{n}{2}}\Abs{\{\varphi_0 < f < \varphi_2\} \cap (-3, 0]\times Q^t_{\frac{1}{2}}}^{\frac{1}{2}}.
\label{eq:up1}
\eal

For $I_1$ we denote with $\Sigma_{t, x}$ the set of times and spaces where 
$\Abs{\{f(t, x, \cdot) \leq \varphi_0\} \cap Q_{3r_0}} \geq \frac{\delta_1}{4}$. In particular by \eqref{eq:IVLc} we know that $\abs{\Sigma} \geq C\delta_1 r_0^{n-d}$. We have
\bal\label{eq:up1.5}
	I_1 = \int_{\Omega_{3r_0}\cap \{f \leq \varphi_0\}}&\Bigg(\int_{\R^d} \Abs{(f-\varphi_1)_+(v) - (f-\varphi_1)_+(w)}^2K(v, w)\dd w\Bigg)^{\frac{1}{2}} \dd z\\
	&\leq Cr_0^{\frac{n}{2}}\Bigg(\int_{\Sigma_{t, x} \times \R^d} \int_{\R^d} (f-\varphi_1)^2_+(v)\chi_{f \leq \varphi_0}(w) K(v, w)\dd w \dd z\Bigg)^{\frac{1}{2}}\\
	&\leq Cr_0^{\frac{n}{2}} \frac{\mu^{\frac{1}{2}}}{\delta_1^{\frac{1}{2}}} \Bigg(-\int_{\Omega_{3r_0}} \int_{\R^d} (f-\varphi_1)_+(v)(f-\varphi_1)_-(w)K(v, w)\dd w \dd z\Bigg)^{\frac{1}{2}}\\
	&\leq C \frac{\mu^{\frac{3}{2}}}{\delta_1^{\frac{1}{2}}}r_0^n.
\eal
We used Cauchy-Schwarz, the fact that $(f-\varphi_1)_+ \leq \mu$, the gap between $\varphi_1$ and $\varphi_0$ and the energy estimate \eqref{eq:enestimimprov}.

To estimate $I_3$ we will split it into three parts
\bal
	I_3 &= \int_{\{f > \varphi_2\}} \Bigg(\int_{\R^d} \abs{(f-\varphi_1)_+(v) - (f-\varphi_1)_+(w)}^2K(v, w)\dd w \Bigg)^{\frac{1}{2}}\dd z \\
	&=  \int_{\{f > \varphi_2\}} \Bigg(\int_{\{f > \varphi_2\}} \dots \dd w\Bigg)^{\frac{1}{2}} \dd z+ \int_{\{f > \varphi_2\}} \Bigg(\int_{\{\varphi_0 < f < \varphi_2\}} \dots \dd w \Bigg)^{\frac{1}{2}} \dd z\\
	&\qquad+ \int_{\{f > \varphi_2\}} \Bigg(\int_{\{f < \varphi_0\}} \dots \dd w\Bigg)^{\frac{1}{2}} \dd z.
\label{eq:splitting}
\eal
The first term gives 
\bals
	\int_{\{f > \varphi_2\}}&\Bigg(\int_{\{f > \varphi_2\}} \Abs{(f-\varphi_1)_+(v) - (f-\varphi_1)_+(w)}^2K(v, w)\dd w \Bigg)^{\frac{1}{2}}\dd z\\
	&\leq \Bigg(\int_{\{f > \varphi_2\}}\int_{\{f > \varphi_2\}} \Abs{(f-\varphi_2)_+(v) - (f-\varphi_2)_+(w)}^2K(v, w)\dd w \dd z\Bigg)^{\frac{1}{2}} \\
	&\qquad+ \Bigg(\int \int \abs{\varphi_2(v) - \varphi_2(w)}^2K(v, w) \dd w \dd v\Bigg)^{\frac{1}{2}}\\
	&\leq C\mu^2 r_0^n.
\eals
since $\varphi_1 = 1-\mu$ on the set where $\{f > \varphi_2\}$ and since $F_2$ is Lipschitz. 
The third term of \eqref{eq:splitting} can be estimated as $I_1$:
\bals
	\int_{\{f > \varphi_2\}}&\Bigg(\int_{\{f < \varphi_0\}}\Abs{(f-\varphi_1)_+(v) - (f-\varphi_1)_+(w)}^2K(v, w)  \dd w \Bigg)^{\frac{1}{2}} \dd z\\
	&\leq Cr_0^{\frac{n}{2}} \Bigg(\int_{\{f > \varphi_2\}}\int_{\{f < \varphi_0\}}\Abs{(f-\varphi_1)_+(v) - (f-\varphi_1)_+(w)}^2K(v, w) \dd w \dd z\Bigg)^{\frac{1}{2}}\\
	&\leq C\frac{r_0^{\frac{n}{2}}}{\delta_1^{\frac{1}{2}}} \Bigg(\int_{\Sigma_{t, x} \times \R^d} \int_{\R^d} (f-\varphi_1)^2_+(v)\chi_{f \leq \varphi_0}(w) K(v, w)\dd w \dd z\Bigg)^{\frac{1}{2}}\\
	&\leq Cr_0^{\frac{n}{2}} \frac{\mu^{\frac{1}{2}}}{\delta_1} \Bigg(-\int_{\Omega_{3r_0}} \int_{\R^d} (f-\varphi_1)_+(v)(f-\varphi_1)_-(w)K(v, w)\dd w \dd z\Bigg)^{\frac{1}{2}}\\
	&\leq C \frac{\mu^{\frac{3}{2}}}{\delta_1}r_0^n.
\eals
For the second term in \eqref{eq:splitting} we further distinguish the singular from the non-singular part. Let $\gamma_1 > 0$. Then
\bals
	I_3 &= \int_{\{f > \varphi_2\}} \Bigg(\int_{\{\varphi_0 < f < \varphi_2\}}  \Abs{(f-\varphi_1)_+(v) - (f-\varphi_1)_+(w)}^2K(v, w)\dd w \Bigg)^{\frac{1}{2}}\dd z \\
	&= \int_{\{f > \varphi_2\}\cap \abs{v-w} < \gamma_1} \Bigg(\int_{\{\varphi_0 < f < \varphi_2\}} \dots\dd w \Bigg)^{\frac{1}{2}}\dd z + \int_{\{f > \varphi_2\}\cap \abs{v-w} > \gamma_1} \Bigg(\int_{\{\varphi_0 < f < \varphi_2\}} \dots\dd w \Bigg)^{\frac{1}{2}}\dd z.
\eals
Away from the diagonal we pick up an intermediate set due to the upper bound \eqref{eq:upperbound}
\bals
	&\int_{\{f > \varphi_2\}\cap \abs{v-w} > \gamma_1}\Bigg(\int_{\{\varphi_0 < f < \varphi_2\}}\Abs{(f-\varphi_1)_+(v) - (f-\varphi_1)_+(w)}^2K(v,w)  \dd w \Bigg)^{\frac{1}{2}} \dd z\\
	&\quad\leq r_0^{\frac{n}{2}}\Bigg(\int_{\{f > \varphi_2\}\cap \abs{v-w} > \gamma_1} \int_{\{\varphi_0 < f < \varphi_2\}} \Abs{(f-\varphi_1)_+(v) - (f-\varphi_1)_+(w)}^2K(v, w)\dd w \dd z\Bigg)^{\frac{1}{2}}\\
	&\quad\leq C\mu r_0^{n}\gamma_1^{-s}\abs{\{\varphi_0 < f < \varphi_2\}}^{\frac{1}{2}}.
\eals
Close to the diagonal we have with Cauchy-Schwarz and the energy estimate \eqref{eq:enestimimprov}
\bals
	\int_{\{f > \varphi_2 \}\cap \abs{v-w} < \gamma_1}&\Bigg(\int_{\{\varphi_0 < f < \varphi_2\}}\abs{(f-\varphi_1)_+(v) - (f-\varphi_1)_+(w)}^2K(v, w)  \dd w \Bigg)^{\frac{1}{2}} \dd z\\
	&\leq C\gamma_1^{\frac{1}{2}} \Bigg(\int \int_{\abs{v-w} < \gamma_1} \abs{(f-\varphi_1)_+(v) - (f-\varphi_1)_+(w)}^2K(v, w)  \dd w \dd z\Bigg)^{\frac{1}{2}}\\
	&\leq C\gamma_1^{\frac{1}{2}} \mu r_0^n.
\eals
Thus we conclude for $I_3$ choosing $\gamma_1$ such that $\gamma_1^{\frac{1}{2}} = \mu^{\frac{1}{2}}$
\bal
	I_3 &\leq C\mu^{\frac{3}{2}} r_0^n + C\frac{\mu^{\frac{3}{2}}}{\delta_1} r_0^n + C\mu^{1-s}\Abs{\{\varphi_0 < f < \varphi_2\} \cap Q_{5r_0}}^{\frac{1}{2}}\\
	&\leq C\frac{\mu^{\frac{3}{2}}}{\delta_1}r_0^n+ C\mu^{1-s}\Abs{\{\varphi_0 < f < \varphi_2\} \cap Q_{5r_0}}^{\frac{1}{2}}.
\label{eq:up2}
\eal

It remains to estimate the error terms in the Poincaré inequality. 
 
\textit{Step 3-(ii): Upper bounds on the anti-symmetric fractional gradient.}
For the skew-symmetric right hand side in \eqref{eq:hypopc} we get
\bals
	2\int_{\R^d}\int_{\R^d}  (f-\varphi_1)_+(v)&\big(K(v, w) -K(w, v)\big)\dd w \dd v \\
	&= \int_{\R^d} \int_{\R^d} \big((f-\varphi_1)_+(w) - (f-\varphi_1)_+(v)\big)\big(K(v, w) -K(w, v)\big)\dd w \dd v \\
	&= \int_{\{\varphi_0 < f < \varphi_2\}} \int_{\R^d} \big((f-\varphi_1)_+(w) - (f-\varphi_1)_+(v)\big)\big(K(v, w) -K(w, v)\big)\dd w \dd v \\
	&\quad + \int_{\{\varphi_2 < f\}} \int_{\R^d} \big((f-\varphi_1)_+(w) - (f-\varphi_1)_+(v)\big)\big(K(v, w) -K(w, v)\big)\dd w \dd v \\
	&\quad +\int_{\{f < \varphi_0\}} \int_{\R^d} \big((f-\varphi_1)_+(w) - (f-\varphi_1)_+(v)\big)\big(K(v, w) -K(w, v)\big)\dd w \dd v.
\eals
First, due to the cancellation \eqref{eq:cancellation1}
\bals
	\int_{\{\varphi_0 < f < \varphi_2\}}&\int_{\R^d} \big((f-\varphi_1)_+(w) - (f-\varphi_1)_+(v)\big)\big(K(v, w) -K(w, v)\big)\dd w \dd v \\
	&\leq C\mu \int_{\{\varphi_0 < f < \varphi_2\}} \int_{\R^d} \big(K(v, w) -K(w, v)\big)\dd w \dd v \\
	&\leq C\mu \Abs{\{\varphi_0 < f < \varphi_2\}}.
\eals
Second with the energy estimate \eqref{eq:enestimimprov}
\bal\label{eq:aux-varphi0}
	\int_{\{f < \varphi_0\}} \int_{\R^d} \big((f-\varphi_1)_+(w)&- (f-\varphi_1)_+(v)\big)\big(K(v, w) -K(w, v)\big)\dd w \dd v \\
	&= \int_{\{f < \varphi_0\}} \int_{\R^d} (f-\varphi_1)_+(w) \big(K(v, w) -K(w, v)\big)\dd w \dd v \\
	&\leq -\frac{C}{\delta_1r_0^{n-d}}\int_{\Sigma_{t, x}} \int_{\R^d} (f-\varphi_1)_+(w) (f-\varphi_1)_-(v) K(v, w)\dd w \dd v \\
	&\leq C \frac{\mu^2}{\delta_1}r_0^d.
\eal
Third we split further and use the energy estimate \eqref{eq:enestimimprov} on $(f-\varphi_2)_+$, the cancellation \eqref{eq:cancellation1}, and \eqref{eq:aux-varphi0} to bound
\bals
	\int_{\{f > \varphi_2\}} &\int_{\R^d} \big((f-\varphi_1)_+(w) - (f-\varphi_1)_+(v)\big)\big(K(v, w) -K(w, v)\big)\dd w \dd v \\
	&= \int_{\{f > \varphi_2\}} \int_{\{f > \varphi_2\}} \big((f-\varphi_1)_+(w) - (f-\varphi_1)_+(v)\big)\big(K(v, w) -K(w, v)\big)\dd w \dd v \\
	&\quad +\int_{\{f > \varphi_2\}} \int_{\{\varphi_0 < f < \varphi_2\}} \big((f-\varphi_1)_+(w) - (f-\varphi_1)_+(v)\big)\big(K(v, w) -K(w, v)\big)\dd w \dd v \\
	&\quad +\int_{\{f > \varphi_2\}} \int_{\{f < \varphi_0\}} \big((f-\varphi_1)_+(w) - (f-\varphi_1)_+(v)\big)\big(K(v, w) -K(w, v)\big)\dd w \dd v \\
	&\leq \int_{\{f > \varphi_2\}} \int_{\{f > \varphi_2\}} \big((f-\varphi_2)_+(w) - (f-\varphi_2)_+(v)\big)\big(K(v, w) -K(w, v)\big)\dd w \dd v \\
	&\quad + \int_{\{f > \varphi_2\}} \int_{\{f > \varphi_2\}} \big(\varphi_2(w) - \varphi_2(v)\big)\big(K(v, w) -K(w, v)\big)\dd w \dd v  \\
	&\quad+  C\mu \abs{\{\varphi_0 < f < \varphi_2\}} + C \frac{\mu^2}{\delta_1}r_0^d\\
	&\leq C\mu^2r_0^d +  C\mu \abs{\{\varphi_0 < f < \varphi_2\}} + C \frac{\mu^2}{\delta_1}r_0^d.
\eals
Thus
\beq
	\int_{\R^d} \int_{\R^d}  (f-\varphi_1)_+(v)\big(K(v, w) -K(w, v)\big)\dd w \dd v \leq C\mu^2r_0^d +  C\mu \abs{\{\varphi_0 < f < \varphi_2\}} + C \frac{\mu^2}{\delta_1}r_0^d.
\label{eq:auxskew}
\eeq

\textit{Step 3-(iii): Upper bounds on the small variation in the spatial variable.}
We realise by Lemma \ref{prop:11} that the term
$\norm{(f-\varphi_1)_+}_{L^1_{t, v}W_x^{\sigma, 1}(Q_{2r_0})}$ is bounded by the right hand side of the energy estimate. Thus, \eqref{eq:2.6} and \eqref{eq:enestimimprov} imply the existence of some constant $C$ depending on $r_0, T_1, T_2$ such that
\beq\label{eq:spatial-grad}
	\norm{(f-\varphi_1)_+}_{L^1_{t, v}W_x^{\sigma, 1}(Q_{2r_0})} \leq C \mu.
\eeq

\textit{Step 3-(iv): Upper bounds on the source $\tilde h$.}
Finally, we estimate the source term $\tilde h$ in the Poincaré-inequality \eqref{eq:hypopc}. The first part consists of
\bal\label{eq:s1-aux0}
	\int_{B_{3r_0}} \mathcal L\varphi_1(v)\chi_{f > \varphi_1}(v) \dd v&= \int_{\{\varphi_1 < f < \varphi_2\}} \int \big(\varphi_1(w) - \varphi_1(v)\big)K(v, w) \dd w\dd v \\
	&\quad  + \int_{\{f > \varphi_2\}} \int_{\{\varphi_0 < f < \varphi_2\}} \big(\varphi_1(w) - \varphi_1(v)\big)K(v, w) \dd w\dd v \\
	&\quad +\int_{\{f > \varphi_2\}} \int_{\{f < \varphi_0\}} \big(\varphi_1(w) - \varphi_1(v)\big)K(v, w) \dd w\dd v.
\eal
The last term uses the energy estimate \eqref{eq:enestimimprov}
\bal\label{eq:s1-aux1}
	&\int_{\{f > \varphi_2\}} \int_{\{f < \varphi_0\}} \big(\varphi_1(w) - \varphi_1(v)\big)K(v, w) \dd w\dd v \\
	&\quad\leq C \mu^{-1}\int_{\{f > \varphi_2\}} \int_{\{f < \varphi_0\}} (f-\varphi_1)_+(v) \Abs{\varphi_1(w) - \varphi_1(v)}K(v, w) \dd w\dd v \\
	&\quad\leq  - \frac{C}{\mu\delta_1r_0^{n-d}}\int_{\{f > \varphi_2\}}\int_{\Sigma_{t, x}} (f-\varphi_1)_+(v) (f-\varphi_1)_-(w)\Abs{\varphi_1(w) - \varphi_1(v)}K(v, w) \dd w\dd v \\
	&\quad\leq   \frac{C}{\mu\delta_1r_0^{n-d}}\int_{\Sigma_{t, x}}\Bigg(\int (f-\varphi_1)^2_+(v) (f-\varphi_1)_-(w)K(v, w) \dd v\Bigg)^{\frac{1}{2}}\\
	&\qquad \qquad\qquad \times\Bigg(\int \Abs{\varphi_1(w) - \varphi_1(v)}^2K(v, w) \dd v\Bigg)^{\frac{1}{2}}\dd w \\
	&\quad\leq   \frac{C}{\delta_1r_0^{n-d}} \int_{\Sigma_{t, x}}\Bigg(\int (f-\varphi_1)^2_+(v) (f-\varphi_1)_-(w)K(v, w) \dd v\Bigg)^{\frac{1}{2}}\dd w \\
	&\quad\leq C \frac{\mu^{\frac{1}{2}}}{\delta_1r_0^{n-d}}\Bigg(\int\int (f-\varphi_1)_+(v) (f-\varphi_1)_-(w)K(v, w) \dd v\dd w\Bigg)^{\frac{1}{2}}\\
	&\quad\leq C\frac{\mu^{\frac{3}{2}}}{\delta_1}r_0^d.
\eal
Moreover, if in either variable we integrate over the set $\{\varphi_0 < f < \varphi_2\}$ we split into the singular and the non-singular part. Let $\gamma_2 > 0$ so that
\bals
	\int_{\{\varphi_0 < f < \varphi_2\}} &\int \big(\varphi_1(w) - \varphi_1(v)\big)K(v, w) \dd w\dd v \\
	&= \int_{\{\varphi_0 < f < \varphi_2\}} \int_{\abs{v-w} < \gamma_2} \big(\varphi_1(w) - \varphi_1(v)\big)K(v, w) \dd w\dd v \\
	&\quad + \int_{\{\varphi_0 < f < \varphi_2\}} \int_{\abs{v-w} > \gamma_2} \big(\varphi_1(w) - \varphi_1(v)\big)K(v, w) \dd w\dd v.
\eals

On the one hand, due to \eqref{eq:upperbound}
\bal\label{eq:s1-aux2}
	 \int_{\{\varphi_0 < f < \varphi_2\}} \int_{\abs{v-w} > \gamma_2} \big(\varphi_1(w) - \varphi_1(v)\big)&K(v, w) \dd w\dd v \\
	 &\leq C\mu \gamma_2^{1-2s}\abs{\{\varphi_0 < f < \varphi_2\} \cap B_{3r_0}}.
\eal
Similarly, we get
\bal\label{eq:s1-aux3}
	 \int_{\{\varphi_2 < f\}} \int_{\{\varphi_0 < f < \varphi_2\} \cap \abs{v-w} > \gamma_2} \big(\varphi_1(w) - \varphi_1(v)\big)&K(v, w) \dd w\dd v \\
	 &\leq C\mu \gamma_2^{1-2s}\abs{\{\varphi_0 < f < \varphi_2\} \cap B_{5r_0}}.
\eal

On the other hand, we distinguish the symmetric from the anti-symmetric part. We start with 
\bals
	\int_{\{f > \varphi_2\}} &\int_{\{\varphi_0 < f < \varphi_2\}\cap \abs{v-w} < \gamma_2} \big(\varphi_1(w) - \varphi_1(v)\big)K(v, w) \dd w\dd v \\
	&= \int_{\{f > \varphi_2\}} \int_{\{\varphi_0 < f < \varphi_2\}\cap\abs{v-w} < \gamma_2} \big(\varphi_1(w) - \varphi_1(v)\big)\left(\frac{K(v, w)+K(w, v)}{2}\right) \dd w\dd v \\
	&\quad + \int_{\{f > \varphi_2\}} \int_{\{\varphi_0 < f < \varphi_2\}\cap\abs{v-w} < \gamma_2} \big(\varphi_1(w) - \varphi_1(v)\big)\left(\frac{K(v, w)-K(w, v)}{2}\right) \dd w\dd v.
\eals
Then we have for the symmetric part using the energy estimate \eqref{eq:enestimimprov} for $(f-\varphi_1)_+$, the upper bound \eqref{eq:upperbound} and the definition of $\varphi_1$ \eqref{eq:barriers},
\bal\label{eq:s1-aux4}
	\int_{\{f > \varphi_2\}} &\int_{\{\varphi_0 < f < \varphi_2\}\cap\abs{v-w} < \gamma_2} \big(\varphi_1(w) - \varphi_1(v)\big)\left(\frac{K(v, w)+K(w, v)}{2}\right) \dd w\dd v \\
	&\leq C\frac{1}{\mu} \int_{\{f > \varphi_2\}} \int_{\{\varphi_0 < f < \varphi_2\}\cap\abs{v-w} < \gamma_2} \big(\varphi_1(w) - \varphi_1(v)\big)(f-\varphi_1)_+(v)\\
	&\qquad \qquad \qquad \times \left(\frac{K(v, w)+K(w, v)}{2}\right) \dd w\dd v \\
	&\leq C\frac{1}{\mu} \int \int_{\abs{v-w} < \gamma_2} \big(\varphi_1(w) - \varphi_1(v)\big)\big((f-\varphi_1)_+(v) - (f-\varphi_1)_+(w)\big)K(v, w)\dd w\dd v \\
	&\leq C\frac{1}{\mu} \int \Bigg(\int_{\abs{v-w} < \gamma_2} \big(\varphi_1(w) - \varphi_1(v)\big)^2K(v, w)\dd w\Bigg)^{\frac{1}{2}}\\
	&\qquad\qquad\qquad\times \Bigg(\int \big((f-\varphi_1)_+(v) - (f-\varphi_1)_+(w)\big)^2K(v, w)\dd w\Bigg)^{\frac{1}{2}}\dd v \\
	&\leq C\gamma_2^{1-s} \mu.
\eal
The anti-symmetric part gives due to \eqref{eq:cancellation2} in case that $s \in [1/2, 1)$ and due to \eqref{eq:upperbound} in case that $s \in (0, 1/2)$
\bal\label{eq:s1-aux5}
	\int_{\{f > \varphi_2\}} &\int_{\{\varphi_0 < f < \varphi_2\}\cap\abs{v-w} < \gamma_2} \big(\varphi_1(w) - \varphi_1(v)\big)\left(\frac{K(v, w)-K(w, v)}{2}\right) \dd w\dd v \\
	&\leq \int_{\{\varphi_0 < f < \varphi_2\}\cap\abs{v-w} < \gamma_2} \int_{\{f > \varphi_2\}} \big(\varphi_1(w) - \varphi_1(v)\big)\left(K(v, w)-K(w, v)\right) \dd v\dd w \\ 
	&\leq C\mu \int_{\{\varphi_0 < f < \varphi_2\}\cap\abs{v-w} < \gamma_2} \Bigg\lvert\int_{\{f > \varphi_2\}} (v-w)\cdot\left(K(v, w)-K(w, v)\right) \dd v\Bigg\rvert\dd w \\ 
	&\leq C\mu \gamma_2^{1-2s} \abs{\{\varphi_0 < f < \varphi_2\}}.
\eal

Similarly
\bals
	\int_{\{\varphi_1 < f < \varphi_2\}} &\int_{\abs{v-w} < \gamma_2} \big(\varphi_1(w) - \varphi_1(v)\big)K(v, w) \dd w\dd v \\
	&= \int_{\{\varphi_1 < f < \varphi_2\}} \int_{\abs{v-w} < \gamma_2} \big(\varphi_1(w) - \varphi_1(v)\big)\left(\frac{K(v, w)+K(w, v)}{2}\right) \dd w\dd v \\
	&\quad + \int_{\{\varphi_1 < f < \varphi_2\}} \int_{\abs{v-w} < \gamma_2} \big(\varphi_1(w) - \varphi_1(v)\big)\left(\frac{K(v, w)-K(w, v)}{2}\right) \dd w\dd v.
\eals
Then the energy estimate \eqref{eq:enestimimprov} for $(f-\varphi_0)_+$ yields
\bal\label{eq:s1-aux6}
	&\int_{\{\varphi_1 < f < \varphi_2\}} \int_{\abs{v-w} < \gamma_2} \big(\varphi_1(w) - \varphi_1(v)\big)\left(\frac{K(v, w)+K(w, v)}{2}\right) \dd w\dd v \\
	&\quad\leq C\int_{\{\varphi_1 < f < \varphi_2\}} \int_{\abs{v-w} < \gamma_2} \big(\varphi_1(w) - \varphi_1(v)\big)(f-\varphi_0)_+(v)\left(\frac{K(v, w)+K(w, v)}{2}\right) \dd w\dd v \\
	&\quad\leq C \int \int_{\abs{v-w} < \gamma_2} \big(\varphi_1(w) - \varphi_1(v)\big)\big((f-\varphi_0)_+(v) - (f-\varphi_0)_+(w)\big)K(v, w)\dd w\dd v \\
	&\quad\leq C\int \Bigg(\int_{\abs{v-w} < \gamma_2} \big(\varphi_1(w) - \varphi_1(v)\big)^2K(v, w)\dd w\Bigg)^{\frac{1}{2}}\\
	&\qquad\qquad\qquad\qquad\times\Bigg(\int \big((f-\varphi_0)_+(v) - (f-\varphi_0)_+(w)\big)^2K(v, w)\dd w\Bigg)^{\frac{1}{2}}\dd v \\
	&\quad\leq C\gamma_2^{1-s} \mu.
\eal
Finally the anti-symmetric part gives again with \eqref{eq:cancellation2} in case that $s \in [1/2, 1)$ or with \eqref{eq:upperbound} in case that $s \in (0, 1/2)$
\bal\label{eq:s1-aux7}
	\int_{\{\varphi_1 < f < \varphi_2\}} \int_{\abs{v-w} < \gamma_2} \big(\varphi_1(w) - \varphi_1(v)\big)\left(\frac{K(v, w)-K(w, v)}{2}\right)& \dd w\dd v \\
	&\leq C\mu \gamma_2^{1-2s} \abs{\varphi_0 < f < \varphi_2}.
\eal

We conclude from \eqref{eq:s1-aux0}-\eqref{eq:s1-aux7} by choosing $\gamma_2^{1-s} = \mu^{\frac{1}{2}}$
\bal\label{eq:s1}
	\int_{Q_{3r_0}} \mathcal L\varphi_1(v)\chi_{f > \varphi_1}(v) \dd z  \leq C \Abs{\{\varphi_0 < f < \varphi_2\} \cap Q_{3r_0}} + \frac{\mu^{\frac{3}{2}}}{\delta_1}r_0^n +  \mu^{\frac{3}{2}} r_0^n.
\eal

The second part of the source term $\tilde h$ in \eqref{eq:hypopc} satisfies by assumption
\bal\label{eq:s2}
	\int_{Q_{3r_0}} h \chi_{f > \varphi_1}(v) \dd z &\leq C \mu^2 r_0^n.
\eal
Gathering these observations in \eqref{eq:s1} and \eqref{eq:s2} we find 
\bal
	\norm{\tilde h}_{L^1(Q_{3r_0})} &\leq C r_0^{n}\frac{\mu^{\frac{3}{2}}}{\delta_1} + Cr_0^{n} \mu^2 + C \abs{\{\varphi_0 < f < \varphi_2\} \cap Q_{3r_0}}^{\frac{1}{2}}\\
	&\leq Cr_0^{n} \frac{\mu^{\frac{3}{2}}}{\delta_1} + C \abs{\varphi_0 < f < \varphi_2\} \cap Q_{3r_0}}^{\frac{1}{2}}.
\label{eq:up4}
\eal

\textit{Step 4: Conclusion.} 
Now we assemble all the pieces. Recall that $n = 2s + d(2+2s)$. Then \eqref{eq:hypopc} gives with \eqref{eq:3.10}, \eqref{eq:up1}, \eqref{eq:up1.5}, \eqref{eq:up2}, \eqref{eq:auxskew}, \eqref{eq:spatial-grad} and \eqref{eq:up4}
\bals	
	\big(\mu \delta_1 - \mu^2\big)\delta_2 r_0^n &\leq  \frac{C}{\varepsilon^{d+2}} \left[r_0^{\frac{n}{2}}\abs{\{\varphi_0 < f < \varphi_2\}\cap (-3, 0] \times Q^t_{\frac{1}{2}}}^{\frac{1}{2}}+Cr_0^n \frac{\mu^{\frac{3}{2}}}{\delta_1}\right] \\
	&\quad+ C\varepsilon^\sigma \mu r_0^n  + C \abs{\{\varphi_0 < f < \varphi_2\} \cap Q_{5r_0}} + C\frac{\mu^2}{\delta_1\varepsilon^d}r_0^n + C\frac{\mu^{\frac{3}{2}}}{\delta_1}r_0^n.
\eals
Thus
\bals
	\big(\delta_1 - \mu\big)\delta_2 &\leq  C \frac{1}{\varepsilon^{d+2}} \left[\mu^{-1}r_0^{-{\frac{n}{2}}}\abs{\{\varphi_0 < f < \varphi_2\}\cap(-3, 0] \times Q^t_{\frac{1}{2}}}^{\frac{1}{2}} +  \frac{\mu^{\frac{1}{2}}}{\delta_1} \right]\\
	&\qquad+  C\varepsilon^\sigma +C\frac{\mu^{\frac{1}{2}}}{\delta_1} + C\frac{\mu}{\delta_1\varepsilon^d}.
\eals
If we choose $\varepsilon$ as
\beqs
	\varepsilon = \left(\frac{\delta_1\delta_2}{4C}\right)^\frac{1}{\sigma}
\eeqs
then we know that
\beqs
	C\varepsilon^\sigma \leq \frac{\delta_1\delta_2}{4}.
\eeqs
And if we choose $\mu$ as
\bals
	\mu = \left(\frac{\delta_1\delta_2}{4C\left(\delta_2 + \delta_1^{-1} + \varepsilon^{-(d+2)}\delta_1^{-1}\right)}\right)^{2} \sim \delta_1^{6d+16}\delta_2^{6d+14},
\eals
then we have assured that
\bals
	C\frac{\mu^{\frac{1}{2}}}{\varepsilon^{d+2}\delta_1} + C\frac{\mu^{\frac{1}{2}}}{\delta_1} + \delta_2 \mu  + \frac{\mu}{\delta_1\varepsilon^d}\leq \frac{\delta_1\delta_2}{4}.
\eals
Hence
\beqs
	\abs{\{\varphi_0 < f < \varphi_2\}\cap(-3, 0] \times Q^t_{\frac{1}{2}}}^{\frac{1}{2}} \geq \frac{\delta_1\delta_2}{4}\varepsilon^{d+2}r_0^{\frac{n}{2}}\mu \gtrsim (\delta_1\delta_2)^{9d+23} r_0^{\frac{n}{2}}.
\eeqs
\end{proof}

\subsection{Measure-to-Pointwise estimate}
\begin{lemma}[Measure-to-pointwise Lemma]\label{lem:16}
Let $\delta \in (0, 1)$ and $r_0 < \frac{1}{3}$. Then there is $\theta = \theta(\delta) \sim\delta^{2(1 + \delta^{-(18d+46)})} > 0$ such that any sub-solution $f$ of \eqref{eq:1.1}-\eqref{eq:1.2} in $[-3, 0]\times Q^t_1$ with \eqref{eq:coercivity}-\eqref{eq:cancellation2} for $\bar R = 2$, with a source term $h$ such that $\norm{h}_{L^\infty(Q_{1})} \leq \theta$ and so that $f \leq 1$ in $(-3, 0]\times B_{2^{-(1+2s)}}\times \R^d$ and
\beq
	\abs{\{f\leq 0\}\cap Q^-_{r_0}} \geq \delta \abs{Q^-_{r_0}}
\label{eq:3.15}
\eeq
satisfies 
\beqs
	f \leq 1 - \theta \quad \textrm{ in } Q_{\frac{r_0}{2}},
\eeqs
where $Q^-_{r_0} := Q_{r_0}(-2r_0^{2s}, 0, 0)$ and $[-3, 0]\times Q^t_1 = [-3, 0]\times B_1\times B_1$ as before.
\end{lemma}
\begin{proof}
By Lemma \ref{lem:6.6} we know that there is a $\delta_0 > 0$ such that for any $r > 0$ any sub-solution $f$ on $Q_{2r}$ satisfying $\int_{Q_r}f^2_+ \leq \delta_0\abs{Q_r}$ there holds $f \leq \frac{1}{2}$ in $Q_{\frac{r}{2}}$. Therefore, we define $\nu, \mu > 0$ as in the proof of Theorem \ref{thm:IVL} with $\delta_1 = \delta$ and $\delta_2 = \delta_0$. 

We consider $f_k := \mu^{-2k}\big[{f- (1 - \mu^{2k})}\big]$ for $k \geq 0$. Note that $f_k$ is a sub-solution to \eqref{eq:1.1}-\eqref{eq:1.2} for each $k\geq 0$ with \eqref{eq:upperbound-2}-\eqref{eq:cancellation2} and with a source term of $L^\infty$ norm less than $\theta$ as long as $k \leq 1 + \frac{1}{\nu}$ since by assumption $\norm{h}_{L^\infty} \leq\theta$ so that $\norm{h}_{L^\infty} \leq \mu^{2(1+\frac{1}{\nu})}$. Moreover, $\{\varphi_0 < f_i < \varphi_2\} \cap \{\varphi_0 < f_j < \varphi_2\} = \emptyset$ for all $i \neq j$ and each $f_k$ satisfies \eqref{eq:3.15}.

In case that $\int_{Q_{r_0}} (f_k)^2_+ \leq \delta_0 \abs{Q_{r_0}}$ there holds $f_k \leq \frac{1}{2}$ in $Q_{\frac{r_0}{2}}$, and we conclude $f \leq 1 - \theta$ with $\theta = \frac{\mu^{2k}}{2}$. Thus it suffices to consider $1 \leq k_0\leq 1 + \frac{1}{\nu}$ such that $\int_{Q_{r_0}} (f_k)^2_+ > \delta_0 \abs{Q_{r_0}}$ for any $0 \leq k \leq k_0$. Then for $0 \leq k \leq k_0 - 1$ there holds
\bals
	&\abs{\{f_k \geq 1 - \mu^2\}\cap Q_{r_0}} = \abs{\{f_{k+1} \geq 0\}\cap Q_{r_0}} \geq \int_{Q_{r_0}}(f_{k+1})^2_+ > \delta_0\abs{Q_{r_0}}, \\
	&\abs{\{f_k \leq 0\}\cap Q^-_{r_0}} \geq \abs{\{f \leq 0\}\cap Q^-_{r_0}} \geq \delta \abs{Q^-_{r_0}},
\eals
where we used the fact that $f \leq 0$ implies $f_k \leq 0$ for all $k \geq 0$. Thus we can apply Theorem \ref{thm:IVL} for our choice of $\delta_1, \delta_2$ so that we get
\beqs
	\abs{\{\varphi_0 < f_k < \varphi_2\}\cap (-3, 0]\times Q^t_{\frac{1}{2}}} \geq \nu\abs{(-3, 0] \times Q^t_{\frac{1}{2}}}.
\eeqs
Since these sets are all disjoint, we get by summing these estimates
\beqs
	\abs{(-3, 0]\times Q^t_{\frac{1}{2}}}\geq \sum^{k_0-1}_{k=0}\abs{\{\varphi_0 < f_k < \varphi_2\}\cap (-3, 0]\times Q^t_{\frac{1}{2}}} \geq k_0\nu\abs{(-3, 0]\times Q^t_{\frac{1}{2}}}.
\eeqs
Thus $k_0 \leq \frac{1}{\nu}$ and we deduce with $f_{k_0 + 1} \leq \frac{1}{2}$ in $Q_{\frac{r_0}{2}}$
\beqs
	f \leq 1 - \frac{\mu^{2k_0 + 2}}{2} \leq 1 - \frac{\mu^{\frac{2 + 2\nu}{\nu}}}{2} \quad \textrm{ in } Q_{\frac{r_0}{2}}.
\eeqs
This yields the claim for $\theta(\delta) := \frac{\mu^{\frac{2 + 2\nu}{\nu}}}{2} \sim \delta^{2(1 + \delta^{-(18d+46)})}$.
\end{proof}

\section{Harnack Inequalities and Hölder Continuity}
\label{sec:6}
\subsection{Harnack Inequalities}
\label{subsec:harnack}
We follow Section 4.1 in \cite{JGCM}. We consider a non-negative super-solution $f$ to \eqref{eq:1.1}-\eqref{eq:1.2} for $h = 0$ on $(-3, 0] \times Q^t_1$ so that \eqref{eq:coercivity}-\eqref{eq:cancellation2} holds for $\bar R = 2$. Then Lemma \ref{lem:16} applied to the sub-solution $g := 1 - \frac{f}{M}$ implies for any $\delta \in (0, 1)$ and $M \sim \delta^{-2(1+\delta^{-(18d+46)})}$ that $\forall Q_r(z) \subset (-3, 0] \times Q^t_1 \textrm{ so that } Q^+_{\frac{r}{2}}(z) \subset (-3, 0] \times Q^t_1$
\beq
	 \frac{\abs{\{f > M\}\cap Q_r(z)}}{\abs{Q_r(z)}} > \delta \implies \inf_{Q^+_{\frac{r}{2}}(z)} f \geq 1,
\label{eq:4.1}	
\eeq
where we recall $Q^+_r(z) = Q_r(z + (2 r^{2s}, 2 r^{2s}v, 0))$ for $z = (t, x, v)$. This implies using the layer-cake representation that if $\inf_{Q_{\frac{r}{2}}} f < 1$, then
\beq
	\frac{\abs{\{f > M\}\cap Q^-_r}}{\abs{Q^-_r}} \lesssim \delta(M) = \Bigg(\frac{1}{\mathrm{ln}(1 + M)}\Bigg)^{\frac{1}{(18d+47)}} \implies \int_{Q^-_{r}} \big(\mathrm{ln}(1 + f)\big)^{\frac{1}{(18d+48)}} \lesssim 1.
\label{eq:4.2}
\eeq
We can improve this logarithmic integrability as follows. We pick $r_0 < \frac{1}{3}$ and consider the sequence of cylinders
\beqs
	\mathcal Q^k := Q_{\frac{r_0}{2}+\alpha_k}\Bigg(-\frac{5}{2}r_0^{2s} + \frac{1}{2}\big(\frac{r_0}{2}+\alpha_k\big)^{2s}, 0, 0\Bigg) \quad \textrm{where} \quad \alpha_k := \frac{r_0}{2}7^{1-k}.
\eeqs
Then $\mathcal Q^1 = Q_{r_0}^-$ and $\mathcal Q^k \to \tilde Q_{\frac{r_0}{2}}^- := Q_{\frac{r_0}{2}}\big((-\frac{5}{2} +\frac{1}{2\cdot 2^{2s}})r_0^{2s}, 0, 0)\big)$ as $k\to\infty$. Note that these cylinders satisfy $\tilde Q_{\frac{r_0}{2}}^-\subset \mathcal Q^k \subset \bar{\mathcal Q}^k \subset \mathring{\mathcal Q}^{k-1} \subset Q^-_{r_0}$ for all $k \geq 1$. We claim as in \cite{JGCM} that for $\delta_0 > 0$ to be determined, for any non-negative super-solution $f$ with $\inf_{Q_{\frac{r}{2}}} f < 1$, there holds
\beq
	\forall k \geq 1, \quad \frac{\abs{\{f \geq M^k\}\cap \mathcal Q^k}}{\abs{\mathcal Q^k}} \leq \delta_0\Big(\frac{1}{35m}\Big)^{(2d + 2s(d+1))k},
\label{eq:4.3}
\eeq
where $M \sim \delta^{-2(1 + \delta^{-(18d+46)})}$ with $\delta := \delta_0\Big(\frac{1}{35m}\Big)^{2d + 2s(d+1)}$ and $m\geq 3$. If we can show that \eqref{eq:4.3} holds, then we deduce with the layer-cake representation that there exists $\zeta \gtrsim \delta_0^{(18d+47)} > 0$ such that $\int_{\tilde Q^-_{\frac{r_0}{2}}} f^\zeta\dd z \lesssim 1$, which in turn yields by linearity
\beq
	\Bigg(\int_{\tilde Q^-_{\frac{r_0}{2}}} f^\zeta(z)\dd z\Bigg)^{\frac{1}{\zeta}} \lesssim \inf_{Q_{\frac{r_0}{2}}} f. 
\label{eq:weakH_aux}
\eeq
We can now derive a non-linear Harnack inequality using De Giorgi's first Lemma \ref{lem:6.6}.  Indeed, assume $f$ is a solution of \eqref{eq:1.1} such that $0 \leq f \leq 1$ a.e. and write 
\beqs
	\inf_{Q_{\frac{r_0}{2}}} f =: \varepsilon_1 \leq 1.
\eeqs
Then due to \eqref{eq:weakH_aux} we have
\beqs
	\Bigg(\int_{\tilde Q^-_{\frac{r_0}{2}}} f^2(z)\dd z\Bigg)^{\frac{1}{\zeta}} = \Bigg(\int_{\tilde Q^-_{\frac{r_0}{2}}} f^\zeta(z) f^{2-\zeta}(z)\dd z\Bigg)^{\frac{1}{\zeta}} \leq \Bigg(\int_{\tilde Q^-_{\frac{r_0}{2}}} f^\zeta(z)\dd z\Bigg)^{\frac{1}{\zeta}} \lesssim \inf_{Q_{\frac{r_0}{2}}} f = \varepsilon_1.
\eeqs
Thus by Lemma \ref{lem:6.6} with $\varepsilon_0 = \varepsilon_1^\zeta$ 
\beq\label{eq:strongH_aux}
	\sup_{\tilde Q^-_{\frac{r_0}{4}}} f \leq C \varepsilon_1^{\zeta\beta_2} = C\Big(\inf_{Q_{\frac{r_0}{2}}} f\Big)^{\zeta\beta_2},
\eeq
where $\beta_2 = \frac{p-2}{2(2p-2)} \in (0, 1)$ and $\tilde Q^-_{\frac{r_0}{4}} = Q_{\frac{r_0}{4}}\big((-\frac{5}{2} +\frac{1}{2\cdot 2^{2s}})r_0^{2s}, 0, 0)\big)$.

This implies the Weak \eqref{eq:weakH} and Not-so-Strong \eqref{eq:strongH} Harnack inequality (with $\beta = \zeta \beta_2$) for any non-negative $\tilde f$ by applying the previous estimate to $f = \tilde f +(1+t)\norm{h}_{L^\infty}$.

We now prove \eqref{eq:4.3} inductively. The case $k = 1$ holds by \eqref{eq:4.2}. We define
\beqs
	A_{k+1} := \{f > M^{k+1}\}\cap\mathcal Q^{k+1}
\eeqs
and denote 
\beqs
	\mathfrak C_r[z] := z \circ Q_{2r}\Bigg(\Big(\frac{1}{2}(2r)^{2s}, 0, 0\Big)\Bigg) = z \circ \Bigg(-\frac{1}{2}(2r)^{2s}, \frac{1}{2}(2r)^{2s}\Bigg] \times B_{(2r)^{1+2s}} \times B_{2r}. 
\eeqs
Just as in \cite{JGCM} we construct $z_l = (t_l, x_l, v_l) \in \mathcal Q^{k+1}$ and $r_l > 0$, $l\geq 1$, $m \geq 3$, $n \geq 1$ so that
\bals
	&\forall l \geq 1, r_l \in \Big(0, \frac{\alpha_{k+1}}{5 m\cdot n}\Big), \\
	&\forall l \geq 1, \abs{A_{k+1}\cap\mathfrak C_{5 m \cdot r_l}[z_l]}\leq \delta_0\abs{\mathfrak C_{5 m \cdot r_l}[z_l]},\\
	&\forall l \geq 1, \abs{A_{k+1}\cap\mathfrak C_{r_l}[z_l]} > \delta_0\abs{\mathfrak C_{r_l}[z_l]},\\
	&\mathfrak C_{m\cdot r_l}, l \geq 1, \textrm{are disjoint cylinders}, \\
	&A_{k+1} \textrm{ is covered by the family } \big(\mathfrak C_{5m\cdot r_l}[z_l]\big)_{l\geq 1}.
\eals
For these cylinders, we have $$\mathfrak C^+_{r_l}[z_l] := z_l \circ Q_{2r_l}^+\Big((\frac{1}{2}(2r_l)^{2s}, 0, 0)\Big) = z_l \circ \Big(\frac{3}{2}(2r_l)^{2s}, \frac{5}{2}(2r_l)^{2s}\Big] \times B_{(2r_l)^{1 + 2s}} \times B_{2r_l}.$$ We remark that for $m \geq 5^{\frac{1}{2s}}$, there holds $\mathfrak C_{r_l}[z_l]^+ \subset \mathfrak C_{m \cdot r_l}[z_l]$. 
Moreover, we pick $n \geq 1$ so that $\forall k \geq 1$
\beqs
	\big(7^{k} + 7\big)^{2s}-\big(7^k+1\big)^{2s} \geq \Big(\frac{2}{n}\Big)^{2s}.
\eeqs
Then by choice of $z_l$ and $r_l$ we have that $\mathfrak C_{5m\cdot r_l}[z_l] \subset \mathcal Q^k$. Note that for example, if $s \geq \frac{1}{2}$ we can choose $n =1$. 

As in \cite{JGCM}, we prove that the family of cylinders $\mathfrak C_r[z]$ with $z \in \mathcal Q^{k+1}$, $r \in \big(0, \frac{\alpha_{k+1}}{5 m\cdot n}\big)$, so that $\abs{A_{k+1}\cap\mathfrak C_{5m\cdot r}[z]}\leq \delta_0\abs{\mathfrak C_{5m\cdot r}[z]}$ and $\abs{A_{k+1}\cap\mathfrak C_{r}[z]} > \delta_0\abs{\mathfrak C_{r}[z]}$ cover $A_{k+1}$. 
By induction hypothesis there holds for all $r \in \big(\frac{\alpha_{k+1}}{5m\cdot n}, \alpha_{k+1}\big)$
\beq
	\abs{A_{k+1}\cap \mathfrak C_r[z]}\leq \abs{A_{k}\cap \mathfrak C_r[z]} \leq \abs{A_{k}\cap \mathcal Q^k}\leq \delta_0\Big(\frac{1}{35m}\Big)^{(2d + 2s(d+1))k}\abs{\mathcal Q^k} \leq \delta_0\abs{\mathcal C_r[z]}.
\label{eq:4.4}
\eeq
If $z \in A_{k+1}$ is not covered by $\mathcal F$ then the continuous positive function $\varphi(r) = \frac{\abs{A_{k+1}\cap \mathfrak C_r[z]}}{\abs{\mathfrak C_r[z]}}$ on $(0, +\infty)$ satisfies $\varphi(r) \leq \delta_0$ or $\varphi(5m\cdot n\cdot r) > \delta_0$ for all $r \in \big(0, \frac{\alpha_{k+1}}{5m\cdot n}\big)$. By \eqref{eq:4.4} and by continuity, we must have $\varphi(r) \leq \delta_0$ for all $r \in \big(0, \frac{\alpha_{k+1}}{5m\cdot n}\big)$. Lebesgue's differentiation theorem implies in the limit $r\to 0$ there holds $z \notin A_{k+1}$. Thus $\mathcal F$ covers $A_{k+1}$. 

In particular, $A_{k+1}$ is covered by the family $\mathcal F'$ of cylinders $\mathfrak C_{m\cdot r}[z]$ with $z \in \mathcal Q^{k+1}$, $r \in \big(0, \frac{\alpha_{k+1}}{5m\cdot n}\big)$, so that $\abs{A_{k+1}\cap\mathfrak C_{5m\cdot r}[z]}\leq \delta_0\abs{\mathfrak C_{5m\cdot r}[z]}$ and $\abs{A_{k+1}\cap\mathfrak C_{r}[z]} > \delta_0\abs{\mathfrak C_{r}[z]}$. Vitali's covering lemma gives us the existence of a countable sub-family, denoted $\big(\mathfrak C_{r_l}[z_l]\big)_{l\geq 1}$, such that $\big(\mathfrak C_{5m\cdot r_l}[z_l]\big)_{l\geq 1}$ covers $A_{k+1}$ and the $(\mathfrak C_{m\cdot r_l}[z_l])_{l\geq 1}$ are disjoint. 

This finishes the construction of the covering with the desired properties. We can then apply Lemma \ref{lem:16} to each $\mathfrak C_{r_l}[z_l]$ to obtain $\mathfrak C_{r_l}[z_l]^+\subset A_k$ and the $(\mathfrak C_{r_l}[z_l]^+)_{l\geq 1}$ are disjoint since $\mathfrak C_{r_l}[z_l]^+ \subset \mathfrak C_{m\cdot r_l}[z_l]$. This yields for the right choice of $\delta_0$, i.e. $\delta_0 \leq \Big(\frac{4}{1225m^2}\Big)^{2d + 2s(d+1)}$,
\bals
	\abs{A_{k+1}} &\leq \sum_{l\geq 1} \abs{A_{k+1} \cap \mathfrak C_{5m\cdot r_l}[z_l]}\leq\delta_0\sum_{l\geq 1}\abs{C_{5m\cdot r_l}[z_l]}\\
	&\leq (5m)^{2d(s+1)+2s}\delta_0\sum_{l\geq 1}\abs{\mathfrak C_{r_l}[z_l]} \leq (5m)^{2d(s+1)+2s} \delta_0\sum_{l\geq 1}\abs{\mathfrak C_{r_l}[z_l]^+} \\
	&\leq (5m)^{2d(s+1)+2s}\delta_0\abs{A_k}\leq (5m)^{2d(s+1)+2s}\delta_0^2\Big(\frac{1}{35m}\Big)^{(2d + 2s(d+1))k}\abs{\mathcal Q^k}\\
	&\leq \delta_0\Big(\frac{1}{35m}\Big)^{(2d + 2s(d+1))(k+1)}\abs{\mathcal Q^{k+1}}.
\eals
This proves \eqref{eq:4.3} and thus concludes the proof.

\subsection{Hölder Continuity}
\label{subsec:hölder}
Obtaining \eqref{eq:holder} is a standard argument, see \cite{Vas} and \cite{JGCM}. Let $f$ be a weak solution to \eqref{eq:1.1}-\eqref{eq:1.2}. For given $r_0 < \frac{1}{3}$ small enough, Lemma \ref{lem:16} implies
\beq
	\underset{Q_{r_0}}{\mathrm{osc }} f \leq \left(1 -\frac{\theta}{2}\right) \max \Bigg\{\underset{Q_1}{\mathrm{osc}} f, e^{2(3+2^{18d + 46})}\norm{h}_{L^\infty(Q_2)}\Bigg\}.
\label{eq:4.5}
\eeq
To see this we apply Lemma \ref{lem:16} rescaled to $Q_2$ to either $F$ or $-F$ depending on which of these satisfy \eqref{eq:3.15}, where $F$ is defined by
\beqs
	F := \frac{2f - (\sup_{Q_1} f + \inf_{Q_1} f)}{\max \Big\{\underset{Q_1}{\mathrm{osc}} f, e^{2(3+2^{18d + 46})}\norm{h}_{L^\infty(Q_2)}\Big\}}.
\eeqs
We want to prove $\forall z_0 \in Q_1, \forall r \in (0, r_0)$,
\beq
	\underset{Q_r(z_0)}{\textrm{osc} }f \leq r^\alpha e^{2(3+2^{18d + 46})}\big(1 +\norm{h}_{L^\infty(Q_2)}\big)\max\Big\{\underset{Q_1}{\mathrm{osc}} f, e^{2(3+2^{18d+46})}\norm{h}_{L^\infty(Q_2)}\Big\},
\label{eq:4.6}
\eeq
where we chose $\alpha \in (0, 1)$ so that $1 - \frac{\theta}{2} = r_0^\alpha$. From there we can deduce Hölder regularity by choosing $z, z' \in Q_1$ so that $\abs{z - z'} \leq r_0$.

In order to prove \eqref{eq:4.6} we proceed iteratively. We define for $n \in \N\setminus\{0\}$ a sequence of solutions to \eqref{eq:1.1}-\eqref{eq:1.2} in $Q_1$ by
\beqs
	f_n(s, y, w) := \frac{2\big(1 - \frac{\theta}{2}\big)^{1-n} f\big(t_0+r_0^{2sn}s, x_0 - r_0^{2sn}sv_0 + r_0^{(1 +2s)n}y, v_0 + r_0^nw\big)}{ \max \Big\{\underset{Q_1}{\mathrm{osc}} f, e^{2(3+2^{18d+46})}\norm{h}_{L^\infty(Q_2)}\Big\}}.
\eeqs
We prove by induction on $n \geq 1$ that $\textrm{osc}_{Q_1} f_n \leq 2e^{2(3+2^{18d+46})}\big(1 +\norm{h}_{L^\infty(Q_2)}\big)$. The case $n = 1$ is clear by definition of $f_1$. Using \eqref{eq:4.5} and the induction hypothesis, we get
\bals
	\underset{Q_1}{\textrm{osc}} f_n &= \left(1 - \frac{\theta}{2}\right)^{-1}\underset{Q_{r_0}}{\textrm{osc}} f_{n-1} \\
	&\leq\max\left\{\underset{Q_1}{\mathrm{osc}} f_{n-1}, e^{2(3+2^{18d+46})}\norm{h}_{L^\infty(Q_2)}\right\} \\
	&\leq 2e^{2(3+2^{18d+46})}\big(1 +\norm{h}_{L^\infty(Q_2)}\big).
\eals
This implies \eqref{eq:4.6} if we choose $\alpha \in (0, 1)$ so that $1 - \frac{\theta}{2} = r_0^\alpha$. Indeed for $z_0 \in Q_1$ and $r \in (0, r_0)$ we write $r = r_0^n$ for some $n \in \N\setminus\{0\}$. Then
\bals
	 \underset{Q_r(z_0)}{\mathrm{osc }} f &= \frac{1}{2} r_0^{\alpha(n-1)}\max\left\{\underset{Q_1}{\mathrm{osc}} f, e^{2(3+2^{18d+46})}\norm{h}_{L^\infty(Q_2)}\right\}\underset{Q_r(z_0)}{\textrm{osc }} f_n \\
	 &= \frac{1}{2} r_0^{\alpha(n-1)}r^\alpha\max\left\{\underset{Q_1}{\mathrm{osc}} f, e^{2(3+2^{18d+46})}\norm{h}_{L^\infty(Q_2)}\right\}\underset{Q_1}{\textrm{osc }} f_{n+1} \\
	 &\leq r_0^{\alpha n}e^{2(3+2^{18d+46})}\left(1 +\norm{h}_{L^\infty(Q_2)}\right)\max\left\{\underset{Q_1}{\mathrm{osc}} f, e^{2(3+2^{18d+46})}\norm{h}_{L^\infty(Q_2)}\right\} \\
	 &= r^{\alpha}e^{2(3+2^{18d+46})}\left(1 +\norm{h}_{L^\infty(Q_2)}\right)\max\left\{\underset{Q_1}{\mathrm{osc}} f, e^{2(3+2^{18d+46})}\norm{h}_{L^\infty(Q_2)}\right\}.
\eals

\appendix
\section{Coercivity estimate for the non-cutoff Boltzmann Kernel}\label{appendix}
We consider a non-negative function $f$ with mass bounded from above and below, energy and entropy bounded from above. We denote with $K_f$ the non-negative kernel for the non-cutoff Boltzmann equation. It is determined by \eqref{eq:boltzmann_kernel} and can be expressed as 
\beqs
	K_f(v, v') \approx \Bigg(\int_{w \cdot (v-v') = 0} f(v+w) \abs{w}^{\gamma+1+2s} \dd w\Bigg)\abs{v-v'}^{-(d+2s)},
\eeqs 
with non-locality parameter $s \in (0, 1)$ and $\gamma \in (-d, 1]$, see \cite{IS59}. In \cite{IS59}, Silvestre determines a cone of directions with vertex $v$ in which the Boltzmann kernel is bounded below, see also \cite[Appendix]{IS}. We denote with $A = A(v) \in \partial B_1$ the set of directions where $K_f(v, v') \geq \lambda (1+\abs{v}^{1+2s+\gamma})\abs{v-v'}^{-(d+2s)}$ for all $v'$ such that $\frac{(v-v')}{\abs{v-v'}} \in A$. The existence of $A$ and further properties are stated in \cite[Lemma A.3]{IS}. Consistent with the notation of \cite{IS} we denote with $\Xi(v)$ the symmetric cone of values of $v'$ such that $\frac{(v-v')}{\abs{v-v'}} \in A$. Then for any $v' \in \Xi(v)$ there holds $K_f(v, v') \geq \lambda (1+\abs{v}^{1+2s+\gamma})\abs{v-v'}^{-(d+2s)}$. In particular, \cite[Lemma A.3]{IS} states a universal lower bound on the density of $\Xi(v)$. This can be used to prove a coercivity estimate \cite[Lemma A.6]{IS}
\beqs
	\int\int\frac{\abs{g(v) - g(v')}^2}{\abs{v-v'}^{d+2s}}\dd v' \dd v \leq C\int \int (g(v) - g(v'))^2K(v, v') \dd v'\dd v
\eeqs
for any kernel with a set of directions where it is pointwisely lower bounded and for any $g$.
We claim that in fact we can use the exact same tools to show a similar coercivity estimate on the square root of the kernel:
\begin{lemma}
Let $K : \R^d \times \R^d \to \R^d$ be any non-negative kernel such that there is a set of directions $A(v)$ where $K(v, v') \geq \lambda (1+\abs{v}^{1+2s+\gamma})\abs{v-v'}^{-(d+2s)}$ for all $v'$ such that  $\frac{(v-v')}{\abs{v-v'}} \in A$. Assume that the $(d-1)$-dimensional measure of $A = A(v)$ is at least $\mu(v) := c(1+\abs{v})^{-1}$. Let $R> 0$ and $\mu =\inf{\{\mu(v) : v \in B_R\}}$. Then there is a constant $C = C(\lambda, s, d, \mu)$ such that for any function $g$ supported in $B_R$ we have
\beqs
	\int_{B_R} \int_{B_R} \frac{\abs{g(v) - g(v')}}{\abs{v-v'}^{d/2+s}}\dd v' \dd v \leq C\int_{B_{2R}} \int_{B_{2R}} \abs{g(v) - g(v')}K^{\frac{1}{2}}(v, v') \dd v'\dd v.
\eeqs
\end{lemma}
The lemma follows similarly to \cite[Lemma A.6]{IS}.
\begin{proof}
If we denote with $\Xi(v)$ the symmetric cone of values such that $\frac{(v-v')}{\abs{v-v'}} \in A$, then for any $v' \in \Xi(v)$ we have the lower bound
\beqs
	K^{\frac{1}{2}}(v, v') \geq \lambda^{\frac{1}{2}} (1+\abs{v}^{1+2s+\gamma})^{\frac{1}{2}}\abs{v-v'}^{-\frac{d+2s}{2}}
\eeqs
by positivity of $K$. Moreover, the assumptions on $A$ imply that for all $v_1, v_2 \in B_{2R}$ there is a constant $C_0$ so that 
\beqs
	\Abs{\Xi(v_1) \cap \Xi(v_2) \cap B_{C_0\abs{v_1-v_2}}(v_2)} \geq c\abs{v_1 - v_2}^d,
\eeqs
see \cite[Lemma A.4, A.5]{IS}. We can choose $c_0 >0$ small enough such that 
\beqs
	\Abs{\Xi(v_1) \cap \Xi(v_2) \cap B_{C_0\abs{v_1-v_2}}(v_2)\setminus B_{c_0\abs{v_1-v_2}}(v_1)\setminus B_{c_0\abs{v_1-v_2}}(v_2)} \geq c\abs{v_1 - v_2}^d.
\eeqs
For this choice of $C_0, c_0$ we define
\beqs
	N_r(v) := \int_{B_{C_0r}(v)\setminus B_{c_0r}(v)} \abs{g(v) - g(w)}K^{\frac{1}{2}}(v, w) \dd w.
\eeqs
Then for any $v_1, v_2 \in B_{2R}$ using the triangle inequality we get for $r := \abs{v_1 - v_2}$
\bals
	N_r(v_1) + N_r(v_2) &\geq c\Bigg(\int_{\Xi(v_1) \cap B_{C_0r}(v_1)\setminus B_{c_0r}(v_1)} \abs{g(v_1) - g(w)}\abs{v_1-w}^{-\frac{d+2s}{2}}\dd w \\
	&\qquad+ \int_{\Xi(v_2) \cap B_{C_0r}(v_2)\setminus B_{c_0r}(v_2)} \abs{g(v_2) - g(w)}\abs{v_2-w}^{-\frac{d+2s}{2}}\dd w\Bigg)\\
	&\geq c\int_{\Xi(v_1) \cap \Xi(v_2) \cap B_{C_0r}(v_2)\setminus B_{c_0r}(v_1)\setminus B_{c_0r}(v_2)} \abs{g(v_1) - g(v_2)}r^{-\frac{d+2s}{2}}\dd w\\
	&\geq c  \abs{g(v_1) - g(v_2)}r^{\frac{d-2s}{2}}.
\eals	
Thus by choice of $r$
\bals
	\int_{B_{2R}}\int_{B_{2R}} &\frac{\abs{g(v_1) - g(v_2)}}{\abs{v_1-v_2}^{s+\frac{d}{2}}}\dd v_1 \dd v_2 \\
	&\leq \int_{B_{2R}} \int_{B_{2R}} \big(N_r(v_1) + N_r(v_2)\big) \abs{v_1 - v_2}^{-d} \dd v_2 \dd v_1\\
	&=2 \int_{B_{2R}} \int_{B_{2R}} N_r(v_1)\abs{v_1 - v_2}^{-d} \dd v_2 \dd v_1\\
	&\leq C\int_{B_{2R}} \int_0^{2R} \int_{\mathbb S^{d-1}} N_r(v_1) r^{-1} \dd \sigma \dd r \dd v_1\\
	&\leq C\int_{B_{2R}} \int_{B_{2R}}  \abs{g(v) - g(w)}K^{\frac{1}{2}}(v, w) \int_{C_0^{-1}\abs{v_1-w}}^{c_0^{-1}\abs{v_1-w}} r^{-1} \dd r \dd w  \dd v_1\\
	&\leq C\int_{B_{2R}} \int_{B_{2R}}  \abs{g(v) - g(w)}K^{\frac{1}{2}}(v, w) \dd w  \dd v_1.
\eals
\end{proof}
\bigskip
\noindent {\bf Acknowledgements.} I thank Clément Mouhot for suggesting this problem and for his ingenious insights in our discussions. I also want to thank Jessica Guerand for pointing out the subtleties of the Intermediate Value Lemma
and Michael Struwe for his remarks. Finally, I wish to express my gratitude towards the associate editor, Luis Silvestre, and an anonymous referee for their useful comments.
This work was supported by the Cambridge International \& Newnham College Scholarship from the Cambridge Trust.

\bibliographystyle{plain}
\bibliography{NonLocalDeGiorgi-revised}

\end{document}